\documentclass[12pt]{amsart}
\usepackage{latexsym}
\usepackage{amssymb}
\usepackage{amscd}
\renewcommand\a{\alpha}
\renewcommand\b{\beta}

\newcommand\la{\lambda}
\newcommand\z{\zeta}
\newcommand\e{\eta}
\renewcommand\th{\theta}

\newcommand\io{\iota}

\newcommand\s{\sigma}

\newcommand\f{\phi}
\newcommand\vf{\varphi}

\renewcommand\t{\tau}
\renewcommand\r{\rho}

\newcommand\w{\omega}

\newcommand{\OO}{\mathbb O}

\newcommand\Ql{\bar{\mathbf Q}_l}

\newcommand\BP{\mathbf P}

\newcommand\BZ{\mathbf Z}
\newcommand\BM{\mathbf M}

\newcommand\BH{\mathbf H} 
\newcommand\BI{\mathbf I}
\newcommand\BJ{\mathbf J}

\newcommand\Bp{\mathbf p}
\newcommand\Bm{\mathbf m}

\newcommand\Bv{\mathbf v}
\newcommand\Bk{\mathbf k}

\newcommand\Bw{\mathbf w}

\newcommand\Bla{\boldsymbol\lambda}

\newcommand\Bmu{\boldsymbol\mu}

\newcommand\Bxi{\boldsymbol{\xi}}

\newcommand\CB{\mathcal{B}}
\newcommand\ZC{\mathcal{C}}
\newcommand\CH{\mathcal{H}}
\newcommand\CI{\mathcal{I}}
\newcommand\CE{\mathcal{E}}
\newcommand\CL{\mathcal{L}}

\newcommand\CM{\mathcal{M}}
\newcommand\CN{\mathcal{N}}
\newcommand\CO{\mathcal{O}}
\newcommand\CK{\mathcal{K}}
\newcommand\CP{\mathcal{P}}
\newcommand\CQ{\mathcal{Q}}
\newcommand\CU{\mathcal{U}}

\newcommand\CF{\mathcal{F}}
\newcommand\CG{\mathcal{G}}

\newcommand\CT{ \mathcal{T}}
\newcommand\CZ{ \mathcal{Z}}
\newcommand\CX{ \mathcal{X}}
\newcommand\CY{ \mathcal{Y}}
\newcommand\CW{ \mathcal{W}}

\newcommand\iv{^{-1}}
\newcommand\wh{\widehat}
\newcommand\wt{\widetilde}
\newcommand\wg{^{\wedge}}

\newcommand\ol{\overline}

\newcommand\hra{\hookrightarrow}
\newcommand\lra{\leftrightarrow}

\newcommand\IC{\operatorname{IC}}
\newcommand\Ker{\operatorname{Ker}}
\newcommand\Hom{\operatorname{Hom}}
\newcommand\End{\operatorname{End}}

\newcommand\Ind{\operatorname{Ind}}

\newcommand\supp{\operatorname{supp}\,}

\newcommand\reg{_{\operatorname{reg}}}
\newcommand\rg{\operatorname{reg}}
\newcommand\unip{\operatorname{uni}}
\newcommand\uni{_{\operatorname{uni}}}

\newcommand\id{\operatorname{id}}

\newcommand\lp{\operatorname{\!\langle\!}}
\newcommand\rp{\operatorname{\!\rangle\!}}
\renewcommand\Im{\operatorname{Im}}

\newcommand\nat{^{\natural}}

\newcommand\odd{\operatorname{odd}}

\newcommand{\isom}{\,\raise2pt\hbox{$\underrightarrow{\sim}$}\,}
\numberwithin{equation}{section}

\newtheorem{thm}{Theorem}[section]
\newtheorem{lem}[thm]{Lemma}
\newtheorem{cor}[thm]{Corollary}
\newtheorem{prop}[thm]{Proposition}

\def \para#1{\par\medskip\textbf{#1}
              \addtocounter{thm}{1}}

\def \remark#1{\par\medskip\noindent
                \textbf{Remark #1}
                \addtocounter{thm}{1}}

\begin{document}
\setlength{\baselineskip}{4.9mm}
\setlength{\abovedisplayskip}{4.5mm}
\setlength{\belowdisplayskip}{4.5mm}
%%%
%%%
\renewcommand{\theenumi}{\roman{enumi}}
\renewcommand{\labelenumi}{(\theenumi)}
\renewcommand{\thefootnote}{\fnsymbol{footnote}}
%%%
\renewcommand{\thefootnote}{\fnsymbol{footnote}}
%%%
\allowdisplaybreaks[2]
%\NoBlackBoxes
\parindent=20pt
%%%%%%%%%%%%%%%%%%%%
%%%%%%%%%%%%%%%%%%%%%%%%%%%%%%%%%%%
\medskip
\begin{center}
 {\bf Exotic symmetric spaces of higher level}
\\
{\bf -- Springer correspondence for complex reflection groups --} 
\\
\vspace{1cm}
Toshiaki Shoji
\\
\vspace{0.7cm}
\title{}
\end{center}

\begin{abstract}
Let $G = GL(V)$ for a $2n$-dimensional vector space $V$, and $\th$ an involutive automorphism 
of $G$ such that $H = G^{\th} \simeq Sp(V)$. Let 
$G^{\io\th}\uni$ be the set of unipotent elements $g \in G$ such that $\th(g) = g\iv$. 
For any integer $r \ge 2$, 
we consider the variety $G^{\io\th}\uni \times V^{r-1}$, on which $H$ acts diagonally.     
Let $W_{n,r} = S_n\ltimes (\BZ/r\BZ)^n$ be a complex reflection group.
In this paper, generalizing the known result for $r = 2$,  we show that there exists a natural bijective 
correspondence (Springer correspondence) between the set of irreducible representations 
of $W_{n,r}$ and a certain set of $H$-equivariant simple perverse sheaves on 
$G^{\io\th}\uni \times V^{r-1}$. 
We also consider a similar problem  for $G \times V^{r-1}$, on which $G$ acts diagonally, 
where $G = GL(V)$ for a finite dimensional vector space $V$. 
\end{abstract}

\maketitle
\markboth{SHOJI}{EXOTIC SYMMETRC SPACES}
\pagestyle{myheadings}

\begin{center}
{\sc Introduction}
\end{center}

Let $V$ be a $2n$ dimensional vector space over $\Bk$, where $\Bk$ 
is an algebraically closed field with char $\Bk \ne 2$. Let $G = GL(V)$
and $\th : G \to G$ an involutive automorphism such that 
$G^{\th} \simeq Sp(V)$. Let $\io : G \to G$ be the anti-automorphism 
$g \mapsto g\iv$, and put 
$G^{\io\th} = \{ g \in G \mid \th(g) = g\iv\}$.
We consider the variety $\CX = G^{\io\th} \times V$, on which $H = G^{\th}$ acts
diagonally.  Let $G^{\io\th}\uni$ be the set of unipotent elements in $G^{\io\th}$, 
and define a closed subvariety $\CX\uni$ of $\CX$ by 
$\CX\uni = G^{\io\th}\uni \times V$.  
$\CX\uni$ is nothing but the exotic nilpotent cone introduced by Kato [K]. 
It is known that $\CX\uni$ is $H$-stable, and the set of $H$-orbits
in $\CX\uni$ is in bijection with the set $\CP_{n,2}$ of double partitions of $n$
([K]). 
Let $B$ be a $\th$-stabel Borel subgroup of $G$, $U$ the unipotent radical of $B$, 
and $(M_i)_{1 \le i \le n}$ 
be an isotropic flag in $V$ such that whose stabilizer in $H$ is $B^{\th}$.  
We define a variety $\wt\CX\uni$ by 
\begin{equation*}
\wt\CX\uni = \{ (x, v, gB^{\th}) \in G^{\io\th}\uni \times V \times H/B^{\th}
                   \mid g\iv xg \in U^{\io\th}, g\iv v \in M_n \},
\end{equation*}
and define a map $\pi_1: \wt\CX\uni \to \CX\uni$ by the projection on the first two factors. 
In [K], [SS1], the Springer correspondence between the set of $H$-orbits in $\CX\uni$ 
and the set of irreducible representations of the Weyl group $W_{n,2}$ of type $C_n$ was studied. 
More precisely, it is stated as follows; $(\pi_1)_!\Ql[\dim \CX\uni]$ is a semisimple 
perverse sheaf on $\CX\uni$, equipped with $W_{n,2}$-action and is decomposed as
\begin{equation*}
\tag{1}
(\pi_1)_!\Ql[\dim \CX\uni] \simeq \bigoplus_{\Bla \in \CP_{n,2}} V(\Bla)
                  \otimes \IC(\ol\CO_{\Bla}, \Ql)[\dim \CO_{\Bla}], 
\end{equation*} 
where $V(\Bla)$ is the irreducible representation of $W_{n,2}$ and $\CO_{\Bla}$ 
is the $H$-orbit in $\CX\uni$ corresponding to $\Bla \in \CP_{n,2}$.
\par
In this paper, we consider the variety $G^{\io\th} \times V^{r-1}$ for a positive
integer $r \ge 2$, with the diagonal action of $H$. We call it the exotic symmetric 
space of level $r$.  
Let $W_{n,r} = S_n \ltimes (\BZ/r\BZ)^n$ be the complex reflection group $G(r,1,n)$,
where $S_n$ is the symmetric group of degree $n$.
We will generalize the previous result to the correspondence between the set of 
irreducible representations of $W_{n,r}$ and a certain set of simple perverse sheaves on 
$G^{\io\th}\uni \times V^{r-1}$. 
Let $\CQ_{n,r}$ be the set of $\Bm = (m_1, \dots, m_r) \in \BZ_{\ge 0}^r$ 
such that $\sum m_i = n$, and $\CQ_{n,r}^0$ the subset of $\CQ_{n,r}$ 
consisting of $\Bm$ such that $m_r = 0$. For each $\Bm \in \CQ_{n,r}$, we define varieties
\begin{align*}
\wt\CX_{\Bm, \unip} &= \{ (x, \Bv, gB^{\th}) \in G^{\io\th}\uni \times V^{r-1} \times H/B^{\th}
               \mid g\iv xg \in U^{\io\th}, g\iv \Bv \in \prod_{i=1}^{r-1}M_{p_i} \}, \\
\CX_{\Bm, \unip} &= \bigcup_{g \in H}g(U^{\io\th} \times \prod_{i=1}^{r-1}M_{p_i}),
\end{align*}
where $p_i = m_1 + \cdots + m_i$ for each $i$.
We define a map $\pi^{(\Bm)}_1: \wt\CX_{\Bm \unip} \to \CX_{\Bm, \unip}$ by the projection 
on the first two factors.   
In the case where $\Bm = (n, 0, \dots, 0)$, we write $\wt\CX_{\Bm, \unip}, \CX_{\Bm, \unip}$ 
and $\pi^{(\Bm)}_1$ simply by $\wt\CX\uni, \CX\uni$ and $\pi_1$.  Note that even in this case, 
the map $\wt\CX\uni \to G^{\io\th}\uni \times V^{r-1}$ is not surjective.  
For each $\Bm \in \CQ_{n,r}$ we consider a map 
$\ol\pi_{\Bm,1} : \pi_1\iv(\CX_{\Bm,\unip}) \to \CX_{\Bm,\unip}$.  Note that 
$\wt\CX_{\Bm,\unip} \subset \pi_1\iv(\CX_{\Bm, \unip}) \subset \wt\CX\uni$.
Let $\CP_{n,r}$ be the set of $r$-tuples of partitions 
$\Bla = (\la^{(1)}, \dots, \la^{(r)})$ such that $\sum_i|\la^{(i)}| = n$.
For $\Bm \in \CQ_{n,r}^0$, let 
$\wt\CP(\Bm)$ be the set of all $\Bla = (\la^{(1)}, \dots, \la^{(r)}) \in \CP_{n,r}$
such that $|\la^{(i)}| = m_i$ for $i = 1, \dots, r-2$ (hence $|\la^{(r-1)}| = k$ for 
$0 \le k \le m_{r-1}$). 
As a generalization of (1), we prove the following result, 
which is regarded as the Springer correspondence for 
$W_{n,r}$.  Assume that $\Bm \in \CQ_{n,r}^0$.  Then 
the complex $(\ol\pi_{\Bm,1})_!\Ql[\dim \CX_{\Bm,\unip}]$ is a semisimple perverse sheaf 
on $\CX_{\Bm, \unip}$, equipped with a $W_{n,r}$-action, and is decomposed as 

\begin{equation*}
\tag{2}
(\ol\pi_{\Bm,1})_!\Ql[\dim \CX_{\Bm, \unip}] 
       \simeq \bigoplus_{\Bla \in \wt\CP(\Bm)}V(\Bla) \otimes \IC(\ol X_{\Bla}, \Ql)[\dim X_{\Bla}],
\end{equation*}   
where $V(\Bla)$ is an irreducible representation of $W_{n,r}$, 
and $X_{\Bla}$ is a certain smooth irreducible subvariety of $\CX_{\Bm, \unip}$ parametrized 
by $\Bla \in \wt\CP(\Bm)$. Any irreducible representation of $W_{n,r}$ is realized in this way 
uniquely for a suitable choice of $\Bm \in \CQ^0_{n,r}$. 
We can determine the varieties $X_{\Bla}$ explicitly. Note that in the case $r \ge 3$, 
$G^{\io\th}\uni \times V^{r-1}$ has infinitely many $H$-orbits. So the description of $X_{\Bla}$
becomes more complicated compared to the case where $r = 2$.  
\par
In the course of the proof, we show a weaker version of the Springer correspondence.
For each $\Bm \in \CQ_{n,r}^0$, we define a subgroup $\CW\nat_{\Bm}$ of $W_{n,2}$ by 
\begin{equation*}
\CW\nat_{\Bm} = S_{m_1} \times \cdots \times S_{m_{r-2}} \times W_{m_{r-1},2}.
\end{equation*}
For $\Bla \in \wt\CP(\Bm)$, one can associate an irreducible representation $\r_{\Bla}$ 
of $\CW\nat_{\Bm}$ in a canonical way. We show that the complex 
$(\pi_1^{(\Bm)})_!\Ql[\dim \CX_{\Bm,\unip}]$ is a semisimple perverse sheaf on $\CX_{\Bm \unip}$, 
equipped with $\CW\nat_{\Bm}$-action, and is decomposed as

\begin{equation*}
\tag{3}
(\pi_1^{(\Bm)})_!\Ql[\dim \CX_{\Bm, \unip}]\simeq 
                     \bigoplus_{\Bla \in \wt\CP(\Bm)}\r_{\Bla}\otimes 
              \IC(\ol X_{\Bla}, \Ql)[\dim X_{\Bla}].
\end{equation*} 
Any irreducible representation of $\CW\nat_{\Bm}$ is realized in this way uniquely 
for a suitable choice of $\Bm \in \CQ_{n,r}^0$. 
Note that the group $W_{n,r}$ is not directly relatd to the geometry of 
$H/B^{\th}$, while $\CW\nat_{\Bm}$ behaves well since it is a subgroup of $W_{n,2}$.
So first we show (3), and then prove (2) by making use of (3). 
\par
We also consider the variety $\CX = G \times V$, where $V$ is an $n$-dimensional 
vector space over $\Bk$ (of any characteristic), 
and $G = GL(V)$.  $G$ acts diagonally on $\CX$.  Put 
$\CX\uni = G\uni \times V$, where $G\uni$ is the set of unipotent elements in $G$.
The variety $\CX\uni$ is called the enhanced nilpotent cone by Achar-Henderson [AH]. 
It is known by [AH], [T] that $\CX\uni$ is $G$-stable, and  
the set of $G$-orbits is in bijection with $\CP_{n,2}$.
For each $\Bm = (m_1, m_2) \in \CQ_{n,2}$, one can define a similar 
map $\pi_1^{(\Bm)} : \wt\CX_{\Bm,\unip} \to \CX_{\Bm, \unip}$ as in the exotic case. 
Achar-Henderson [AH] and Finkelberg-Ginzburg-Travkin [FGT]
proved the Springer correspondence between the set of irreducible representaions of 
$S_{m_1} \times S_{m_2}$ and the set of simple perverse sheaves associated to the  
$G$-orbits in $\CX_{\Bm, \unip}$, which are direct summands of 
$(\pi_1^{(\Bm)})_!\Ql[\dim \CX_{\Bm, \unip}]$.
In this paper, we consider the variety $\CX = G \times V^{r-1}$ with the diagonal $G$-action, 
which we call the enhanced space of level $r$.
The arguments used to prove the Springer correspondence (3) in the exotic case can be applied 
also to the enhanced case, step by step, by a suitable modification. 
Actually the argument becomes drastically simple. 
We show that the Springer correspondence holds for 
$\CW\nat_{\Bm} = S_{m_1} \times \cdots \times S_{m_r}$ for any $\Bm \in \CQ_{n,r}$.   
(This result was announced by the author in 2009, but was not published.)
In [Li], Li established the Springer correspondence for such $\CW_{\Bm}\nat$ 
in connection with certain perverse sheaves arising from 
the framed Jordan quiver.  Considering the framed Jordan quiver is essentially the same 
as considering the enhanced variety. So in this case our result is regarded as 
an alternate approach for his result.

\par\bigskip\bigskip

\begin{center}
{\bf Contents}
\end{center}
\par\medskip

1. \ Intersection cohomology on $G^{\io\th}\reg \times V^{r-1}$ (exotic case)

2. \ Intersection cohomology on $G^{\io\th} \times V^{r-1}$ (exotic case)

3. \ A variant of Theorem 2.2

4. \ Intersection cohomology on $G^{\io\th} \times V^{r-1}$ (enhanced case)

5. \ Unipotent variety of enhanced type 

6. \ Unipotent variety of exotic type

7. \ Springer correspondence

8.  \ Determiniation of the Springer correspsondence

\par\bigskip 
\section{Intersection cohomology on $G\reg^{\io\th} \times V^{r-1}$ (exotic case)}

\para{1.1.}
Let $\Bk$ be an algebraically closed field.
In this paper, we consider the following two cases.
\par
\medskip
(I) \ {\bf The exotic case.} 
\par
Let $V$ be a $2n$ dimensional vector space over $\Bk$ (with char $\Bk \neq 2$), 
with basis $\{ e_1, \dots e_n, f_1, \dots f_n\}$.
Let $G =  GL_{2n}$.  
Consider an involutive automorphism  $\th : G \to G$ given by
\begin{equation*}
\th(g) = J\iv({}^tg\iv)J \quad\text{ with }
J = \begin{pmatrix}
          0 & 1_n \\
            -1_n & 0
         \end{pmatrix},
 \end{equation*}
where $1_n$ is the identity matrix of degree $n$,
and put $H = G^{\th}$.  Then $H$ is the symplectic group $Sp_{2n}$
with respect to the symplectic form 
$\lp u, v\rp ={}^tuJv$ for $u,v \in V$ 
under the identification $V \simeq \Bk^{2n}$ via the basis 
$\{ e_1, \dots, e_n, f_1, \dots f_n\}$, which  gives rise to 
a symplectic basis.  
\par
Let $\io : G \to G$ be the anti-automorphism $g \to g\iv$.  We consider the set
$G^{\io\th} = \{ g \in G \mid \th(g) = g\iv\}$. It is known that 
$G^{\io\th} = \{ g\th(g)\iv \mid g \in G\}$, and so $G^{\io\th} \simeq G/H$. 
Let $T \subset B$ be a pair of a $\th$-stable maximal torus and a $\th$-stable
Borel subgroup of $G$.  Let $M_1 \subset \cdots \subset M_n$ be an isotropic flag 
in $V$ such that its stabilizer in $H$ coincides with $B^{\th}$.  We assume that 
$M_i = \lp e_1, \dots, e_i\rp$ for $i = 1, \dots n$, and that $e_i, f_j$ are 
weight vectors for $T$. 

\medskip
(II) \  {\bf The enhanced case.}
\par
Let $\wt V = V \oplus V$, where $V$ is an $n$-dimensional vector space over $\Bk$, 
and 
$G = G_0 \times G_0$ a subgroup of $GL(\wt V)$ with 
$G_0 = GL(V)$. Let $\th : G \to G$ be an 
involution defined by $\th(g_1, g_2) = (g_2, g_1)$.  
Put $H = G^{\th} \simeq G_0$.  Then $H$ acts naturally on $V$.  
Let $G^{\io\th}$ be a subset of $G$ defined similarly to the case (I).  Then 
$G^{\io\th} \simeq G_0$ (as a set) and $H \simeq G_0$ acts on $G^{\io\th}$ by 
conjugation.  
Let $T \subset B$ be a pair of a $\th$-stable maximal torus and a $\th$-stable
Borel subgroup of $G$.  We can write $T = T_0 \times T_0$ and 
$B = B_0 \times B_0$ so that $B^{\th} \simeq B_0, T^{\th} \simeq T_0$.
Let $M_1 \subset \cdots \subset M_n = V$ be a complete flag 
in $V$ such that its stabilizer in $H$ coincides with $B^{\th}$.  We fix a basis 
$\{ e_1, \dots, e_n \}$ of $V$ such that 
$M_i = \lp e_1, \dots, e_i\rp$ for $i = 1, \dots n$, and that $e_i$ are 
weight vectors for $T^{\th}$.  Let $\{ e_1, \dots, e_n, f_1, \dots, f_n \}$ be 
a basis of $\wt V$, where $f_i = e_i$ for each $i$, and define a symplectic form
$\lp \ , \ \rp$ on $\wt V$ so that $\{ e_i, f_j\}$ gives a symplectic basis of $\wt V$.  

\para{1.2.}
For an integer $r \ge 1$, 
we consider the variety $G^{\io\th} \times V^{r-1}$ on which $H$ acts diagonally.
We call $G^{\io\th} \times V^{r-1}$ the exotic symmetric space of level $r$ in 
the case (I), and the enhanced space of level $r$ in the case (II).
Let $Q_{n,r} = \{ \Bm = (m_1, \dots, m_r) \in \BZ_{\ge 0}^r \mid \sum_im_i = n\}$.
We define $\CQ_{n,r}^0 = \{ \Bm \in \CQ_{n,r} \mid m_r = 0\}$. 
For each $\Bm \in \CQ_{n,r}$, we define $\Bp(\Bm) = (p_1, p_2, \dots, p_r)$ by 
$p_i = m_1 + \cdots + m_i$ for each $i$. 
We define varieties 
\begin{align*}
\wt\CX_{\Bm} &= \{ (x, \Bv, gB^{\th}) \in G^{\io\th} \times V^{r-1} \times H/B^{\th}
                 \mid g\iv xg \in B^{\io\th}, g\iv \Bv \in \prod_{i=1}^{r-1}M_{p_i} \}, \\ 
\CX_{\Bm}  &= \bigcup_{g \in H}g(B^{\io\th} \times \prod_{i=1}^{r-1} M_{p_i}). 
\end{align*}
We define a map $\pi^{(\Bm)}: \wt\CX_{\Bm} \to G^{\io\th} \times V^{r-1}$ by 
$\pi^{(\Bm)}(x,\Bv,gB^{\th}) = (x,\Bv)$. 
Clearly $\CX_{\Bm} = \Im \pi^{(\Bm)}$. Since 
$\wt\CX_{\Bm} \simeq  H \times^{B^{\th}}(B^{\io\th} \times \prod_i M_{p_i})$, 
$\wt\CX_{\Bm}$ is smooth and irreducible. Since $\pi^{(\Bm)}$ is proper, 
$\CX_{\Bm}$ is a closed irreducible subvariety of $G^{\io\th} \times V^{r-1}$.
In the case where $\Bm = (n, 0, \dots, 0)$, namely, $\Bp(\Bm) = (n, \dots, n)$, 
we write $\wt\CX_{\Bm}, \CX_{\Bm}$
and $\pi^{(\Bm)}$ by $\wt\CX, \CX$ and $\pi$, respectively. 
Note that for any $\Bm \in \CQ_{n,r}$, 
$\CX_{\Bm}$ is contained in $\CX$. 
The dimension of $\wt\CX_{\Bm}$ is computed as follows;
\begin{equation*}
\tag{1.2.1}
\dim \wt\CX_{\Bm} = \begin{cases}
                       2n^2 + \sum_{i=1}^r(r-i)m_i &\quad\text{ exotic case,} \\
                        n^2 + \sum_{i=1}^r(r-i)m_i &\quad\text{ enhamced case. }
                     \end{cases}
\end{equation*}
In fact, in the exotic case, by [SS1; (3.1.1)], we have
\begin{align*}
\dim \wt\CX_{\Bm} &= \dim H/B^{\th} + \dim B^{\io\th} + \sum_{i=1}^{r-1} \dim M_{p_i} \\
                  &= 2n^2 + \sum_{i=1}^{r-1}(m_1 +\cdots + m_i) \\
                  &= 2n^2 + \sum_{i=1}^r(r-i)m_i.
\end{align*}
The computation for the enhanced case is similar (in this case,
$\dim B^{\th} = \dim B^{\io\th}$). 
\par
Let $T^{\io\th}\reg$ be the set of regular semisimple elements in $T^{\io\th}$, namely, 
the set of elements in $T^{\io\th}$ such that all the eigenspaces in $V$ have 
dimension 2 (resp. dimension 1) in the exotic case (resp. in the enhanced case). 
We put $G^{\io\th}\reg = \bigcup_{g \in H}gT^{\io\th}\reg g\iv$, 
$B^{\io\th}\reg = G^{\io\th}\reg \cap B$. 
We define  varieties $\wt\CY_{\Bm}, \CY_{\Bm}$ by
\begin{align*}
\wt\CY_{\Bm} &= \{ (x,\Bv,gB^{\th}) \in G^{\io\th}\reg \times V^{r-1}\times H/B^{\th}
                     \mid g\iv xg \in B^{\io\th}\reg, g\iv \Bv \in \prod_{i=1}^{r-1} M_{p_i}\} \\
\CY_{\Bm} &= \bigcup_{g \in H}g(B^{\io\th}\reg \times \prod_{i=1}^{r-1}M_{p_i}), 
\end{align*}
and a map $\psi^{(\Bm)}: \wt\CY_{\Bm} \to G^{\io\th} \times V^{r-1}$ by 
$\psi^{(\Bm)}(x,\Bv, gB^{\th}) = (x,\Bv)$. 
Clearly $\Im \psi^{(\Bm)} = \CY_{\Bm}$.  As in the case of $\wt\CX_{\Bm}$, 
 we write $\wt\CY_{\Bm},\CY_{\Bm}$ and 
$\psi^{(\Bm)}$ by $\wt\CY,\CY, \psi$ 
in the case where $\Bm = (n,0,\dots, 0)$,. As in [SS1; (3.1.2)], $\wt\CY_{\Bm}$ can be 
expressed as 
\begin{align*}
\tag{1.2.2}
\wt\CY_{\Bm} &\simeq H \times^{B^{\th}}(B^{\io\th}\reg \times \prod_iM_{p_i}) \\
             &\simeq H \times^{B^{\th}\cap Z_H(T^{\io\th})}(T^{\io\th}\reg 
                 \times \prod_iM_{p_i}). 
\end{align*}

\para{1.3.}
In the remainder of this section, we assume that $\CX$ is of exotic type. 
We define a partial order on $\CQ_{n,r}$ by $\Bm \le \Bm'$ if 
$p_i \le p_i'$ for each $i$, where 
$\Bp(\Bm) = (p_1, \dots, p_r)$ and $\Bp(\Bm') = (p_1', \dots, p_r')$,
respectively. 
Then we have $\CY_{\Bm} \subseteq \CY_{\Bm'}$ if and only if $\Bm \le \Bm'$. 
Assume that $\Bm \le \Bm'$, then $\CX_{\Bm}$ is a closed subset of $\CX_{\Bm'}$. 
Since $\CY_{\Bm} = \CY \cap \CX_{\Bm}$, $\CY_{\Bm}$ is a closed subset of $\CY_{\Bm'}$. 
We define a subset $\CY_{\Bm}^0$ of $\CY_{\Bm}$ by 
$\CY_{\Bm}^0 = \CY_{\Bm} \backslash \bigcup_{\Bm'<\Bm}\CY_{\Bm'}$.  
Thus $\CY_{\Bm}^0$ is an open dense subset of $\CY_{\Bm}$, and we have a
partition $\CY = \coprod_{\Bm \in \CQ_{n,r}}\CY_{\Bm}^0$. 
\par 
As in [SS1; 3.2], for each subset $I \subset [1,n]$, put
$M_I = \{ v \in M_n \mid \supp(v) = I \}$, where for $v = \sum_{i=1}^na_ie_i \in M_n$,  
$\supp(v)$ is the set of $j \in [1,n]$ such that $a_j \ne 0$.
$M_I$ is an open dense subset of the space spanned by $\{ e_i \mid i \in I\}$,
which we denote by $\ol M_I$.     
For each $\Bm \in \CQ_{n,r}$, we define $\CI(\Bm)$ as the set of 
$\BI = (I_1, \dots, I_r)$ such that $[1,n] = \coprod_{i=1}^rI_i$ with 
$|I_i| = m_i$. For $\BI = (I_1, \dots, I_r)$, put 
$I_{< i} = I_1\cup I_2 \cup\cdots\cup I_{i-1}$.  Hence $|I_{< i}| = p_{i-1}$. 
For each $\BI \in \CI(\Bm)$, we define a set $\BM_{\BI} \subset (M_n)^{r-1}$
by
\begin{equation*}
\BM_{\BI} = \{ \Bv = (v_1, \dots v_{r-1}) \in (M_n)^{r-1} \mid  
      v_i \in M_{I_i} + \ol M_{I_{< i}} \}. 
\end{equation*}
and define a variety $\wt\CY_{\BI}$ by
\begin{equation*}
\wt\CY_{\BI} = H \times^{B^{\th} \cap Z_H(T^{\io\th})}(T^{\io\th}\reg \times \BM_{\BI}).
\end{equation*}

\par\noindent
Note that $Z_H(T^{\io\th}) \simeq SL_2 \times \cdots \times SL_2$ ($n$-times) and 
$B^{\th} \cap Z_H(T^{\io\th})$ can be identified with the subgroup 
$B_2 \times \cdots \times B_2$, where $B_2$ is the Borel subgroup of $SL_2$.
Since the action of $B^{\th} \cap Z_H(T^{\io\th})$ on $M_n$ is given by the 
action of the torus part $T^{\th}$, $\wt\CY_{\BI}$ is well-defined.
Let $\psi_{\BI}: \wt\CY_{\BI} \to \CY$ be the map induced from the map
given by $(g, (t,\Bv)) \mapsto (gtg\iv, g\Bv), 
    H \times (T\reg^{\io\th} \times \BM_{\BI}) \to \CY$.
Then $\Im \psi_{\BI}$ coincides with $\CY_{\Bm}^0$ (independent of $\BI$). 
For $\BI \in \CI(\Bm)$, we define a parabolic subgroup $Z_H(T^{\io\th})_{\BI}$
of $Z_H(T^{\io\th})$ by the condition that the $i$-th factor is $SL_2$ if 
$i \in I_r$ and is $B_2$ otherwise.  Since $Z_H(T^{\io\th})_{\BI}$ stabilizes 
$\BM_{\BI}$, one can define

\begin{equation*}
\wh \CY_{\BI} = H \times^{Z_H(T^{\io\th})_{\BI}}(T^{\io\th}\reg \times M_{\BI}).
\end{equation*}  
Then the map $\psi_{\BI}$ factors through $\wh\CY_{\BI}$,

\begin{equation*}
\tag{1.3.1}
\begin{CD}
\psi_{\BI} : \wt\CY_{\BI} @>\xi_{\BI} >> \wh\CY_{\BI} @>\e_{\BI}>> \CY_{\Bm}^0,
\end{CD}
\end{equation*}
where $\xi_{\BI}$ is the natural projection and $\e_{\BI}$ 
is the map induced from the map $(g,(t,\Bv)) \mapsto (gtg\iv, g\Bv)$. 
Then $\xi_{\BI}$ is a locally trivial fibration with fibre isomorphic to
\begin{equation*}
Z_H(T^{\io\th})_{\BI}/(B^{\th}\cap Z_H(T^{\io\th})) \simeq (SL_2/B_2)^{I_r}
   \simeq \BP_1^{I_r},
\end{equation*}
where $(SL_2/B_2)^{I_r}$ denotes the direct product of $SL_2/B_2$ with respect to the
factors corrsponding to $I_r$, and similarly for $\BP_1^{I_r}$.  Thus 
$\BP_1^{I_r} = (\BP_1)^{m_r}$. 
\par
Let $S_{\BI}\simeq S_{I_1} \times\cdots\times S_{I_r}$ be the stabilizer 
of $(I_1, \dots, I_r)$ in $S_n$. 
Let $\CW = N_H(T^{\io\th})/Z_H(T^{\io\th}) \simeq S_n$,  and $\CW_{\BI}$ the 
subgroup of $\CW$ corresponding to the subgroup $S_{\BI}$.  
Then $\CW$ acts on $Z_H(T^{\io\th}) \simeq SL_2 \times\cdots\times SL_2$ as 
the permutation of factors, and $\CW_{\BI}$ stabilizes the group $Z_H(T^{\io\th})_{\BI}$.  
Since $\CW_{\BI}$ stabilizes $M_{\BI}$, $\CW_{\BI}$ acts on $\wt\CY_{\BI}$ and on 
$\wh\CY_{\BI}$. 
Now the map $\e_{\BI}: \wh\CY_{\BI} \to \CY_{\Bm}^0$ turns out to be a finite Galois 
covering with group $\CW_{\BI}$. 
\par
We define $\BI(\Bm) = (I^{\circ}_1, \dots, I^{\circ}_r) \in \CI(\Bm)$ by 
$I^{\circ}_i = [p_{i-1}+1, p_i]$ for $i = 1, \dots, r$.
For $\BI = \BI(\Bm)$, put $\wt\CY_{\BI} = \wt\CY_{\Bm}^0$ and $\CW_{\BI} = \CW_{\Bm}$.  
Note that $\wt\CY_{\Bm}^0$ 
is an open dense subset of $\wt\CY_{\Bm}$, hnce irreducible. 
Put $\psi\iv(\CY_{\Bm}^0) = \wt\CY_{\Bm}^+$.  
$\CW$ acts naturally on $\wt\CY$ and the map $\psi$ is $\CW$-equivariant with respect to the
trivial action of $\CW$ on $\CY$. Hence it preserves the subset $\wt\CY_{\Bm}^+$, and 
the stabilizer of $\wt\CY_{\Bm}^0$ in $\CW$ coincides with $\CW_{\Bm}$. 
One can check that 
\begin{equation*}
\tag{1.3.2}
\wt\CY_{\Bm}^+ = \coprod_{\BI \in \CI(\Bm)}\wt\CY_{\BI} 
                = \coprod_{w \in \CW/\CW_{\Bm}}w(\wt\CY_{\Bm}^0),
\end{equation*}
where $\wt\CY_{\BI}$ is an irreducible component of $\wt\CY_{\Bm}^+$. 
\par
We have the following lemma (cf. [SS1, Lemma 3.3]).

\begin{lem} %%%  Lemma 1.4
Assume that $r \ge 2$. 
\begin{enumerate}
\item
$\CY_{\Bm}$ is open dense in $\CX_{\Bm}$, and $\wt\CY_{\Bm}$ is open dense 
in $\wt\CX_{\Bm}$. 
\item
$\dim \wt\CX_m = \dim \wt\CY_{\Bm} = 2n^2 + \sum_{i=1}^r(r-i)m_i$.
\item
$\dim \CX_{\Bm} = \dim \CY_{\Bm} = 2n^2 + \sum_{i=1}^r(r-i)m_i - m_r$. 
\item
$\CY = \coprod_{\Bm \in \CQ_{n,r}}\CY_{\Bm}^0$ gives a stratification of $\CY$ 
by smooth strata $\CY_{\Bm}^0$, and the map $\psi: \wt\CY \to \CY$ is semismall
with respect to this stratification.  
\end{enumerate}
\end{lem}

\begin{proof}
Since $B^{\io\th}\reg \times \prod_i M_{p_i}$ is open dense in 
$B^{\io\th} \times \prod_iM_{p_i}$, $\wt\CY_{\Bm}$ is open dense in $\wt\CX_{\Bm}$.
Since $\pi^{(\Bm)}$ is a closed map and $(\pi^{(\Bm)})\iv(\CY_{\Bm}) = \wt\CY_{\Bm}$, 
$\CY_{\Bm}$ is open dense in $\CX_{\Bm}$. So (i) holds. (ii) follows from (1.2.1).
By using 
the decomposition $\psi_{\BI} = \e_{\BI}\circ \xi_{\BI}$ for $\BI = \BI(\Bm)$, we see that 
$\dim \wt\CY_{\Bm} = \dim \CY_{\Bm} + m_r$. Hence (iii) holds.  
By (1.3.1) and (1.3.2), $\dim \psi\iv(x,\Bv) = m_r$ for $(x,\Bv) \in \CY_{\Bm}^0$. 
Since
\begin{align*}
\dim \CY - \dim \CY^0_{\Bm} &= (2n^2 + (r-1)n) - (2n^2 + \sum_{i=1}^r(r-i)m_i - m_r) \\
              &= \sum_{i=1}^r(r-1)m_i - \sum_{i=1}^r(r-i)m_i + m_r  \\
              &= \sum_{i=1}^r(i-1)m_i + m_r  \\
              &\ge 2m_r.
\end{align*}
Hence $\dim \psi\iv(x,\Bv) \le (\dim\CY - \dim \CY^0_{\Bm})/2$ for $(x,\Bv) \in \CY_{\Bm}^0$,
and so $\psi$ is semismall.  
\end{proof}

\para{1.5.}
Let $\psi_{\Bm} : \wt\CY_{\Bm}^+ \to \CY_{\Bm}^0$ be the restriction of $\psi$ on 
$\wt\CY_{\Bm}^+$.  Then $\psi_{\Bm}$ is $\CW$-equivariant with respect to the natural 
action of $\CW$ on $\wt\CY_{\Bm}^+$ and the trivial action on $\CY_{\Bm}^0$. 
We consider the diagram 
\begin{equation*}
\tag{1.5.1}
\begin{CD}
T^{\io\th} @<\a_0<<  \wt\CY @ >\psi >> \CY,
\end{CD}
\end{equation*} 
where $\a_0$ is given by $\a_0(x,\Bv, gB^{\th}) = p_T(g\iv xg)$ 
($p_T : B^{\io\th} \to T^{\io\th}$ is the natural projection).  
Let $\CE$ be a tame local system on $T^{\io\th}$, and we consider 
the complex $\psi_*\a_0^*\CE$ on $\CY$. For each $\BI \in \CI(\Bm)$, 
we have a similar diagram as (1.5.1) by replacing $\wt\CY, \CY, \psi, \a_0$ by 
$\wt\CY_{\BI}, \CY_{\Bm}^0, \psi_{\BI}, \a_{\BI}$.
By (1.3.2), we have

\begin{equation*}
\tag{1.5.2}
(\psi_{\Bm})_!\a_0^*\CE|_{\wt\CY_{\Bm}^+} \simeq \bigoplus_{\BI \in \CI(\Bm)}
                         (\psi_{\BI})_!\a_0^*\CE|_{\wt\CY_{\BI}}.
\end{equation*} 

Let $\CE_{\BI} = \xi_{\BI}^*\a_0^*\CE|_{\wt\CY_{\BI}} = \xi_{\BI}^*\a_{\BI}^*\CE$ 
be a local system on $\wh\CY_{\BI}$, and 
$\CW_{\CE_{\BI}}$ the stabilizer of $\CE_{\BI}$ in $\CW_{\BI}$. 
In the case where $\BI = \BI(\Bm)$, we put $\CW_{\CE_{\BI}} = \CW_{\Bm,\CE}$.
In the case where $\Bm = (n,0,\dots, 0)$, we put $\CW_{\Bm,\CE} = \CW_{\CE}$, 
which is the stabilizer of $\CE$ in $\CW$.  
$\CW_{\CE}$ acts on $(\psi_{\Bm})_!\a_0^*\CE|_{\wt\CY_{\Bm}^+}$ as automorphisms of 
complexes, and permutes each direct summand $(\psi_{\BI})_!\a_{\BI}^*\CE$ according 
to the permutaion action of $S_n$ on $\CI(\Bm)$. 
Since $\e_{\BI}$ is a finite Galois covering with group $\CW_{\BI}$, 
$(\e_{\BI})_!\CE_{\BI}$ is a semisimple local system.  As in [SS1, 3.4]
the endomorphism algebra $\End((\e_{\BI})_!\CE_{\BI})$ is canonically 
isomorphic to the group algebra $\Ql[\CW_{\CE_{\BI}}]$, and 
$(\e_{\BI})_!\CE_{\BI}$ is decomposed as 
\begin{equation*}
\tag{1.5.3}
(\e_{\BI})_!\CE_{\BI} \simeq \bigoplus_{\r \in \CW_{\CE_{\BI}}\wg}
        \r \otimes \CL_{\r},
\end{equation*} 
where $\CL_{\r} = \Hom (\r, (\e_{\BI})_!\CE_{\BI})$ is a simple local system on 
$\CY_{\Bm}^0$. 

\para{1.6.}
Since $\psi_{\Bm}$ is proper and $\wt\CY_{\BI}$ is closed in $\wt \CY_{\Bm}^+$, 
$\psi_{\BI}$ is proper. Hence $\xi_{\BI}$ is also proper.  
We note that 
\par\medskip\noindent
(1.6.1) \ $R^i(\xi_{\BI})_!\Ql$ is a constant sheaf for each $i$.
\par\medskip\noindent
In fact, we have a commutative diagram 
\begin{equation*}
\begin{CD}
\wt\CY_{\BI} @>\xi_{\BI}>>  \wh\CY_{\BI}  \\
  @VVV        @VVV       \\
H/B^{\th}\cap Z_H(T^{\io\th})   @>\xi'_{\BI}>>  H/Z_H(T^{\io\th})_{\BI},
\end{CD}
\end{equation*}
where vertical maps are natural projections (see 1.3), and the map 
$\xi'_{\BI}$ is the map induced from the inclusion 
$B^{\th} \cap Z_H(T^{\io\th}) \hra Z_H(T^{\io\th})_{\BI}$.   
Since this diagram is cartesian, (1.6.1) is equivalent to the statement 
that 
\par\medskip\noindent
(1.6.2) \ $R^i(\xi'_{\BI})_!\Ql$ is a constant sheaf for each $i$.
\par\medskip\noindent
We show (1.6.2).  Since $\xi'_{\BI}$ is a locally trivial fibration, 
$R^i(\xi'_{\BI})_!\Ql$ is a locally constant sheaf on $H/Z_H(T^{\io\th})_{\BI}$.
Since $\xi'_{\BI}$ is $H$-equivariant, $R^i(\xi'_{\BI})_!\Ql$ is an $H$-equivariant
local system  on $H/Z_H(T^{\io\th})_{\BI}$, hence it is a constant sheaf, as 
$Z_H(T^{\io\th})_{\BI}$ is connected. 
Thus (1.6.2), and so (1.6.1) holds.
\par
Since 
$\xi_{\BI}$ is a $\BP_1^{I_r}$-bundle, we see that 
\begin{equation*}
\tag{1.6.3}
(\xi_{\BI})_!\a_{\BI}^*\CE \simeq H^{\bullet}(\BP_1^{I_r})\otimes \CE_{\BI},
\end{equation*}  
where $H^{\bullet}(\BP_1^{I_r})$ denotes 
$\bigoplus_{i \ge 0}H^{2i}(\BP_1^{I_r},\Ql)$, which we regard as a complex of
vector spaces $(K_i)$ with $K_{\odd} = 0$. 
It follows that
\begin{equation*}
\tag{1.6.4}
(\psi_{\BI})_!\a_{\BI}^*\CE \simeq (\e_{\BI})_!(\xi_{\BI})_!\a_{\BI}^*\CE 
       \simeq H^{\bullet}(\BP_1^{I_r})\otimes(\e_{\BI})_!\CE_{\BI}.
\end{equation*}
\par
Let $W_{n,r} = S_n\ltimes (\BZ/r\BZ)^n$ be the complex reflection group 
$G(r,1,n)$. We put $\wt\CW = \CW\ltimes (\BZ/r\BZ)^n$.
We define a subgroup $\wt\CW_{\CE}$ (resp. $\wt\CW_{\Bm, \CE}$) of $\wt\CW$ by
$\wt\CW_{\CE} = \CW_{\CE}\ltimes (\BZ/r\BZ)^n$ (resp. 
$\wt\CW_{\Bm,\CE} = \CW_{\Bm,\CE}\ltimes (\BZ/r\BZ)^n$).  
Let $\z$ be a primitive $r$-th root of unity in $\Ql$, and define a linear 
character $\t_i : \BZ/r\BZ \to \Ql^*$ by $\t_i(a) = \z^{i-1}$ for $i = 1, \dots, r$, 
where $a$ is a generator of $\BZ/r\BZ$. 
Let $\r$ be an irreducible representation of $\CW_{\Bm,\CE}$. 
Since $\CW_{\Bm, \CE}$ is decomposed as 
$\CW_{\Bm,\CE} = \CW_1 \times \cdots \times \CW_r$
with subgroups $\CW_i \subset S_{m_i}$, $\r$ can be written as
$\r = \r_1\boxtimes\cdots\boxtimes \r_r$ with $\r_i \in \CW_i\wg$.
We define an irreducible $\wt\CW_{\Bm,\CE}$-module 
$\wt\r$ (resp. $\wt\r'$) by defining the action of $(\BZ/r\BZ)^{m_i}$ on $\r_i$ via 
$\t_i^{\otimes m_i}$ for $i = 1, \dots, r$ 
(resp. via $\t_i^{\otimes m_i}$ for $i = 1, \dots, r-1$, and 
via the trivial action for $i = r$).  
Put 
$\wt V_{\r} = \Ind_{\wt\CW_{\Bm,\CE}}^{\wt\CW_{\CE}}\wt\r$.  Then 
$\wt V_{\r}$ is an irreducible $\wt\CW_{\CE}$-module.
\par
We regard $H^{\bullet}(\BP_1^{m_r})\otimes\r$ as a complex of 
$\CW_{\Bm,\CE}$-modules.
We define an action of $\BZ/r\BZ$ on 
$H^{\bullet}(\BP_1) = H^2(\BP_1) \oplus H^0(\BP_1)\simeq \Ql\oplus\Ql$ by
$\t_r \oplus \t_1$, and define an action of $(\BZ/r\BZ)^{m_r}$ on 
$H^{\bullet}(\BP_1^{m_r}) \simeq H^{\bullet}(\BP_1)^{\otimes m_r}$ 
by $(\t_r \oplus \t_1)\boxtimes \cdots \boxtimes (\t_r \oplus \t_1)$ 
$(m_r$-factors). 
Thus we can consider an extension $H^{\bullet}(\BP_1^{m_r})\otimes \wt\r'$
of $H^{\bullet}(\BP_1^{m_r})\otimes \r$, as a complex of 
$\wt\CW_{\Bm,\CE}$-modules.
In view of (1.5.2), (1.5.3) and (1.6.4), we have
\begin{equation*}
\tag{1.6.5}
(\psi_{\Bm})_!\a_0^*\CE|_{\wt\CY^+_{\Bm}} \simeq \bigoplus_{\r \in \CW_{\Bm,\CE}\wg}
   \Ind_{\wt\CW_{\Bm,\CE}}^{\wt\CW_{\CE}}\bigl(H^{\bullet}(\BP_1^{m_r})\otimes \wt\r'\bigr)
            \otimes \CL_{\r}.
\end{equation*}  
Note that by our construction, (1.6.5) can be rewritten as 
\begin{equation*}
\tag{1.6.6}
(\psi_{\Bm})_!\a_0^*\CE|_{\wt\CY^+_{\Bm}} \simeq
\biggl(\bigoplus_{\r \in \CW_{\Bm,\CE}\wg}\wt V_{\r}\otimes \CL_{\r}\biggr)[-2m_r] + \CN_{\Bm}, 
\end{equation*}
where $\CN_{\Bm}$ is a sum of various $\CL_{\r}[-2i]$ for $\r \in \CW_{\Bm,\CE}\wg$ with 
$0 \le i < m_r$. 

\par
For each $\Bm \in \CQ_{n,r}$, let $\ol\psi_{\Bm}$ be the restriction of 
$\psi$ on $\psi\iv(\CY_{\Bm})$.  In what follows, we denote 
$\a_0^*\CE|_{\psi\iv(\CY_{\Bm})}$ by $\a_0^*\CE$ for short. 
Put $d_{\Bm} = \dim \CY_{\Bm}$. 
For $1 \le j < r-1$, $0 \le k \le m_j$, 
we define a subset $\CQ(\Bm; j,k)$ of $\CQ_{\Bm}$  by 
\begin{equation*}
\CQ(\Bm; j,k) = \{ \Bm' \in \CQ_{n,r} \mid \Bm'  \le \Bm,
     p_i = p_i' \ (1 \le i \le  j-1), p_j' = p_{j-1} + k \},
\end{equation*}
where $\Bp(\Bm) = (p_1, \dots, p_r)$ and $\Bp(\Bm') = (p_1', \dots, p_r')$.
We also define $\Bm(j,k) \in \CQ_{n,r}$ by 
$\Bm(j,k) = (m_1', \dots, m_r')$ such that $m'_j = k, 
m'_{j+1} = m_{j+1} + (m_j -k)$ and $m_i' = m_i$ 
for $i \ne j, j+1$.  
In the case where $j = r-1$, we write $\Bm(j,k)$ simply as $\Bm(k)$.
Hence if $\Bm \in \CQ_{n,r}^0$, 
$\Bm(k) = (m_1, \cdots, m_{r-2}, k, k')$ with $k + k' = m_{r-1}$.  
We have the following proposition.

\begin{prop}  %%%%  Prop. 1.7
For each $\Bm \in \CQ^0_{n,r}$, $(\ol\psi_{\Bm})_!\a_0^*\CE[d_{\Bm}]$
is a semisimple perverse sheaf on $\CY_{\Bm}$, equipped with $\wt\CW_{\CE}$-action, 
and is decomposed as

\begin{equation*}
(\ol\psi_{\Bm})_!\a_0^*\CE[d_{\Bm}] \simeq \bigoplus_{0 \le k \le m_{r-1}}
      \bigoplus_{\r \in \CW_{\Bm(k),\CE}^{\wg}}\wt V_{\r}\otimes 
                \IC(\CY_{\Bm(k)}, \CL_{\r})[d_{\Bm(k)}].
\end{equation*} 
\end{prop} 

\begin{proof}
$\ol\psi_{\Bm}$ is proper, and a similar argument as in the proof 
of Lemma 1.4 (iv) shows that $\ol\psi_{\Bm}$ is semismall with respect to 
the stratification $\CY_{\Bm} = \coprod_{\Bm' \le \Bm}\CY_{\Bm'}^0$ (note that 
$m_r = 0$).  It follows 
that $(\ol\psi_{\Bm})_!\a_0^*\CE[d_{\Bm}]$ is a semisimple perverse sheaf on 
$\CY_{\Bm}$.  
\par
For a given $\Bm \in \CQ_{n,r}$ (not necessarily in $Q_{n,r}^0$) 
we define, 
for each integer $1 \le j \le r-1$, and $0 \le k \le m_j$, 
\begin{align*}
M^{(j,k)} &= \prod_{i=1}^{j-1}(M_{[p_{i-1}+1, p_i]} + M_{p_{i-1}}) 
\times (M_{[p_{j-1} + 1, p_{j-1}+k]} + M_{p_{j-1}}) \times 
                  \prod_{i = j+1}^{r-1}M_{p_i}, \\
\CY_{j,k}^{0} &= \bigcup_{g \in H}g(T^{\io\th}\reg \times M^{(j,k)}), \\
\wt\CY_{j,k}^{+} &= \psi\iv(\CY^0_{j,k}),           
\end{align*}
and let $\psi_{j,k}: \wt\CY_{j,k}^{+} \to \CY_{j,k}^{0}$ be 
the restriciton of $\psi$ on $\wt\CY_{j,k}^{+}$.
(As a convention, we also consider the case where $j = 0, k = 0$.  
In which case $M^{(0,0)} = \prod_{i=1}^{r-1}M_{p_i}$.)
Then $\psi_{j,k}$ is a proper map. 
$\CY_{j,k}^0$ coincides with $\CY^0_{\Bm}$ in the case where
$j = r-1, k = m_{r-1}$, and coincides with $\CY_{\Bm}$ in the case where 
$j = 0$ and $k = 0$.
We also conisder the varieties
\begin{align*}
\ol M^{(j,k)} &= \prod_{i=1}^{j-1}(M_{[p_{i-1}+1, p_i]} + M_{p_{i-1}}) 
\times M_{p_{j-1}+k} \times 
                  \prod_{i = j+1}^{r-1}M_{p_i}, \\
\CY_{j,k} &= \bigcup_{g \in H}g(T^{\io\th}\reg \times \ol M^{(j,k)}), \\
\ol{\wt\CY_{j,k}^{+}} &= \psi\iv(\CY_{j,k}),           
\end{align*}
and let $\ol\psi_{j,k}: \ol{\wt\CY_{j,k}^+} \to \CY_{j,k}$ be the restriction of
$\psi$ on $\ol{\wt\CY_{j,k}^+}$.  Then $\CY_{j,k}^0$ is open dense in $\CY_{j,k}$. 
We have 
\par\medskip\noindent
(1.7.1) \ 
$\CY_{j,k} \backslash \CY_{j,k-1} = \CY_{j,k}^0$ if $k \ge 1$, and
$\CY_{j,0}(\Bm) = \CY_{j+1, m_j + m_{j+1}}(\Bm(j,0))$. 
Moreover, $\CY^0_{j,k}(\Bm)$ coincides with $\CY_{j+1, m_{j+1}}(\Bm(j,k))$.
(Here we use the notation $\CY_{j,k}(\Bm)$, etc. to inidcate the dependence 
on $\Bm$.)
\par\medskip
For $\Bm' \in \CQ(\Bm; j,k)$, $\CY_{\Bm'}^0$ is contained in $\CY_{j,k}$. 
Hence one can define an intersection cohomology $\IC(\CY'_{\Bm'}, \CL_{\r})$ 
associated to the local system $\CL_{\r}$ on $\CY^0_{\Bm'}$ (here $\CY'_{\Bm'}$
denotes the closure of $\CY_{\Bm'}^0$ in $\CY_{j,k}$).
Returning to the setting in the proposition, we consider $\Bm \in \CQ_{n,r}^0$.
We show the following formulas. 
First assume that $j = r-1$ and $0 \le k \le m_{r-1}$.  Then we have
\begin{equation*}
\tag{1.7.2}
\begin{split}
(\ol\psi_{r-1,k}&)_!\a_0^*\CE \\
&\simeq 
    \bigoplus_{0 \le k' \le k}\bigoplus_{\rho \in \CW_{\Bm(k'),\CE}\wg}
         \wt V_{\rho} \otimes \IC(\CY_{r-1, k'}, \CL_{\rho})[-2(m_{r-1} - k')]
            + \CN_{r-1,k},
\end{split}
\end{equation*}
where $\CN_{r-1,k}$ is a sum of various $\IC(\CY_{r-1,k'}, \CL_{\r})[-2i]$ 
for $0 \le k' \le k$ and $\r \in \CW_{\Bm(k'), \CE}\wg$ with 
$0 \le i < m_{r-1}-k'$. 
Next assume that $0 \le j < r-1$ and that $0 \le k \le m_j$.  Then we have 

\begin{equation*}
\tag{1.7.3}
\begin{split}
(&\ol\psi_{j,k})_*\a_0^*\CE \\
  &\simeq 
    \bigoplus_{0 \le k' \le m_{r-1}}\bigoplus_{\rho \in \CW_{\Bm(k'),\CE}\wg}
         \wt V_{\rho} \otimes \IC(\CY'_{\Bm(k')}, \CL_{\rho})[-2(m_{r-1} - k')]
             + \CN_{j,k},
\end{split}
\end{equation*}
where $\CN_{j,k}$ is a sum of various $\IC(\CY'_{\Bm'}, \CL_{\r})[-2i]$
for $\Bm' \in \CQ(\Bm;j,k)$ 
and $\r \in \CW\wg_{\Bm',\CE}$ with $i$ such that 
$0 \le 2i < d_{\Bm} - d_{\Bm'}$.

\par
Note that (1.7.3) will imply the proposition.   In fact, 
in the case where $j = 0, k = 0$, $\ol\psi_{j,k}$ coincides with $\ol\psi_{\Bm}$,
and $\CY'_{\Bm(k')}$ coincides with 
$\CY_{\Bm(k')}$.
Take $\IC(\CY_{\Bm'}, \CL_{\r})[-2i] \in \CN_{0,0}$. Since
$d_{\Bm} - d_{\Bm'}> 2i$, 
$\IC(\CY_{\Bm'}, \CL_{\r})[d_{\Bm} - 2i]$ is not a perverse sheaf. 
Since 
$(\ol\psi_{\Bm})_*\a_0^*\CE[d_{\Bm}]$ is a semisimple perverse sheaf, 
we conclude that $\CN_{0,0} = 0$. 
By Lemma 1.4, (iii), we have $d_{\Bm} - d_{\Bm(k')} = 2(m_{r-1} - k')$.
Thus the proposition follows from (1.7.3).
\par
First we show (1.7.2) by induction on $k$.  Put $j = r-1$. 
In the case where $k = 0$, $\CY_{j,0}$ coincides with $\CY_{\Bm(j,0)}^0$.  
Thus (1.7.2) follows from (1.6.5).  We assume that 
(1.7.2) holds for any $k' < k$. 
By (1.7.1), $\CY_{j,k} \backslash \CY_{j,k-1} \simeq \CY^0_{j,k} 
  = \CY_{\Bm(k)}^0$, 
and 
$\CY^0_{\Bm(k)}$ is an open dense subset of $\CY_{j,k}$.  
Since $\ol\psi_{j,k}$ is proper,
$(\ol\psi_{j,k})_!\a_0^*\CE$ is a semisimple complex on $\CY_{j,k}$.
Here we note that $(\ol\psi_{j,k})_!\a_0^*\CE$ has a natural structure of 
$\wt\CW_{\CE}$-complex.  In fact,
$(\psi_{\Bm(k)})_!\a_0^*\CE$ has a $\wt\CW_{\CE}$-action by (1.6.5).  It induces a 
$\wt\CW_{\CE}$-action on 
$(\io_0\circ\psi_{\Bm(k)})_!\a_0^*\CE$, where $\io_0$ is an open immersion 
$\CY_{\Bm(k)}^0 \hra \CY_{j,k}$, and hence on its perverse chohomology 
${}^pH^i((\io_0\circ\psi_{\Bm(k)})_!\a_0^*\CE)$.  On the other hand, by induction, 
$(\ol\psi_{j,k-1})_!\a_0^*\CE$ has a natural  $\wt\CW_{\CE}$-action, which induces a 
$\wt\CW_{\CE}$-action on ${}^pH^i((\ol\psi_{j,k-1})_!\a_0^*\CE)$. Thus, by using the
perverse cohomology exact sequence, one can define an action of $\wt\CW_{\CE}$ on 
${}^pH^i((\ol\psi_{j,k})_!\a_0^*\CE)$.  Since $(\ol\psi_{j,k})_!\a_0^*\CE$ is a 
semisimple complex, in this way the action of $\wt\CW_{\CE}$ on 
$(\ol\psi_{j,k})_!\a_0^*\CE$ can be defined. 
\par
Now, since $(\ol\psi_{j,k})_!\a_0^*\CE$ is a semisimple complex,  
it is a direct sum of 
the form $A[s]$ for a simple perverse sheaf $A$.   
Suppose that $\supp A$ is not contained in $\CY_{j,k-1}$.  
Then $\supp A \cap \CY_{\Bm(k)}^0 \ne \emptyset$ and  
the restriction of $A$ on $\CY^0_{\Bm(k)}$
is a simple perverse sheaf on $\CY_{\Bm(k)}^0$.  
The restriction of $(\ol\psi_{j,k})_!\a_0^*\CE$
on $\CY_{\Bm(k)}^0$ is isomorphic to $(\psi_{\Bm(k)})_!\a_0^*\CE$.  Hence 
its decomposition is given by the formula in (1.6.5).  It follows that 
$A|_{\CY^0_{\Bm(k)}} = \CL_{\rho}[d_{j,k}]$ for some $\r$ 
(here $d_{j,k} = \dim \CY_{\Bm(k)}$).
This implies that $A = \IC(\CY_{j,k},\CL_{\rho})[d_{j,k}]$ and 
that the direct sum of $A[s]$ appearing in $(\ol\psi_{j,k})_!\a_0^*\CE$ such that 
$\supp A \cap \CY_{\Bm(k)}^0 \ne \emptyset$ is given, in view of (1.6.6), by 
\begin{equation*}
\tag{1.7.4}
K_1 = \bigoplus_{\r \in \CW_{\Bm(k),\CE}\wg}\wt V_{\r} 
        \otimes \IC(\CY_{j,k}, \CL_{\r})[-2(m_{r-1} - k)] + \CN'_{\Bm(k)},
\end{equation*} 
where $\CN_{\Bm(k)}'$ is a sum of various $\IC(\CY_{j,k}, \CL_{\r})[-2i]$ 
with $  0 \le i < m_{r-1}-k$. 
\par
If $\supp A$ is contained in $\CY_{j,k-1}$, $A[s]$ appears as a direct summand 
of $(\ol\psi_{j,k-1})_!\a_0^*\CE$, which is decomposed as in (1.7.2) 
by induction hypothesis.  
Thus if we remove the contribution from 
the restriction of $K_1$, such $A[s]$ is determined from $(\ol\psi_{j,k-1})_!\a_0^*\CE$.
So, we consider the restriction of $K_1$ on $\CY_{j, k-1}$. 
The summands $\IC(\CY_{j,k}, \CL_{\r})[-2i]$ in $\CN_{\Bm(k)}'$ are already contained in 
$\CN_{r-1,k}$ if $0 \le i < m_{r-1}-k$.  So it is enough to consider 
$A = \IC(\CY_{j,k}, \CL_{\r})[-2(m_{r-1} - k)]$. 
Note that the multiplicity space of $A$ in $K_1$ is $\wt V_{\r}$. 
Hence the multiplicity space of a simple perverse 
sheaf $A'$ appearing in the decomposition of $A|_{\CY_{j,k-1}}$,up to shift, has a structure of 
$\wt\CW_{\CE}$-module which is a sum of $\wt V_{\r}$. 
But by (1.7.2) applied for $k-1$, the multiplicity space of a simple perverse sheaf $B$ appearing 
in the first term of $(\ol\psi_{j, k-1})_!\a_0^*\CE$ is a sum of $\wt V_{\r'}$ 
with $\r' \in \CW\wg_{\Bm(k'), \CE}$ for $k'< k$. Thus $A|_{\CY_{j,k-1}}$ gives 
no contribution on those first terms. This proves (1.7.2) for $k$.   
Hence (1.7.2) holds.     
\par
We now prove (1.7.3) by backwards induction on $j$ and induction on $k$.
So assume that $j < r-1$.  
By (1.7.1), 
$\CY_{j,0}(\Bm) = \CY_{j+1, m_j + m_{j+1}}(\Bm(j,0))$.  Hence by induction on $j$, 
we may assume that (1.7.3) holds for $\CY_{j,0}$.  Take $k \ge 1$, and 
assume that (1.7.3) holds for $k-1$.  We have 
$\CY_{j,k} \backslash \CY_{j, k-1} = \CY_{j,k}^0$, and 
$\CY_{j,k}^0(\Bm) = \CY_{j+1, m_{j+1}}(\Bm(j,k))$ by (1.7.1).  
Thus by induction on $j$, $(\psi_{j,k})_!\a_0^*\CE$ can be described by the 
formula in (1.7.3).  In particular, $(\psi_{j,k})_!\a_0^*\CE$ is a semisimple 
complex consisting of $\IC(\CY'_{\Bm'}, \CL_{\r})$, up to shift, for various 
$\Bm' \in \CQ(\Bm; j,k)$.  
Similarly, by induction on $k$, $(\ol\psi_{j-1,k})_!\a_0^*\CE$ is described by 
(1.7.3), and it is a semisimple complex consisitng of 
$\IC(\CY'_{\Bm'}, \CL_{\r})$, up to shift,  for various $\Bm' \in \CQ(\Bm; j, k-1)$.
Let $K_1$ be a semisimple complex on $\CY_{j,k}$ obtained from $(\psi_{j,k})_!\a_0^*\CE$
as in (1.7.4).  It is described by the formula (1.7.3) by replacing 
$\Bm$ by $\Bm'' = \Bm(j,k)$.  
Here we note that 
\par\medskip\noindent
(1.7.5) \ Assume that $A$ is a direct summand of $K_1$.  Then except the 
case where $A = \IC(\CY_{\Bm(k')}, \CL_{\r})[-2(m_{r-1} - k')]$, 
$A$ is contained in $\CN_{j,k}$.
\par\medskip

Assume that $A$ is a direct summand in the former 
part of $K_1$.  Then $A = \IC(\CY_{\Bm''(k')}, \CL_{\r})[-2(m''_{r-1}-k')]$
for $0 \le k' \le m''_{r-1}$, where we write $\Bm'' = (m_1'', \dots, m_r'')$.
Then $\Bp(\Bm''(k')) = (p_1', \dots, p_r')$, where $p_i = p_i'$ for $i \ne j, r-1$,
and $p'_j \le p_j, p'_{r-1} \le p_{r-1}$.  Then by Lemma 1.4 (iii), we have
\begin{align*}
d_{\Bm} - d_{\Bm''(k')} &= (p_j - p_j') + (p_{r-1} - p'_{r-1}) + m'_r  \\
                        &\ge (p_j - p_j') + 2m_r'.
\end{align*}
for $\Bm''(k') = (m_1', \dots, m_r')$.  (Note that $m_{r-1}''-k' = m_r'$.) 
Hence $d_{\Bm} - d_{\Bm''(k')} = 2(m_{r-1}''-k')$ if and only if $p_j = p_j'$, i.e., 
$\Bm''(k') = \Bm(k')$.  In that case $m''_{r-1} = m_{r-1}$.  
In all other cases, $A$ is contained in $\CM_{j,k}$. 
Next assume that $A$ is a direct summand of the latter part of $K_1$.  Thus 
$A$ is written as $A = \IC(\CY'_{\Bm'}, \CL_{\r})[-2i]$ for $\Bm' \le \Bm''$ 
with $d_{\Bm''} - d_{\Bm'} > 2i$. 
But since $\Bm'' = \Bm(j,k)$, we have $\Bm'' \le \Bm$ and $d_{\Bm} \ge d_{\Bm''}$.
Hence $\Bm' \le \Bm$ and $d_{\Bm} - d_{\Bm'} \ge d_{\Bm''} - d_{\Bm'} > 2i$.
This implies that $A$ is contained in $\CN_{j,k}$.    
(1.7.5) is proved.
\par
Now (1.7.5) shows that the former part of $K_1$ coinicdes with the former 
part of $(\ol\psi_{j,k-1})_!\a_0^*\CE$.  Hence by a similar argument as in the 
proof of (1.7.2), we obtain (1.7.3) for $(j,k)$. This proves (1.7.3), and so  
the proposition follows.
\end{proof}

\par\bigskip
\section{Intersection cohomology on $G^{\io\th} \times V^{r-1}$ (exotic case)}

\para{2.1.}
In this section we assume that $\CX, \CY$ are of exotic type.  
We keep the notation in Section 1.  For each $\Bm \in \CQ^0_{n,r}$, 
we consider 
the complex $(\ol\psi_{\Bm})_!\a_0^*\CE[d_{\Bm}]$ as in Proposition 1.7.
Under the notation there, $\CY_{\Bm(k)}^0$ is an open dense subset of 
$\CX_{\Bm(k)}$.  
Hence one can consider the complex 
\begin{equation*}
\tag{2.1.1}
K_{\Bm, T,\CE} = \bigoplus_{0 \le k \le m_{r-1}}\bigoplus_{\r \in \CW_{\Bm(k),\CE}\wg}
                   \wt V_{\rho}\otimes \IC(\CX_{\Bm(k)}, \CL_{\rho})[d_{\Bm(k)}].
\end{equation*} 
We consider the diagram

\begin{equation*}
\begin{CD}
T^{\io\th} @<\a <<  \wt\CX @>\pi >>  \CX,
\end{CD}
\end{equation*}
where $\a: \wt\CX \to T^{\io\th}$ is defined by 
$\a(x,\Bv,gB^{\th}) = p_T(g\iv xg)$.
Let $\ol\pi_{\Bm} : \pi\iv(\CX_{\Bm}) \to \CX_{\Bm}$ 
be the rstriction of $\pi$ on $\pi\iv(\CX_{\Bm})$. 
We consider the complex $(\ol\pi_{\Bm})_!\a^*\CE[d_{\Bm}]$, 
where $\a^*\CE$ is regarded as a local system on $\pi\iv(\CX_{\Bm})$
by restriction. THe following result is a generalization of 
[SS1, Theorem 4.2].

\begin{thm}  %%%  Theorem 2.2
For each $\Bm \in \CQ_{n,r}^0$, 
$(\ol\pi_{\Bm})_!\a^*\CE[d_{\Bm}] \simeq K_{\Bm,T, \CE}$ as 
perverse sheaves on $\CX_{\Bm}$. 
\end{thm} 
 
\para{2.3.}
The remainder of this section is devoted to the proof of the theorem. 
As in the case of $\CY_{\Bm}^0$,  consider 
$\CX_{\Bm}^0 = \CX_{\Bm}\ \backslash \ \bigcup_{\Bm' < \Bm}\CX_{\Bm'}$
for each $\Bm \in \CQ_{n,r}$. 
We shall describe the set $\CX_{\Bm}^0$ explicitly. Put 
$\CX_{\unip} = \CX \cap (G^{\io\th}\uni \times V^{r-1})$ and 
we define $\CX_{\Bm,\unip}, \CX^0_{\Bm, \unip}$, etc. as the intersection 
of $\CX_{\Bm}, \CX_{\Bm}^0$, etc. with $\CX\uni$.
For $(x,v) \in G^{\io\th}\uni \times V$, we denote 
by $\Bk[x]v$ the subspace of $V$ spanned by $v, xv, x^2v, \dots$.  
First we note the following.

\par\medskip\noindent
(2.3.1) \ For each $(x,\Bv) \in \CX\uni$ with $\Bv = (v_1, \dots, v_{r-1})$, 
define a sequence $W_1 \subset W_2 \subset \cdots \subset W_{r-1}$ 
of subspaces of $V$ by 
$W_i = \Bk[x]v_1 + \cdots + \Bk[x]v_i$ for $i = 1,\dots, r-1$.
Then $(x,\Bv) \in \CX_{\Bm,\unip}^0$ if and only if $\dim W_i = p_i$ for each $i$. 

\par\medskip
In fact, if $(x,\Bv) \in \CX_{\Bm,\unip}$, there exists an $x$-stable isotropic flag
$(V_i)_{1 \le i \le n}$ such that $v_i \in V_{p_i}$.  
Hence we have $\dim W_i \le p_i$ for each $i$. This implies that 
$(x,\Bv)$ satisfying the condition on $(W_i)$ is contained in $\CX_{\Bm,\unip}^0$. 
Conversely, assume that $(x,\Bv) \in \CX_{\Bm,\unip}^0$. Take an $x$-stable 
isotropic flag $(V_i)$ such that $v_i \in V_{p_i}$. Suppose there exists $k$ such that
$\dim W_i = p_i$ for $i < k$ and that $\dim W_k < p_k$. 
Then  $W_i = V_{p_i}$ for $i = 1, \dots, k-1$, and  
$\Bk[x]v_k + V_{k-1}$ is an $x$-stable  proper subspace of $V_{p_k}$. 
One can find an $x$-stable flag $V_{p_{k-1}}\subset V'_{p_{k-1}+1} \subset\cdots 
\subset V'_{p_k-1} \subset V_{p_k}$, and $v_k \in V_j'\subsetneq V_{p_k}$ for some $j$.
This implies that $(x,\Bv) \in \CX_{\Bm'}$ for $\Bm' < \Bm$, a contradiction.  
Hence (2.3.1) holds.

\par\medskip
More generally, we consider $(x,\Bv) \in \CX_{\Bm}$. 
Let $x = su = us$ be the Jordan decomposition of $x \in G^{\io\th}$, 
where $s \in G^{\io\th}$ is semisimple, $u \in G^{\io\th}$ is unipotent.  
We consider the 
decoposition $V = V_1\oplus\cdots \oplus V_t$ into eigenspaces of $s$. 
Then $Z_G(s) \simeq GL_{2n_1} \times \cdots \times GL_{2n_t}$ with 
$\dim V_j = 2n_j$.  Put $G_j = GL_{2n_j}$ for each $j$.  Then $Z_G(s)$ is 
$\th$-stable, and $\th$ stabilizes each factor so that 
$Z_H(s) \simeq G_1^{\th} \times \cdots \times G_t^{\th}$ with 
$G_j^{\th} \simeq Sp_{2n_j}$.   
Take $\Bv = (v_1, \dots, v_{r-1}) \in V^{r-1}$. For $j = 1, \dots, t$, 
we define $\Bv_j = (v_{1,j}, \dots v_{r-1,j}) \in V^{r-1}$, where 
$v_{ij}$ is the projection of $v_i \in V$ on $V_j$.   
Let $u_j$ be the restriction of $u$ on $V_j$.  Then 
$(u_j, \Bv_j) \in (G_j)\uni^{\io\th} \times V_j^{r-1}$. 
We denote by $(\CX^{G_j})^0_{\Bm_j, \unip}$ the subvariety 
of $(G_j)^{\io\th}\uni \times V_j^{r-1}$
defined in a similar way as $\CX_{\Bm, \unip}^0$. 
The following property is checked easily. 

\par\medskip\noindent
(2.3.2) \ Assume that $(x,\Bv)$ is contained in $\CX_{\Bm}^0$.
Then there exist unique $\Bm_1, \dots, \Bm_t$ such that
$(u_j,\Bv_j) \in (\CX^{G_j})_{\Bm_j,\unip}^0$, where 
$\Bm_j = (m_{1,j}, \dots m_{r, j})$ with $\sum_{j=1}^t m_{ij} = m_i$
for $1 \le i \le r$. 
Coversely, if $(u_j, \Bv_j) \in (\CX^{G_j})^0_{\Bm_j, \unip}$ for each $j$, 
then $(x, \Bv) \in \CX_{\Bm}^0$ for $\Bm$ determined from 
$\Bm_1, \dots, \Bm_t$.

\par\medskip
For each $(x,\Bv) \in \CX$, let $(u_j, \Bv_j)$ be defined as above. 
For $j = 1, \dots, t$, we define a flag $W_{1,j} \subset \cdots \subset W_{r-1,j}$
of $V_j$ with respect to $(u_j, \Bv_j)$ as in (2.3.1).  Put 
$W_i = W_{i,1}\oplus\cdots\oplus W_{i,t}$ for $i = 1, \dots, r-1$, then  
we obtain a sequence $W_1 \subset \cdots \subset W_{r-1}$ of subspaces in $V$. 
We put $W_i(x,\Bv) = W_i$ for each $i$.  Then (2.3.2) can be rewritten as 
\begin{equation*}
\tag{2.3.3}
\CX_{\Bm}^0 = \{ (x,\Bv) \in \CX \mid \dim W_i(x,\Bv) = p_i \  (1 \le i \le r-1) \}.
\end{equation*}   

Recall the map $\pi: \wt\CX \to \CX$.  For each $\Bm \in \CQ_{n,r}$, 
we define $\wt\CX_{\Bm}^+ = \pi\iv(\CX_{\Bm}^0)$, and 
let $\pi_{\Bm} : \wt\CX^+_{\Bm} \to \CX_{\Bm}^0$ be the restriction of $\pi$
on $\wt\CX^+_{\Bm}$. 
Since $\CY_{\Bm}^0$ is open in $\CX_{\Bm}^0$, $\wt\CY_{\Bm}^+$ is an open 
subset of $\wt\CX_{\Bm}^+$.
For $(x,\Bv, gB^{\th}) \in \wt\CX_{\Bm}^+$, we shall associate  
$\BI \in \CI(\Bm)$ as follows; 
assume that $(x,\Bv) \in B^{\io\th} \times M_n^{r-1}$, and that $x = su$ is the 
Jordan decomposition of $x$. 
Then $M_n$ is $s$-stable, and is decomposed as 
$M_n = M_{n,1}\oplus\cdots\oplus M_{n,t}$, where $M_{n,j} = M_n \cap V_j$ 
is a maximal isotropic subspace of $V_j$. Here $M_n = \lp e_1, \dots e_n\rp$.
Since $s \in B^{\io\th}$, $M_{n,j}$ determines a set
$\{ e_{k_1}, \dots, e_{k_{n_j}}\}$ with $k_1 < k_2 < \dots < k_{n_j}$,
where $n_j = \dim M_{n,j}$ (if $s_1 \in T^{\io\th}$ is the projection of $s$, 
$\{ e_{k_1}, \dots, e_{k_{n_j}}\}$ are eigenvectors of $s_1$ on $M_{n,j}$).   
.
Let $(u_j, \Bv_j)$ be as before. 
We define $\BI= (I_1, \dots, I_{r}) \in \CI(\Bm)$ as follows; 
let $(W_i(u_j, \Bv_j))$ be as in (2.3.1) and put $p_{ij} = \dim W_i(u_j,\Bv_j)$. 
Then $W_i(u_j, \Bv_j)$ is a subspace of $M_{n,j}$. 
We define a subset $\wt I_{i,j}$ of $\{ k_1, \dots, k_{n_j}\}$ as the first 
$p_{i,j}$ numbers in $\{ k_1, \dots, k_{n_j}\}$, and define $I_{i,j}$ by 
$I_{i,j} = \wt I_{i,j} \backslash \wt I_{i-1,j}$.  Thus $|I_{i,j}| = m_{i,j}$ 
for $\Bm_j = (m_{1,j}, \dots m_{r,j})$. 
Put $I_i = I_{i,1}\coprod \dots \coprod I_{i,t}$, and 
$\BI = (I_1, \dots, I_r)$.  Then $\BI \in \CI(\Bm)$.     
Note that the attachment $(x,\Bv) \mapsto \BI$ depends only on the 
$B^{\th}$-conjugacy class of $(x, \Bv)$.  Thus we have a well-defined map
$(x,\Bv, gB^{\th}) \mapsto \BI$. 
We define a subvariety $\wt\CX_{\BI}$ of $\wt\CX_{\Bm}^+$ by 
\begin{equation*}
\wt\CX_{\BI} = \{(x,\Bv, gB^{\th}) \in \wt\CX_{\Bm}^+ \mid 
                       (x,\Bv, gB^{\th}) \mapsto \BI \}. 
\end{equation*}
We show the following lemma (cf. [SS1, Lemma 4.4]).

\begin{lem} %%%% Lemma 2.4
$\wt\CX_{\Bm}^+$ is decomposed as
\begin{equation*}
\wt\CX_{\Bm}^+ = \coprod_{\BI \in \CI(\Bm)}\wt\CX_{\BI},
\end{equation*}
where $\wt\CX_{\BI}$ is an irreducible component of $\wt\CX_{\Bm}^+$
for each $\BI$. 
\end{lem}

\begin{proof}
It is clear from the definition that $\wt\CX_{\Bm}^+$ is a disjoint union of 
various $\wt\CX_{\BI}$, and that $\wt\CX_{\BI}$ contains $\wt\CY_{\BI}$ as an 
open dense subset. Since $\wt\CY_{\Bm}^+ = \coprod_{\BI}\wt\CY_{\BI}$,  
$\wt\CY_{\Bm}^+$ is open dense in $\wt\CX_{\Bm}^+$.  Hence 
$\wt\CX_{\Bm} = \bigcup_{\BI}\ol{\wt\CY}_{\BI}$ gives a decomposition into 
irreducible components, where $\ol{\wt\CY}_{\BI}$ is the closure of $\wt\CY_{\BI}$ in 
$\wt\CX_{\Bm}^+$.  Thus in order to show the lemma, it is enough to see that 
$\wt\CX_{\BI}$ is closed in $\wt\CX_{\Bm}^+$ for each $\BI$. 
But the closure $\CZ_{\BI}$ of $\wt\CX_{\BI}$ in $\CX$ is contained in the set 
$\wt\CX_{\BI} \cup\bigcup_{\Bm'< \Bm}\wt\CX_{\Bm'}^+$.  Hence 
$\wt\CX_{\BI} = \CZ_{\BI} \cap \wt\CX_{\Bm}^+$ is closed in $\wt\CX_{\Bm}^+$.   
\end{proof}

\para{2.5.}
We fix $\Bm \in \CQ_{n,r}$.  
Let us consider the spaces $V_0 = M_{m_1}$ and $\ol V_0 = V_0^{\perp}/V_0$. 
We put $G_1 = GL(V_0)$ and $G_2 = GL(\ol V_0)$. 
Then $\ol V_0$ has a natural symplectic structure, and $G_2$ is identified with 
a $\th$-stable subgroup of $G$. 
We consider the variety $\CX' \subset G_2^{\io\th} \times \ol V_0^{r-2}$ as 
in the case of $G^{\io\th} \times V^{r-1}$.  Put $\Bm' = (m_2, \dots, m_r)$.
Thus $\Bm' \in \CQ_{n',r-1}$, where $n' = \dim \ol V_0/2$. 
The subvariety $\CX_{\Bm'}'$ of $\CX'$ with respect to $\Bm'$ is defined 
similar to $\CX_{\Bm}$. 
Let $G_1^0$ be the set of regular elements in $G_1$ (namely, the set of $x \in G_1$
such that $u$ is regular unipotent in $Z_{G_1}(s)$ for the Jordan decomposition 
$x = su$). 
For each $z = (x, \Bv) \in \CX_{\Bm}^0$, put $W_z = W_1(x, \Bv)$.   
Note that $W_z$ is an $x$-stable subspace of $V$ with $\dim W_z = m_1$, and 
that $x|_{W_z}$ is a regular element in $GL(W_z)$. 
Moreover, $W_z$ is the unique $x$-stable subspace of $V$ containing $v_1$ 
with dimension $m_1$.  
\par
We define a variety 
\begin{equation*}
\tag{2.5.1}
\begin{split}
\CK_{\Bm} = \{&(z, \f_1, \f_2) \mid z = (x,\Bv) \in \CX_{\Bm}^0 \\
              &\f_1: W_z \isom V_0, \f_2: W_z^{\perp}/W_z \isom \ol V_0
                \text{ (symplectic isom.) } \},
\end{split}
\end{equation*}
and morphisms
\begin{align*}
q: &\CK_{\Bm} \to \CX^0_{\Bm}, \quad (x, \Bv, \f_1, \f_2) \mapsto (x, \Bv), \\ 
\s: &\CK_{\Bm} \to G^0_1 \times \CX'^0_{\Bm'}, \\
       &(x, \Bv,\f_1,\f_2) \mapsto (\f_1(x|_{W_z})\f_1\iv, 
         \f_2(x|_{W_z^{\perp}/W_z})\f_2\iv, \f_2(\ol\Bv)), 
\end{align*}
where $\ol\Bv = (\ol v_2, \dots, \ol v_{r-1}) \in (W_z^{\perp}/W_z)^{r-2}$, 
and $\ol v_i$ is the image of $v_i \in W_z^{\perp}$ to $W_z^{\perp}/W_z$.   
Put $H_0 = G_1 \times G_2^{\th}$.  Then $H \times H_0$ acts on $\CK_{\Bm}$ by 
\begin{equation*}
(g, (h_1,h_2)): (x,\Bv, \f_1,\f_2) \mapsto 
        (gxg\iv, g\Bv,  h_1\f_1g\iv, h_2\f_2g\iv)
\end{equation*}
for $g \in H, (h_1,h_2) \in H_0$.
Moreover, $\s$ is $H \times H_0$-equivariant with respec to the natural action of $H_0$ 
and the trivial action of $H$ on $G^0_1\times \CX'^0_{\Bm'}$.  
We have 
\par\medskip\noindent
(2.5.2) \ The map $q$ is a principal bundle with fibre isomorphic to $H_0$. 
\par\medskip\noindent
(2.5.3) \ The map $\s$ is a locally trivial fibration with smooth fibre of 
dimension $\dim H + (r-2)m_1$.
\par\medskip
In fact, (2.5.2) is clear.  We show (2.5.3). 
For a fixd $z = (x', (x'',\Bv')) \in G^0_1 \times \CX'^0_{\Bm'}$ 
with $\Bv' = (v_2', \dots, v_{r-1}')$, the fibre 
$\s\iv(z)$ is determined by the following procedure. 
\par\medskip
(i) \ Choose an isotropic subspace $W_1$ of $V$ with $\dim W_1 = m_1$. 
\par
(ii) \ For such $W_1$, choose an isomorphism $\f_1 : W_1 \to V_0$ and a symplectic 
isomorphism $\f_2: W_1^{\perp}/W_1 \to \ol V_0$. 
\par
(iii) \ Choose $x \in G^{\io\th}$ such that $\f_1(x|_{W_1})\f_1\iv = x'$, 
$\f_2(x|_{W_1^{\perp}/W_1})\f_2\iv = x''$.  
\par
(iv) \ Choose $v_1 \in W_1$ and $v_i \in W_1^{\perp}$ such that $\f_2(\ol v_i) = v_i'$ for 
$i = 2, \dots, r-1$. 

\par\medskip
Let $P$ be the stabilizer of the flag $(V_0 \subset V_0^{\perp})$ in $G$.
Then $P$ is $\th$-stable, and is decomposed as $P = LU_P$, where $L$ is 
a $\th$-stable Levi subgorup of $P$ containing $T$ and $U_P$ is the unipotent 
radical of $P$.  For (i), such $W_1$ are parametrized by $H/P^{\th}$. 
For (ii), they are paramerrized by $G_1 \times G_2^{\io\th}$.  For (iii), 
$x$ should be contained in $P^{\io\th}$, but $x', x''$ determines the part 
corresponding to $L^{\io\th}$.  Hence the choice of $x$ is parametrized by 
$U_P^{\io\th}$.  Finally, $v_1$ is any element in $W_1$, $v_2, \dots, v_{r-1}$
are determined uniquely by $v_2', \dots, v_{r-1}'$ modulo $W_1$.    
One can check that thus obtained $(x, \Bv)$ is contained in $\in X^0_{\Bm}$.
It follows that each fibre $\s\iv(z)$ is smooth with dimension 
$\dim H + (r-2)m_1$. Hence (2.5.3) holds. 
\par
Let $B_1$ be a Borel subgorup of $G_1$ which is the stabilizer of the flag 
$(M_k)_{0 \le k \le m_1}$ in $G_1$, and $B_2$ a $\th$-stable Borel subgroup of 
$G_2$ which is the stabilizer of the flag 
$(M_{m_1+1}/M_{m_1} \subset \cdots \subset M_n/M_{m_1})$ in $G_2$.   
Put
\begin{equation*}
\wt G_1 = \{ (x,gB_1) \in G_1 \times G_1/B_1 \mid g\iv xg \in B_1 \},
\end{equation*}
and define the map $\pi^1: \wt G_1 \to G_1$ by $(x,gB_1) \mapsto x$. 
Put $\wt G_1^0 = (\pi^1)\iv(G^0_1)$, and let $\vf^1: \wt G^0_1 \to G^0_1$ 
be the restriction of $\pi^1$. 
We define $\wt\CX'$ as the subvariety of 
$G_2^{\io\th} \times \ol V_0^{r-2} \times G_2^{\th}/B_2^{\th}$ as in the 
case of $\CX'$, and let $\pi^2: \wt\CX' \to \CX'$ be the projection 
$(x,\Bv', gB^{\th}_2) \mapsto (x,\Bv')$. 
We put $\wt\CX'^+_{\Bm'} = (\pi^2)\iv ({\CX'}^0_{\Bm'})$, and 
let $\pi^2_{\Bm'}$ be the restriction of $\pi^2$ on $\wt\CX'^+_{\Bm'}$.
We define a variety 
\begin{equation*}
\begin{split}
\wt\CZ_{\Bm}^+ = \{ (x,\Bv,&gB^{\th},\f_1,\f_2) \mid (x,\Bv,gB^{\th}) \in \wt\CX_{\Bm}^+, \\  
    &\f_1: W_z \isom V_0, \f_2: W_z^{\perp}/W_z \isom \ol V_0  \text{ for } z = (x, \Bv)\}, 
\end{split}        
\end{equation*}  
and define a map $\wt q: \wt\CZ_{\Bm}^+ \to \wt\CX_{\Bm}^+$ by a natural projection. 
We define a map $\wt\s : \wt\CZ_{\Bm}^+ \to \wt G^0_1 \times \wt\CX'^+_{\Bm'}$ as follows; 
take $(x,\Bv,gB^{\th},\f_1,\f_2) \in \wt\CZ_{\Bm}^+$.  
Since $z = (x, \Bv) \in \CX_{\Bm}^0$, $W_z$ coinicdes with $g(M_{m_1})$.
Let $g_1B_1$ be the element corresponding to the flag 
$\f_1(g(M_i))_{0 \le i \le m_1}$, and $g_2B_2^{\th}$ be the element 
corresponding to the isotorpic flag $\f_2(g(M_i)/g(M_{m_1}))_{i \ge m_1}$.    
Then 
\begin{equation*}
\begin{split}
\wt\s : (x, &\Bv, gB^{\th}, \f_1, \f_2) \mapsto  \\ 
  &((\f_1(x|_{W_z})\f_1\iv, g_1B_1), (\f_2(x|_{w_z^{\perp}/W_z})\f_2\iv, \f_2(\ol \Bv), g_2B_2^{\th})).  
\end{split}
\end{equation*}
We also define a map $\wt\pi_{\Bm}: \wt\CZ_{\Bm}^+ \to \CK_{\Bm}$ by 
$(x,\Bv,gB^{\th}, \f_1, \f_2) \mapsto (x, \Bv, \f_1, \f_2)$.
Then we have the following commutative diagram.
\begin{equation*}
\begin{CD}
T_1 \times T_2^{\io\th} @<f <<  T^{\io\th} @>\id >> T^{\io\th}  \\
@A\a^1 \times \a' AA    @AA\wt\a A   @AA\a A    \\
\wt G^0_1 \times \wt\CX_{\Bm'}'^+  @<\wt\s<<   \wt\CZ_{\Bm}^+  @>\wt q >>  \wt\CX^+_{\Bm}  \\
@V\vf^1 \times \pi_{\Bm'} VV     @VV\wt\pi_{\Bm} V    @VV\pi_{\Bm} V   \\
G^0_1 \times \CX'^0_{\Bm'}  @ <\s <<  \CK_{\Bm}  @> q >>  \CX^0_{\Bm}, 
\end{CD}
\end{equation*}
where the map $\wt\a$ is defined naturally.  Note that $T^{\io\th}$ can be written as
$T^{\io\th} \simeq T_1 \times T_2^{\io\th}$, where $T_1$ is a maximal torus of $G_1$, and 
$T_2$ is a $\th$-stable maximal torus of $G_2$. We fix an isomorphism 
$f : T^{\io\th} \to T_1 \times T_2^{\io\th}$.
The map $\a^1: \wt G_1 \to T_1$ is defined as in 2.1, by ignoring $v$.  The maps 
$\a',\pi_{\Bm'}$ are defined similarly to $\a,\pi_{\Bm}$. 

\para{2.6.}
Let $T$ be a tame local system on $T^{\io\th}$. Under the isomorphism 
$f: T^{\io\th} \to T_1 \times T_2^{\io\th}$, $\CE$ can be written as 
$\CE \simeq \CE_1\boxtimes \CE_2$, where $\CE_1$ (resp. $\CE_2$) is a tame local 
system on $T_1$ (resp. $T_2^{\io\th}$). 
Then we have $\CW_{\Bm,\CE} \simeq \CW_1 \times \CW'_{\Bm',\CE_2}$, 
where $\CW_1$ is the stabilizer of $\CE_1$ in 
$S_{m_1}\simeq N_{G_1}(T_1)/T_1$, and $\CW'_{\Bm',\CE_2}$ is defined 
similarly to $\CW_{\Bm,\CE}$ with respect to 
$\CW' = N_{G_2^{\th}}(T_2^{\io\th})/Z_{G_2^{\th}}(T_2^{\io\th})$.  
As in 1.6, $\CW_{\Bm,\CE}$ is decomposed as 
$\CW_{\Bm,\CE} \simeq \CW_1 \times\cdots \times \CW_r$ with subgroups 
$\CW_i \subset S_{m_i}$.  Then $\CW'_{\Bm',\CE_2} \simeq \CW_2 \times \cdots \times \CW_r$. 
For each $\r \in \CW_{\Bm,\CE}\wg$, we construct a simple perverse sheaf $A_{\rho}$ on 
$\CX^0_{\Bm}$ as follows; 
The decomposition of the complex $\pi^1_!(\a^1)^*\CE_1[\dim G_1]$ into simple 
summands is well-known.  
Since $G_1^0$ is an open dense subset of $G_1$ containing $(G_1)\reg$, 
the decompostion of $\vf^1_!\a^*\CE_1$ is described similarly, namely we have
\begin{equation*}
\tag{2.6.1}
\vf^1_!(\a^1)^*\CE_1 \simeq \bigoplus_{\r_1 \in \CW_1\wg}\r_1
        \otimes \IC(G^0_1, \CL_{\r_1}),  
\end{equation*} 
where $\CL_{\r_1}$ is a simple local system on the set of regular semisimple elements 
in $G_1$. 
Write $\r$ as  $\r = \r_1\boxtimes\cdots\boxtimes\r_r$ with
$\r_i \in \CW_i\wg$. Then $\r' = \r_2\boxtimes\cdots\boxtimes\r_r \in \CW_{\Bm',\CE_2}\wg$.
Suppose that a simple perverse sheaf $A_{\r'}$ on $\CX^0_{\Bm'}$ was constructed. 
Put $A_1 = \IC(G^0_1,\CL_{\r_1})[\dim G_1]$.  Then $A_1\boxtimes A_{\r'}$ is an 
$H_0$-equivariant simple perverse sheaf on $G^0_1 \times \CX'^0_{\Bm'}$, and so 
$\s^*(A_1\boxtimes A_{\r'})[\b_1]$ is an $H_0$-equivariant  simple perverse sheaf on 
$\CK_{\Bm}$ by
(2.5.3), where $\b_1 = \dim H + (r-2)m_1$. 
Since $q$ is a principal bundle with group $H_0$ by (2.5.2), there exists a unique simple 
perverse sheaf $A_{\r}$ on $\CX^0_{\Bm}$ such that
\begin{equation*}
q^*A_{\r}[\b_2] \simeq \s^*(A_1\boxtimes A_{\r'})[\b_1], 
\end{equation*}
where $\b_2 = \dim H_0$. 
Note that in the case where $r = 2$, $A_{\r}$ coincides with the simple perverse sheaf 
$A_{\r}$ constructed in [SS1, 4.6]. 
\par
Let $\CL_{\r}$ be a simple local system on $\CY_{\Bm}^0$ appeared  in (1.5.3).
Since $\CY_{\Bm}^0$ is an open dense smooth subset of $\CX_{\Bm}^0$, 
one can consider the intersection cohomology $\IC(\CX_{\Bm}^0, \CL_{\r})$ on
$\CX_{\Bm}^0$.  
We have the following lemma.

\begin{lem}  %%%  Lemma 2.7.
$A_{\r} \simeq \IC(\CX_{\Bm}^0,\CL_{\r})[d_{\Bm}]$.  
\end{lem}  
\begin{proof}
In order to prove the lemma, 
it is enough to see that 
\begin{equation*}
\tag{2.7.1}
\CH^{-d_{\Bm}}A_{\r}|_{\CY_{\Bm}^0} \simeq \CL_{\r}.
\end{equation*}
 
We consider the following commutative diagram.

\begin{equation*}
\tag{2.7.2}
\begin{CD}
T_1 \times T_2^{\io\th}  @<f <<  T^{\io\th} @>\id >> T^{\io\th}  \\
@A\a_0^1 \times \a_0' AA    @AA\wt\a_0A   @AA\a_0A      \\
\wt G_{1,\rg}\times \wt\CY_{\Bm'}'^0   @<\wt\s_0<<  \wt\CZ^0_{\Bm}  
         @>\wt q_0>>   \wt\CY_{\Bm}^0  \\
@V\xi^1 \times \xi'_0 VV    @VV\wt\xi_0V     @VV\xi_0V   \\
(G_1/T_1 \times T_{1,\rg}) \times \wh\CY'^0_{\Bm'}  @<\wh \s_0<< \wh\CZ^0_{\Bm}  
                   @>\wh q_0>> \wh\CY_{\Bm}^0 \\  
@V\e^1 \times \e'_0 VV    @VV\wt\e_0V   @VV\e_0V   \\
G_{1,\rg} \times \CY'_{\Bm'}  @<\s_0<<   \CK_{\Bm,\rg} @>q_0>>  \CY^0_{\Bm},  \\
\end{CD}
\end{equation*}
where $\wt\CY_{\Bm}^0 = \wt\CY_{\BI}$, $\wh\CY_{\Bm}^0 = \wh\CY_{\BI}$ 
for $\BI = \BI(\Bm)$ (see 1.3), and $\wt\CY'^0_{\Bm'} = \wt\CY'_{\BI'}, 
\wh\CY'^0_{\Bm'} = \wh\CY'_{\BI'}$ 
are defined similarly with respect to $G_2^{\io\th}\times \ol V_0^{r-2}$
with $\BI' = \BI(\Bm')$. 
Moreover, 
\begin{align*}
\wt G_{1,\rg} &= (\pi^1)\iv(G_{1,\rg}), \\
\CK_{\Bm,\rg} &= q\iv(\CY^0_{\Bm}), \\
\wt\CZ^0_{\Bm}  &= \wt q\iv(\wt\CY_{\Bm}^0)
\end{align*}
and $\wh\CZ^0_{\Bm}$ is defined as the quotient of $\wt\CZ_{\Bm}^0$ under 
the natural action of $Z_H(T^{\io\th})_{\BI}/(B^{\th} \cap Z_H(T^{\io\th}))$.
The maps $\wt q_0, q_0, \wt\s_0, \s_0$ are defined as the restriction of the
corresponding maps $\wt q,q, \wt\s, \s$. The map $\xi_0$ is $\xi_{\BI}$ 
for $\BI = \BI(\Bm)$. 
The maps $\wt\xi_0, \wt\e_0$ are defined according to $\xi_0,\e_0$. 
$\xi'_0, \e_0'$ are defined similar to $\xi_0,\e_0$ with respect 
to $G_2^{\io\th} \times \ol V_0^{r-2}$.
$\xi^1, \e^1$ are standard maps in the groups case ($\xi^1$ is an isomorphism). 
The map $\wh\s_0$ is naturally induced from $\wt\s_0$.
\par
It follows from the diagram (2.7.2) that 
\begin{equation*}
\tag{2.7.3}
\wt\s_0^*((\a_0^1)^*\CE_1\boxtimes (\a_0^2)^*\CE_2) \simeq \wt q_0^*\a_0^*\CE. 
\end{equation*}       
It is easy to check that the squares in the middle row and in the bottom row are 
all cartesian squares.  Here $\psi^1 = \e^1\circ\xi^1 : \wt G_{1,\rg} \to G_{1,\rg}$
is a finite Galois covering with group $S_{m_1}$, and $\psi^1_!(\a_0^1)^*\CE_1$ is 
decomposed as 
\begin{equation*}
\tag{2.7.4}
\psi^1_!(\a_0^1)^*\CE_1 \simeq \bigoplus_{\r_1 \in \CW_1\wg}\r_1\otimes \CL_{\r_1}. 
\end{equation*} 
On the other hand, 
by (1.5.3) and (1.6.3), we have
\begin{equation*}
\tag{2.7.5}
(\e'_0\circ\xi_0')_!(\a^2_0)^*\CE_2 \simeq \bigoplus_{\r' \in (\CW'_{\Bm',\CE_2})\wg}
         H^{\bullet}(\BP_1^{m_r})\otimes  \r'\otimes \CL_{\r'}.
\end{equation*}
Similarly, the map $\e_0\circ\xi_0$ coincides with 
$\psi_{\BI}: \wt\CY_{\BI} = \wt\CY_{\Bm}^0 \to \CY_{\Bm}^0$. Hence we have
\begin{equation*}
\tag{2.7.6}
(\e_0\circ\xi_0)_!\a_0^*\CE \simeq \bigoplus_{\r \in \CW_{\Bm,\CE}\wg}
              H^{\bullet}(\BP_1^{m_r})\otimes\r\otimes \CL_{\r}. 
\end{equation*}
Simce the Galois covering is compatible with $\s_0$ and $q_0$ thanks to the diagram (2.7.2) 
(it corresponds to the squares in the bottom row), we have 
$\s_0^*(\CL_{\r_1}\boxtimes \CL_{\r'}) \simeq q_0^*\CL_{\r}$ under 
the identification $\r = \r_1\boxtimes \r'$ for $\CW_{\Bm,\CE} = \CW_1 \times \CW'_{\Bm',\CE_2}$. 
This implies that $A_{\r}|_{\CY^0_{\Bm}}\simeq  \CL_{\r}[d_{\Bm}]$.
Hence (2.7.1) holds and the lemma follows. 
\end{proof}

By using Lemma 2.7, we show the following.

\begin{prop}  %%%%  Prop. 2.8
Under the notation in Lemma 2.7, $(\pi_{\Bm})_!\a^*\CE$ is decomposed as 
\begin{equation*}
(\pi_{\Bm})_!\a^*\CE \simeq H^{\bullet}(\BP_1^{m_r})\otimes\bigoplus_{\r \in \CW_{\Bm,\CE}\wg}
                     \wt V_{\r}\otimes \IC(\CX_{\Bm}^0,\CL_{\r}),
\end{equation*}
where $\wt V_{\rho}$ is regarded as a vector space, ignoring the $\wt\CW_{\CE}$-action.
\end{prop}

\begin{proof}
We prove the proposition by induction on $r$.  
In the case where $r = 2$, the proposition holds by Proposition 4.8 in [SS1].
We assume that the proposition holds for $r' < r$. 
We fix $\BI \in \CI(\Bm)$. 
Put $\wt\CZ_{\BI} = \wt q\iv(\wt\CX_{\BI})$.
We have the following commutative diagram

\begin{equation*}
\tag{2.8.1}
\begin{CD}
\wt G^0_1 \times \wt\CX'_{\BI'}  @<<< \wt\CZ_{\BI}  
      @>>>  \wt\CX_{\BI}  \\  
  @V\vf^1\times \pi_{\BI'} VV    @VVV    @VV\pi_{\BI} V   \\
G^0_1\times \CX'^0_{\Bm'}   @<\s<<  \CK_{\Bm}  @>q>>  \CX_{\Bm}^0, 
\end{CD}
\end{equation*}
where $\pi_{\BI}$ is the restriction of $\pi_{\Bm}$ on $\wt\CX_{\BI}$, and similarly
for $\pi_{\BI'}$. Other maps are determined correspondingly. 
We note that both squares are cartesian squares.  
\par
We show the following.
\par\medskip\noindent
(2.8.2) \ Any simple summand (up to shift) of the semisimple complex 
$(\pi_{\BI})!\a^*\CE$ is contained in the set 
$\{ A_{\r} \mid \r \in \CW_{\Bm, \CE}\wg\}$. 
\par\medskip
Put $K_1 = \vf^1_!(\a^1)^*\CE_1$, and 
$K_2 = (\pi_{\Bm'})_!(\a')^*\CE_2$.
By (2.6.1) and induction hypothesis, we have 
\begin{align*}
K_1 &\simeq \bigoplus_{\r_1 \in \CW_1\wg}\r_1\otimes 
        \IC(G^0_1, \CL_{\r_1})  \\
K_2 &\simeq H^{\bullet}(\BP_1^{m_r})\otimes
      \bigoplus_{\r' \in \CW_{\Bm',\CE_2}\wg}\wt V_{\r'}\otimes 
               \IC(\CX'^0_{\Bm'}, \CL_{\r'}).
\end{align*}
Put $K_{\BI'} = (\pi_{\BI'})_!(\a')^*\CE$.
Since $\wt\CX'_{\BI'}$ is a connected component of $\wt\CX'^+_{\Bm'}$, 
any simple summand of $K_{\BI'}$ is contained in 
$K_2$, up to shift, thus it is of the form 
$\IC(\CX'^0_{\Bm'}, \CL_{\r'})$. 
A simple perverse sheaf on $\CX_{\Bm}^0$ obtained from 
$\IC(G_1^0, \CL_{\r_1})$ and $\IC(\CX'^0_{\Bm'}, \CL_{\r'})$ by the 
procedure in 2.6 actually coincides with $A_{\r}$.  
On the other hand, since the squares in the diagram (2.8.1) are both cartesian, 
we have $\s_1^*(K_1\boxtimes K_{\BI'}) \simeq q_1^*((\pi_{\BI})_!\a^*\CE)$.
(2.8.2) follows from this.
\par
By Lemma 2.4 and Lemma 2.7, (2.8.2) implies that 
\par\medskip\noindent
(2.8.3) \ Any simple summand (up to shift) of the semisimple complex 
$(\pi_{\Bm})_!\a^*\CE$ is contained in the set $\{ \IC(\CX^0_{\Bm},\CL_{\r})
     \mid \r \in \CW_{\Bm, \CE}\wg \}$.   
\par\medskip

(2.8.3) implies, in particular, that any simple summand of $K = (\pi_{\Bm})_!\a^*\CE$
has its support $\CX_{\Bm}^0$. Since the restriction of $K$ on $\CY_{\Bm}^0$ 
coincides with $K_0 = (\psi_{\Bm})_!\a_0^*\CE$, the decompostion of $K$ into 
simple summands is determined by the decomposition of $K_0$.  Hence 
the proposition follows from (1.6.5).  
\end{proof}

\remark{2.9.} 
Proposition 2.8 is a generalization of Proposition 4.8 in [SS1]. But the argument here
is much simpler than that of [SS1].  Note that $A_{\r}^0$ in [SS1, Lemma 4.7] corresponds 
to $A_{\r}$ in Lemma 2.7. 

\para{2.10.}
For $\Bm \in \CQ_{n,r}$, and for each $j, k$, 
we consider $M^{(j,k)}$ and $\ol M^{(j,k)}$ as in the proof of Proposition 1.7. 
Put
\begin{align*}
\CX^0_{j,k} &= \bigcup_{g \in H}g(B^{\io\th} \times M^{(j,k)}), \\
\wt\CX_{j,k}^+ &= \pi\iv(\CX^0_{j,k}),
\end{align*}
and let $\pi_{j,k}: \wt\CX_{j,k}^+ \to \CX_{j,k}^0$ be the restriction of $\pi$ 
on $\wt\CX^+_{j,k}$. 
Then $\pi_{j,k}$ is a proper map, and $\CY_{j,k}^0$ is open dense in 
$\CX^0_{j,k}$. 
Moreover 
$\CX_{j,k}^0$ coincides with $\CX^0_{\Bm}$ in the case where $j = r-1, k = m_{r-1}$,
coincides with
$\CX_{\Bm}$ in the case where $j = 0$. Also put 
\begin{align*}
\CX_{j,k} &= \bigcup_{g \in H}g(B^{\io\th} \times \ol M^{(j,k)}), \\
\ol{\wt\CX_{j,k}^+} &= \pi\iv(\CX_{j,k}),
\end{align*}
and let $\ol\pi_{j,k}: \ol{\wt\CX_{j,k}^+} \to \CX_{j,k}$ be the restriction of $\pi$ 
on $\ol{\wt\CX^+_{j,k}}$.
Then $\CX_{j,k}^0$ is open dense in $\CX_{j,k}$.  As in (1.7.1), we have
\par\medskip\noindent
(2.10.1) \  
$\CX_{j,k} \backslash \CX_{j,k-1} = \CX_{j,k}^0$ if $k \ge 1$, and
$\CX_{j,0}(\Bm) = \CX_{j+1, m_j + m_{j+1}}(\Bm(j,0))$.  
Moreover, $\CX^0_{j,k}(\Bm)$ coincides with $\CX_{j+1, m_{j+1}}(\Bm(j,k))$.
\par\medskip
For $\Bm' \in \CQ(\Bm;j,k)$, $\CY^0_{\Bm'}$ is containd in $\CY_{j,k}$ hence in 
$\CX_{j,k}$.
One can define an intersection cohomology $\IC(\CX'_{\Bm'}, \CL_{\r})$
assoicated to the local sysmte $\CL_{\r}$ on $\CY^0_{\Bm'}$
(here $\CX'_{\Bm'}$ denotes the closure of $\CY^0_{\Bm'}$ in $\CX_{j,k}$). 
We show the following formulas.  
First assume that $j = r-1$ and $0 \le k \le m_{r-1}$.
\begin{equation*}
\tag{2.10.2}
\begin{split}
(\ol\pi_{r-1,k}&)_!\a^*\CE \\
     &\simeq 
     \bigoplus_{0 \le k' \le k}\bigoplus_{\rho \in \CW_{\Bm(k'),\CE}\wg}
    \wt V_{\rho} \otimes \IC(\CX_{r-1,k'}, \CL_{\rho})[-2(m_{r-1}-k')] + \CM_{r-1,k}, 
\end{split}  
\end{equation*}
where $\CM_{r-1, k}$ is a sum of various $\IC(\CX_{r-1, k'}, \CL_{\r})[-2i]$ 
for $0 \le k' \le k$ and $\r \in \CW_{\Bm(k'), \CE}\wg$ with $0 \le i < m_{r-1}- k'$.  
Next assume that $0 \le j < r-1$ and that $0 \le k \le m_j$.   Then we have
\begin{equation*}
\tag{2.10.3}
\begin{split}
(\ol\pi_{j,k}&)_!\a^*\CE  \\
&\simeq 
     \bigoplus_{0 \le k' \le m_{r-1}}\bigoplus_{\rho \in \CW_{\Bm(k'),\CE}\wg}
     \wt V_{\rho} \otimes \IC(\CX'_{\Bm(k')}, \CL_{\rho})[-2(m_{r-1}-k')]  + \CM_{j,k},  
  \end{split}  
\end{equation*}
where $\CM_{j,k}$ is a sum of various $\IC(\CX'_{\Bm'}, \CL_{\r})[-2i]$ 
for $\Bm' \in \CQ(\Bm;j,k)$ and $\r \in \CW_{\Bm', \CE}\wg$ with $i$ such that
$0 \le 2i < d_{\Bm} - d_{\Bm'}$.
\par
As in the proof of Proposition 1.7, one can define an action of $\wt\CW_{\CE}$ on 
$(\ol\pi_{j,k})_!\a^*\CE$. 
Then (2.10.2) and (2.10.3) can be proved by a similar argument as in the proof of 
Proposition 1.7. 
\par
Now apply (2.10.3) to the case where $j = 0, k = 0$. In this case, $\ol\pi_{j,k}$
coincides with $\ol\pi_{\Bm}$. (2.10.3) shows that any simple perverse sheaf 
$A[s]$ appearing in the semisimple complex $(\ol\pi_{\Bm})_!\a^*\CE$ has the property 
that $\supp A \cap \CY_{\Bm} \neq \emptyset$. 
$\CY_{\Bm}$ is open dense in $\CX_{\Bm}$, and the restriction of 
$(\ol\pi_{\Bm})_!\a^*\CE$ on $\CY_{\Bm}$ coincides with $(\ol\psi_{\Bm})_!\a_0^*\CE$.
Thus the theorem  follows from Proposition 1.7.   This complets the proof 
of Theorem 2.2.

\par\bigskip\bigskip
\section{A variant of Theorem 2.2}

\para{3.1.}
In this section, we assume that $\CX$ is of exotic type.  We keep the 
notation in Section 1 and Section 2. 
For $\Bm = (m_1, \dots, m_{r-1}, 0) \in \CQ_{n,r}^0$, 
put $\CW_{\Bm}\nat = S_{m_1} \times \cdots \times S_{m_{r-2}} \times 
W_{m_{r-1}}$, where $W_n$ is the Weyl group of type $C_n$, and let 
$\CW_{\Bm,\CE}\nat$ be the stabilizer of $\CE$ in $\CW_{\Bm}\nat$.  
(Note that $\CW_{\Bm}\nat$ is not a subgroup of $W_{n,r}$ if $r \ge 3$.)
Recall that $\Bm(k) = (m_1, \dots, m_{r-2}, k, k')$ with 
$k + k' = m_{r-1}$ for $\Bm \in \CQ^0_{n,r}$.  Hence 
$\CW_{\Bm(k)} \simeq S_{m_1} \times \cdots \times S_{m_{r-2}}
  \times S_k \times S_{k'}$. 
For $\r = \r_1 \boxtimes \cdots \boxtimes \r_r \in \CW_{\Bm(k),\CE}\wg$, 
we define an irreducible 
$\CW_{\Bm,\CE}\nat$-module $V_{\r}\nat$ by  
$V_{\r}\nat = \r_1\boxtimes\cdots\boxtimes \r_{r-2} \boxtimes \wt\r_{r-1}$, 
where $\wt\r_{r-1}$ is an irreducible $W_{m_{r-1}}$-module obtained from 
$\r_{r-1}\boxtimes \r_r \in (S_k \times S_{k'})\wg$ (apply 1.6 for the case 
$r = 2$).  
Recall the map $\pi^{(\Bm)}: \wt\CX_{\Bm} \to \CX_{\Bm}$ as in 1.2 and 
consider 
$\a|_{\wt\CX_{\Bm}} : \wt\CX_{\Bm} \to T^{\io\th}$, which we denote by the 
same symbol $\a$.
The following result is a variant of Theorem 2.2.

\begin{thm}  %%%  Theorem 3.2
For each $\Bm \in \CQ_{n,r}^0$, $\pi^{(\Bm)}_!\a^*\CE[d_{\Bm}]$ is a semisimple 
perverse sheaf on $\CX_{\Bm}$ equipped with $\CW_{\Bm,\CE}\nat$-action, and is 
decomposed as 
\begin{equation*}
\pi^{(\Bm)}_!\a^*\CE[d_{\Bm}] \simeq \bigoplus_{0 \le k \le m_{r-1}}
        \bigoplus_{\r \in \CW_{\Bm(k),\CE}\wg} V_{\r}\nat\otimes 
               \IC(\CX_{\Bm(k)},\CL_{\r})[d_{\Bm(k)}].
\end{equation*}

\end{thm}

\para{3.3.}
The theorem  can be proved in a similar way as in the proof of Theorem 2.2.
We will give an outline of the proof below.
We follow the notation in 1.3.
We fix $\Bm \in \CQ_{n,r}^0$, and put $\pi^{(\Bm)} = \pi^{\bullet}$. 
We also consider the map 
$\psi^{(\Bm)}: \wt\CY_{\Bm} \to \CY_{\Bm}$, which we denote by $\psi^{\bullet}$.
For each $\Bm' \le \Bm$ ($\Bm' \in \CQ_{n,r}$), put 
$\wt\CY^{\dag}_{\Bm'} = (\psi^{\bullet})\iv(\CY_{\Bm'}^0)$.
For each $\BI \in \CI(\Bm')$, the variety $\wt\CY_{\BI} \subset \wt\CY^{\dag}_{\Bm'}$ 
is defined as in 1.3.
Put $\BI^{\bullet} = \BI(\Bm) \in \CI(\Bm)$, namely, 
$\BI^{\bullet} = (I_1^{\circ}, \dots, I_r^{\circ})$ with 
$I_i^{\circ} = [p_{i-1}+1, p_i]$. 
In particular, $I_r^{\circ} = \emptyset$.
Put $\CW^{\bullet}= S_{m_1} \times \cdots \times S_{m_{r-1}} (= \CW_{\Bm})$, 
and we denote by 
$\CW^{\bullet}_{\Bm'}$ the subgroup of $\CW^{\bullet}$ which is the stabilizer 
of $\BI(\Bm')$. 
Put 
\begin{equation*}
\CI^{\bullet}(\Bm') = \{ \BI \in \CI(\Bm') \mid I_{\le i} \subset I_{\le i}^{\circ} 
 \  (1 \le i \le r)\},
\end{equation*}
where $I_{\le i}$ is defined similarly to $I_{< i}$ in 1.3.
Then as in (1.3.2), $\CW^{\bullet}$ acts naturally on $\wt\CY^{\dag}_{\Bm'}$, and 
\begin{equation*}
\tag{3.3.1}
\wt\CY_{\Bm'}^{\dag} = \coprod_{\BI \in \CI^{\bullet}(\Bm')}\wt\CY_{\BI}
                        = \coprod_{w \in \CW^{\bullet}/\CW_{\Bm'}^{\bullet}}w(\wt\CY^0_{\Bm'}),
\end{equation*}
where $\CY^0_{\Bm'} = \CY_{\BI(\Bm')}$ is as in 1.3.
For a tame local system $\CE$ on $T^{\io\th}$, we denote by $\CW^{\bullet}_{\CE}$
(resp. $\CW^{\bullet}_{\Bm',\CE}$) the stabilizer of $\CE$ in $\CW^{\bullet}$ 
(resp. in $\CW^{\bullet}_{\Bm'}$). 
Put $\wt\CW^{\bullet} = \CW^{\bullet} \ltimes (\BZ/2\BZ)^{m_{r-1}}$, which coincides with 
$\CW_{\Bm}\nat$.
We define a subgroup $\wt\CW^{\bullet}_{\CE}$ 
of $\wt\CW^{\bullet}$ by 
$\wt\CW^{\bullet}_{\CE} = \CW^{\bullet}_{\CE} \ltimes (\BZ/2\BZ)^{m_{r-1}}$. 
Thus $\wt\CW^{\bullet}_{\CE} = \CW\nat_{\Bm, \CE}$ in the notation of 3.1.
Let $\psi^{\bullet}_{\Bm'}$ be the restriction of $\psi^{\bullet}$ on $\wt\CY^{\dag}_{\Bm'}$.
As an analogue of (1.5.2), we have
\begin{equation*}
\tag{3.3.2}
(\psi^{\bullet}_{\Bm'})_!\a_0^*\CE|_{\wt\CY^{\dag}_{\Bm'}}
        \simeq \bigoplus_{\BI \in \CI^{\bullet}(\Bm')}(\psi_{\BI})_!\a_0^*\CE|_{\wt\CY_{\BI}}.
\end{equation*}
Put $\Bm' = (m_1', \dots, m_r')$.
The action of $(\BZ/2\BZ)^{m_r'}$ on $H^{\bullet}(\BP_1^{m_r'})$ 
is defined as in 1.6 by considering the case 
where $r = 2$.  
By a similar argument as in 1.5 and 1.6, we see that $(\psi^{\bullet}_{\Bm'})_!\a_0^*\CE$
is equipped with $(\BZ/2\BZ)^{m_r'} \times \CW^{\bullet}_{\CE}$-action, and is decomposed as 
\begin{equation*}
\tag{3.3.3}
(\psi^{\bullet}_{\Bm'})_!\a_0^*\CE|_{\wt\CY^{\dag}_{\Bm'}} \simeq 
          H^{\bullet}(\BP_1^{m_r'})\otimes 
      \bigoplus_{\r \in (\CW^{\bullet}_{\Bm',\CE})\wg}
     \Ind_{\CW^{\bullet}_{\Bm',\CE}}^{\CW^{\bullet}_{\CE}}\r \otimes 
          \CL_{\r},
\end{equation*}
where $\CL_{\r}$ is a simple local system on $\CY^0_{\Bm'}$ 
obtained from the Galois covering $\wh \CY_{\BI(\Bm')} \to \CY_{\Bm'}^0$ 
as in (1.5.3). 
\par
As in (1.6.6), (3.3.3) can be rewritten in the following form;
\begin{equation*}
\tag{3.3.4}
(\psi^{\bullet}_{\Bm'})_!\a^*_0\CE|_{\wt\CY^{\dag}_{\Bm'}} 
      \simeq \biggl(\bigoplus_{\r \in \CW_{\Bm(k), \CE}\wg}
            V_{\r}\nat \otimes \CL_{\r}\biggr)[-2m_r']  + \CN_{\Bm'} 
\end{equation*}
if $\Bm' = \Bm(k)$ for some $k$, and 
$(\psi^{\bullet}_{\Bm'})_!\a^*\CE|_{\wt\CY^{\dag}_{\Bm'}} \simeq \CN_{\Bm'}$ 
otherwise, where $\CN_{\Bm'}$ is a sum of various $\CL_{\r}[-2i]$ for 
$\r \in (\CW^{\bullet}_{\Bm', \CE})\wg$ with $2i < d_{\Bm} - d_{\Bm'}$.
\par
In fact, put $\Bp(\Bm') = (p_1', \dots, p_r')$.  Since $\Bm \ge \Bm'$, 
$p_i \ge p_i'$ for each $i$.  Moreover $p_{r-1} - p_{r-1}' = m_r'$ 
since $p_{r-1} = n$.
Then by Lemma 1.4 we have
\begin{align*}
d_{\Bm} - d_{\Bm'} &= \sum_{i=1}^{r-1}(p_i - p_i') + m_r' \\
                   &= \sum_{i=1}^{r-2}(p_i - p_i') + 2m_r' \\
                   &\ge 2m_r',
\end{align*}
and the equality holds only when $p_i = p_i'$ for $i = 1, \dots, r-2$, 
namely when $\Bm' = \Bm(k)$ for some $k$. 
By (3.3.3), $K = (\psi^{\bullet}_{\Bm'})_!\a_0^*\CE$ is a semisimple 
complex and each direct summand is of the form $\CL_{\r}[-2i]$ with 
$i \le m_r'$.  Hence $K \simeq \CN_{\Bm'}$ if $\Bm'$ is not of the form 
$\Bm(k)$. 
Now assume that $\Bm' = \Bm(k)$.
In this case, $\CW^{\bullet}_{\Bm', \CE} = \CW_{\Bm(k),\CE}$, and 
it follows from (3.3.3) that 
\begin{equation*}
\tag{3.3.5}
K \simeq \bigoplus_{\r \in \CW_{\Bm(k), \CE}\wg}
\Ind_{\wt\CW^{\bullet}_{\Bm(k), \CE}}^{\wt\CW^{\bullet}_{\CE}}
          (H^{\bullet}(\BP_1^{m_r'})\otimes \r)\otimes \CL_{\r},
\end{equation*}
where $\wt\CW^{\bullet}_{\Bm(k), \CE} = 
   \CW_{\Bm(k),\CE}^{\bullet} \ltimes (\BZ/2\BZ)^{m_{r-1}}$,
and $H^{\bullet}(\BP_1^{m_r'})\otimes \r$ is regarded as an $\wt\CW^{\bullet}_{\Bm(k),\CE}$-module
by the trivial action of $(\BZ/2\BZ)^k$, and by the action of $(\BZ/2\BZ)^{m_r'}$ through 
$H^{\bullet}(\BP_1^{m_r'})$ (here $k + m_r' = m_{r-1}$).  
The direct summand $\CL_{\r}[-2i]$ of $K$ satisfies the relation 
$d_{\Bm} - d_{\Bm'} = 2i$ only when $i = m'_{r-1}$. 
Hence the first assertion of (3.3.4) follows from (3.3.5). 

\para{3.4.}
For $1 \le j \le r-1$ and $0 \le k \le m_j$, 
$M^{(j,k)}, \ol M^{(j,k)}, \CY_{j,k}^0, \CY_{j,k}$ are defined as in 
the proof of Proposition 1.7.
Put $\wt\CY^{\dag}_{j,k} = (\psi^{\bullet})\iv(\CY^0_{j,k})$
and $\ol{\wt\CY^{\dag}_{j,k}} = (\psi^{\bullet})\iv(\CY_{j,k})$.
Let $\psi^{\bullet}_{j,k} : \wt\CY^{\dag}_{j,k} \to \CY^0_{j,k}$ 
be the restriction of $\psi^{\bullet}$ on $\wt\CY^{\dag}_{j,k}$, 
and $\ol\psi^{\bullet}_{j,k} : \ol{\wt\CY^{\dag}_{j,k}} \to \CY_{j,k}$
the restriction of $\psi^{\bullet}$ on $\ol{\wt\CY^{\dag}_{j,k}}$.
By using a similar argument as in the proof of (1.7.2) and (1.7.3), 
we can show the following formulas; first assume that $j = r-1$ and 
$0 \le k \le m_{r-1}$. 
\begin{equation*}
\tag{3.4.1}
\begin{split}
(&\ol\psi^{\bullet}_{r-1,k})_!\a_0^*\CE \\
    &\simeq \bigoplus_{0 \le k' \le k}\bigoplus_{\r \in \CW_{\Bm(k'), \CE}\wg}
        V_{\r}\nat\otimes \IC(\CY_{r-1, k'}, \CL_{\r})[-2(m_{r-1} - k')]
           + \CN_{r-1, k},
\end{split}
\end{equation*}   
where $\CN_{r-1, k}$ is a sum of various $\IC(\CY_{r-1,k'}, \CL_{\r})[-2i]$
for $0 \le k' \le k$ and $\r \in \CW_{\Bm(k'),\CE}\wg$ with 
$i < m_{r-1} - k'$.  
Next assume that $0 \le j < r-1$ and that $0 \le k \le m_j$. Then we have
\begin{equation*}
\tag{3.4.2}
\begin{split}
(&\ol\psi^{\bullet}_{j,k})_*\a_0^*\CE \\
  &\simeq 
    \bigoplus_{0 \le k' \le m_{r-1}}\bigoplus_{\rho \in \CW_{\Bm(k'),\CE}\wg}
         V\nat_{\rho} \otimes \IC(\CY'_{\Bm(k')}, \CL_{\rho})[-2(m_{r-1} - k')]
             + \CN_{j,k},
\end{split}
\end{equation*}
where $\CN_{j,k}$ is a sum of various $\IC(\CY'_{\Bm'}, \CL_{\r})[-2i]$
for $\Bm' \in \CQ(\Bm;j,k)$ 
and $\r \in (\CW^{\bullet}_{\Bm',\CE})\wg$ with $i$ such that 
$2i < d_{\Bm} - d_{\Bm'}$.
\par
Note that in the proof of (3.4.1), the role of the irreducible $\wt\CW_{\CE}$-module
$\wt V_{\r}$ is replaced by the irreducible $\CW_{\Bm,\CE}\nat$-module $V_{\r}\nat$.
\par
By a similar argument as in the proof of Lemma 1.4 (iv), one can show that 
$\psi^{(\Bm)}$ is semismall for $\Bm \in \CQ_{n,r}^0$. 
Then, as in the proof of Proposition 1.7 (see the paragraph after (1.7.3)), 
we obtain the following proposition from (3.4.2).  

\begin{prop}  %%%%  Prop. 3.5 
For each $\Bm \in \CQ_{n,r}^0$, $\psi^{(\Bm)}_!\a_0^*\CE[d_{\Bm}]$ is a semisimple 
perverse sheaf on $\CY_{\Bm}$ equipped with $\CW\nat_{\Bm ,\CE}$-action, and is decomposed as 
\begin{equation*}
\psi^{(\Bm)}_!\a_0^*\CE[d_{\Bm}] \simeq \bigoplus_{0 \le k \le m_{r-1}}
                 \bigoplus_{\r \in \CW_{\Bm(k), \CE}\wg}
                   V_{\r}\nat\otimes \IC(\CY_{\Bm(k)}, \CL_{\r})[d_{\Bm(k)}].
\end{equation*} 
\end{prop}

\para{3.6.}
We follow the notation in 2.3. 
For $\Bm'\in \CQ_{n,r}$ such that $\Bm' \le \Bm$, 
put $\wt\CX^{\dag}_{\Bm'} = (\pi^{\bullet})\iv(\CX_{\Bm'}^0)$.
Then $\wt\CY_{\Bm'}^{\dag}$ is an open dense subset of $\wt\CX^{\dag}_{\Bm'}$. 
For each $\BI \in \CI(\Bm')$, the subvariety $\wt\CX_{\BI}$ of $\wt\CX_{\Bm'}^+$ 
is defined as in 2.3.   Then as in Lemma 2.4, we have 
\begin{equation*}
\wt\CX^{\dag}_{\Bm'} = \coprod_{\BI \in \CI^{\bullet}(\Bm')}\wt\CX_{\BI},
\end{equation*}
where $\wt\CX_{\BI}$ is an irreducible component of $\wt\CX^{\dag}_{\Bm'}$. 
\par
Let $\pi_{\Bm'}^{\bullet} : \wt\CX^{\dag}_{\Bm'} \to \CX^0_{\Bm'}$ be the restriction 
of $\pi^{(\Bm)}$ on $\wt\CX^{\dag}_{\Bm'}$. 
The following result is an analogue of Proposition 2.8, and is proved in a similar way.

\begin{prop}  %%%%  Prop. 3.7.
Assume that $\Bm' \le \Bm$.  Then
$(\pi^{\bullet}_{\Bm'})_!\a^*\CE$ is decomposed as 
\begin{equation*}
(\pi_{\Bm'}^{\bullet})_!\a^*\CE \simeq H^{\bullet}(\BP_1^{m'_r}) \otimes 
          \bigoplus_{\r \in \CW^{\bullet \wg}_{\Bm', \CE}}
    \Ind_{\CW^{\bullet}_{\Bm',\CE}}^{\CW^{\bullet}_{\CE}}\r \otimes \IC(\CX_{\Bm'}^0, \CL_{\r}),
\end{equation*}
where $\Ind_{\CW^{\bullet}_{\Bm',\CE}}^{\CW^{\bullet}_{\CE}}\r$ is regarded as a vector space
ignoring the $\CW^{\bullet}_{\CE}$-module structure. 
\end{prop}

\para{3.8.}
By making use of Proposition 3.5 and 3.7, the theorem can be proved in a similar way 
as in 2.9, 2.10. 

\par\bigskip\bigskip

\section{Intersetion cohomology on $G^{\io\th} \times V^{r-1}$ (enhanced case)}

\para{4.1.}
In this section we assume that $\CX$ is of enhanced type. 
We fix $\Bm \in \CQ_{n,r}$ (note that we don't assume $\Bm \in \CQ_{n,r}^0$), 
and consider the map 
$\pi^{(\Bm)} : \wt\CX_{\Bm} \to \CX_{\Bm}$.  Here in order to emphasize 
a similarlty with the exotic case, we follow the notation in Section 1.  
But of course a simpler expression is possible for the enhanced case. 
For example, if we write $G = G_0 \times G_0$ and 
$B = B_0 \times B_0$ for a Borel subgroup 
$B_0$ of $G_0$, $\wt\CX_{\Bm}, \CX_{\Bm}$ are given by 

\begin{align*}
\wt\CX_{\Bm} &= \{ (x, \Bv, gB_0) \in G_0 \times V^{r-1} \times G_0/B_0
                 \mid g\iv xg \in B_0, g\iv \Bv \in \prod_{i=1}^{r-1}M_{p_i} \} \\
\CX_{\Bm} &=  \bigcup_{g \in G_0}g(B_0 \times \prod_{i=1}^{r-1}M_{p_i}).
\end{align*} 
The map $\psi^{(\Bm)} : \wt\CY_{\Bm} \to \CY_{\Bm}$ is defined as in 1.2.
We put $\pi^{(\Bm)} = \pi^{\bullet}$ and $\psi^{(\Bm)} = \psi^{\bullet}$.
The subset $\CY_{\Bm'}^0$ is defined for each $\Bm'\in \CQ_{n,r}$ as in 1.3.
For each $\Bm' \le \Bm$, put $\wt\CY^{\dag}_{\Bm'} = (\psi^{\bullet})\iv(\CY^0_{\Bm'})$.
For each $\BI \in \CI(\Bm')$, the subvariety $\wt\CY_{\BI}$ of $\wt\CY_{\Bm'}^{\dag}$ 
and the map $\psi_{\BI}$ are 
defined as in 1.3. 
Note that in the enhanced case, if we write $T = T_0 \times T_0$, then 
$Z_H(T^{\io\th}) = Z_{G_0}(T_0) = T_0$.  Hence $B^{\th} \cap Z_H(T^{\io\th}) = T_0$, and 
$T^{\io\th}\reg$ is the set of regular semisimple elements in $T_0$.  
Hence $\wt\CY_{\BI}$ is written as  
\begin{equation*}
\wt\CY_{\BI} \simeq G_0 \times^{T_0}((T_0)\reg \times M_{\BI}),  
\end{equation*}
where 
$M_{\BI}$ is defined as in 1.3. 
As in 3.3, we define $\CW^{\bullet} = S_{m_1} \times \cdots \times S_{m_r}$ 
$(= \CW_{\Bm})$, and its subgroup $\CW_{\Bm'}^{\bullet}$.  For each $\Bm' \le \Bm$, 
we define $\CI^{\bullet}(\Bm')$ as in 3.3.   
Then a simlar formula as (3.3.1) holds for $\wt\CY_{\Bm'}^{\dag}$. 
Let $\CE$ be a tame local system on $T^{\io\th}$, and we denote by 
$\CW^{\bullet}_{\CE}$ (resp. $\CW^{\bullet}_{\Bm', \CE}$) the stabilizer 
of $\CE$ in $\CW^{\bullet}$ (resp. in $\CW^{\bullet}_{\Bm'}$). 
As in (1.5.2), we have a similar formula as (3.3.2).
Note that in the enhanced case, one can check that $\psi_{\BI}$ is a finite Galois covering with group 
$\CW^{\bullet}_{\BI}$ ( the stbilizer of $\BI$ in $\CW^{\bullet}$).
It follws from (3.3.1) and (3.3.2) (corresponding formulas for the enhanced case)
we have

\begin{equation*}
\tag{4.1.1}
(\psi^{\bullet}_{\Bm'})_!\a_0^*\CE|_{\wt\CY^{\dag}_{\Bm'}}
    \simeq \bigoplus_{\r \in (\CW^{\bullet}_{\Bm', \CE})\wg}
       \bigl(\Ind_{\CW^{\bullet}_{\Bm',\CE}}^{\CW^{\bullet}_{\CE}}\r \bigr)\otimes \CL_{\r}
\end{equation*}

Here we note that

\begin{lem}  %%%%  Lemma 4.2.
\begin{enumerate}
\item
$\CY_{\Bm}$ is open dense in $\CX_{\Bm}$ and $\wt\CY_{\Bm}$ is open dense in 
$\wt\CX_{\Bm}$.
\item
$\dim \CX_{\Bm} = \dim \wt\CX_{\Bm} = n^2 + \sum_{i=1}^r(r-i)m_i$.
\item
For any $(x, \Bv) \in \CY_{\Bm}$, $(\psi^{(\Bm)})\iv(x, \Bv)$ is a finite set.
\end{enumerate}
\end{lem} 

\begin{proof}
$\wt\CY_{\Bm}$ is an open dense subset of $\wt\CX_{\Bm}$.  Since
$\wt\CY_{\Bm} = (\pi^{(\Bm)})\iv(\CY_{\Bm})$ and $\pi^{(\Bm)}$ is proper, 
$\CY_{\Bm}$ is an open dense subset of $\CX_{\Bm}$.  Hence (i) holds.
$\wt\CY^{\dag}_{\Bm}$ is an open dense subset of $\wt\CY_{\Bm}$, and 
$\CY_{\Bm}^0$ is an open dense subset of $\CY_{\Bm}$.  Since $\psi_{\BI}$ 
is a finite Galois covering for $\BI \in \CI(\Bm)$, 
we have $\dim \wt\CY_{\Bm}^{\dag} = \dim \CY_{\Bm}^0$.  
Hence $\dim \wt\CY_{\Bm} = \dim \CY_{\Bm}$ and (ii) follows from (1.2.1).
(iii) is clear since $\psi_{\BI}$ is a finite Galois covering for any 
$\BI \in \CI^{\bullet}(\Bm')$. 
\end{proof}

Next we show the following proposition.

\begin{prop}  %%%%% Prop. 4.3.
For each $\Bm \in \CQ_{n,r}$, $\psi^{(\Bm)}_!\a_0^*\CE[d_{\Bm}]$  is a 
semisimple perverse sheaf on $\CY_{\Bm}$ equipped with $\CW^{\bullet}_{\Bm,\CE}$-action,
and is decomposed as 
\begin{equation*}
\psi^{(\Bm)}_!\a_0^*\CE[d_{\Bm}] \simeq
              \bigoplus_{\r \in \CW_{\Bm, \CE}\wg}\r \otimes 
                 \IC(\CY_{\Bm}, \CL_{\r})[d_{\Bm}].
\end{equation*}
\end{prop}

\begin{proof}
$\psi^{(\Bm)}: \wt\CY^{(\Bm)} \to \CY_{\Bm}$ is proper.  By Lemma 4.2 (iii), 
$\psi^{(\Bm)}$ is semismall.  Hence $\psi^{(\Bm)}_!\a_0^*\CE[d_{\Bm}]$ is 
a semisimple perverse sheaf. 
The definitions of $\CY^0_{j,k}, \CY_{j,k}, \psi^{\bullet}_{j,k}, \ol\psi^{\bullet}_{j,k}$, 
etc. in 3.4 make sense
also in the enhanced case.
As in (3.4.1) and (3.4.2), the following formulas hold;
first assume that $j = r-1, 0 \le k < m_{r-1}$.  Then we have
\begin{equation*}
\tag{4.3.1}
(\ol\psi^{\bullet}_{j,k})_!\a_0^*\CE \simeq \CN_{r-1, k}.
\end{equation*}
Next assume that $j = r-1, k = m_{r-1}$ or $0 \le j < r-1, 0 \le k \le m_j$. 
Then we have
\begin{equation*}
\tag{4.3.2}
(\ol\psi^{\bullet}_{j,k})_!\a_0^*\CE
      \simeq \bigoplus_{\r \in \CW_{\Bm, \CE}\wg}\r \otimes \IC(\CY'_{\Bm}, \CL_{\r})
               + \CN_{j, k},
\end{equation*}
where $\CN_{j,k}$ is a sum of various $\IC(\CY'_{\Bm'}, \CL_{\r})$ for 
$\Bm' \in \CQ(\Bm;j,k)$ and $\r \in (\CW^{\bullet}_{\Bm',\CE})\wg$ such that 
$\Bm' < \Bm$.   (Recall that $\CY'_{\Bm'}$ denotes the closure of 
$\CY^0_{\Bm'} \subset \CY_{j,k}$ in $\CY_{j,k}$ for any $\Bm' \le \Bm$.) 
\par
We show (4.3.1) by induction on $k$. By (1.7.1) (or directly from the definition of
$M^{(r-1,k)}$), $\CY_{r-1, 0}(\Bm)$ coincides with $\CY^0_{\Bm(0)}$.  Hence 
by (4.1.1) for $\Bm' = \Bm(0)$, (4.3.1) holds for $k = 0$.  A similar 
argument as in the proof of (1.7.2) shows, thanks to (1.7.1) and (4.1.1),  that 
(4.3.1) holds for $k < m_{r-1}$.
Next consider the case where $j = r-1, k = m_{r-1}$. 
By (1.7.1), $\CY_{j, k} \backslash \CY_{j, k-1} = \CY_{j, k}^0$, and 
$\CY^0_{j, k}$ coinicides with $\CY_{\Bm}^0$. Thus (4.1.1) implies that 
\begin{equation*}
(\psi^{\bullet}_{j,k})_!\a_0^*\CE \simeq \bigoplus_{\r \in \CW\wg_{\Bm,\CE}}
  \r \otimes \CL_{\r}. 
\end{equation*}
Hence (4.3.2) holds in this case.  Now (4.3.2) can be proved by induction on $k$ and 
by backwards induction on $j$, starting from $j = r-2, k= 0$, which case corresponds to 
the case where $j = r-1, k = m_{r-1}$ by (1.7.1).  Note that in the enhanced case, 
we don't need a discussion such as (1.7.5). 
\par
Applying (4.3.2) to the case where $j = 0, k = 0$, we obtain the proposition.
In fact, in that case, $\CN_{0.0}$ is a sum of $A = \IC(\CY_{\Bm'}, \CL_{\r})$
such that $\Bm' < \Bm$.  But then $A[d_{\Bm}]$ is not a perverse sheaf.  Since 
$\psi^{(\Bm)}_!\a_0^*\CE[d_{\Bm}]$ is a semisimple perverse sheaf, this implies 
that $\CN_{0.0} = 0$, and the proposition follows.    
\end{proof}

\para{4.4.}
As a special case of (4.1.1) for $\Bm' = \Bm$, we have
\begin{equation*}
(\psi^{\bullet}_{\Bm})_!\a_0^*\CE \simeq \bigoplus_{\r \in \CW_{\Bm, \CE}\wg}
                     \r \otimes \CL_{\r}.
\end{equation*}
Since $\CY_{\Bm}^0$ is a smooth open dense subset of $\CX_{\Bm}$, one can define 
a semisimple pervers sheaf $K_{\Bm, T, \CE}$ on $\CX_{\Bm}$ as 

\begin{equation*}
\tag{4.4.1}
K_{\Bm, T, \CE} = \bigoplus_{\r \in \CW_{\Bm,\CE}\wg}\r \otimes 
             \IC(\CX_{\Bm}, \CL_{\r})[d_{\Bm}].
\end{equation*}
We consider a diagram 

\begin{equation*}
\begin{CD}
T^{\io\th} @<\a <<  \wt\CX_{\Bm} @>\pi^{(\Bm)}>>  \CX_{\Bm},
\end{CD}
\end{equation*}
where $\a$ is as in 2.1, and 
a complex $\pi^{(\Bm)}_!\a^*\CE[d_{\Bm}]$ on $\CX_{\Bm}$. 
We shall prove the following theorem.

\begin{thm}  %%%%  Theorem 4.5
For each $\Bm \in \CQ_{n,r}$, $\pi^{(\Bm)}_!\a^*\CE[d_{\Bm}] \simeq K_{\Bm, T, \CE}$
as perverse sheaves on $\CX_{\Bm}$. 
\end{thm}

\para{4.6.}
For each $\Bm' \le \Bm$, the set $\CX_{\Bm'}^0$ is defined as in 2.3. 
We define a subvariety $\CX^{\dag}_{\Bm'}$ of $\wt\CX_{\Bm}$ by
$\wt\CX^{\dag}_{\Bm'} = (\pi^{\bullet})\iv(\CX_{\Bm'}^0)$. 
Then $\wt\CY^{\dag}_{\Bm'}$ is an open dense subset of $\wt\CX^{\dag}_{\Bm'}$.  
The discussion in 2.3 makes sense also for the enhanced case, and $\CX^0_{\Bm'}$ 
is characterized by a similar formula as (2.3.3).  For each $\BI \in \CI(\Bm')$, 
the set $\wt\CX_{\BI}$ is defined as in 2.3, 
\begin{equation*}
\wt\CX_{\BI} = \{ (x, \Bv, gB^{\th}) \in \wt\CX_{\Bm'}^{\dag} 
                     \mid (x, \Bv, gB^{\th}) \mapsto \BI \}.
\end{equation*} 
Let $\CI^{\bullet}(\Bm')$ be as in 3.3.  
Then as in Lemma 2.4 (see also 3.6), we have 
\begin{equation*}
\wt\CX^{\dag}_{\Bm'} = \coprod_{\BI \in \CI^{\bullet}(\Bm')}\wt \CX_{\BI}, 
\end{equation*}
where $\wt\CX_{\BI}$ is an irreducible component of $\wt\CX^{\dag}_{\Bm'}$.
Let $\pi^{\bullet}_{\Bm'} : \wt\CX^{\bullet}_{\Bm'} \to \CX^0_{\Bm'}$ be 
the restriciton of $\pi^{(\Bm)}$ on $\wt\CX^{\dag}_{\Bm'}$. The following
result is an analogue of Proposition 2.8 (see also Proposition 3.7) and 
can be proved in a similar way.

\begin{prop}  %%%%  Proposition 4.7.
Assume that $\Bm' \le \Bm$.  Then $(\pi^{\bullet}_{\Bm'})_!\a_0^*\CE$
is decomposed as 
\begin{equation*}
(\pi^{\bullet}_{\Bm'})_!\a_0^*\CE|_{\wt\CX^{\dag}_{\Bm'}} \simeq 
         \bigoplus_{\r \in (\CW^{\bullet}_{\Bm',\CE})\wg}
                    \bigl(\Ind_{\CW^{\bullet}_{\Bm',\CE}}^{\CW^{\bullet}_{\CE}}\r\bigr) 
                       \otimes \IC(\CX_{\Bm'}^0, \CL_{\r}). 
\end{equation*}
\end{prop} 

\para{4.8.}
By making use of Proposition 4.3 and Proposition 4.7, the theorem 
is proved by a similar argument as in 2.8 and 2.10.

\par\bigskip\bigskip
\section{Unipotent variety of enhanced type }

In this section, we study the ``unipotent part'' of the enhanced space, whih
we call the unipotent variety of enhanced type. First we prepare some combinatorial
notation.

\para{5.1.}
A composition is a sequence of integers $\la = (\la_1, \la_2, \dots)$ with finitely many nonzero
terms.  A composition $\la$ satisfying the property that $\la_1 \ge \la_2 \ge \cdots$ 
is called a partition. For a composition $\la$, we denote by  $|\la| = \sum_i \la_i$ 
the size of $\la$. For a positive integer $r$, an $r$-tuple of partitions 
$\Bla = (\la^{(1)}, \dots, \la^{(r)})$ is called an $r$-partition.  We denote by 
$|\Bla| = \sum_i |\la^{(i)}|$ the size of $\Bla$. We express an $r$-partition by 
$\Bla = (\la^{(i)}_j)$ with partitions $\la^{(i)} = (\la^{(i)}_1, \dots, \la^{(i)}_m)$ by 
choosing sufficiently large $m$ so that $\la^{(i)}_j = 0$ for $j > m$ and for any $i$. 
The set of $r$-partitions of size $n$ is denoted by $\CP_{n,r}$.  In the case where 
$r = 1$, the set $\CP_{n,1}$ of partitions of $n$ is simply denoted by $\CP_n$. 
For a given $\Bm \in \CQ_{n,r}$, we denote by $\CP(\Bm)$ the set of $\Bla \in \CP_{n,r}$
such that $|\la^{(i)}| = m_i$ for each $i$. 
\par 
Let $\Bla = (\la^{(i)}_j)$ be an $r$-partition of $n$.  We define a composition 
$c(\Bla)$ of $n$ associated to $\Bla$ by 
\begin{equation*} 
c(\Bla) = (\la^{(1)}_1, \la^{(2)}_1, \dots, \la^{(r)}_1,  
            \la^{(1)}_2, \la^{(2)}_2, \dots, \la^{(r)}_2, 
            \dots, \la^{(1)}_m, \la^{(2)}_m, \dots, \la^{(r)}_m) 
\end{equation*}  
For example, if $\Bla = (320; 211; 411) \in \CP_{15,3}$, we have $c(\Bla) = (324211011)$.  
\par
For a composition $\la = (\la_1, \la_2, \dots, ), \mu = (\mu_1, \mu_2, \dots,)$,  
we denote by $\la \le \mu$ if 
\begin{equation*} 
\la_1 + \cdots + \la_k \le \mu_1 + \cdots + \mu_k  
\end{equation*} 
for $k = 1, 2, \cdots$. 
We define a dominance order $\le $  on $\CP_{n,r}$ by the condition that 
$\Bla \le \Bmu$ if $c(\Bla) \le c(\Bmu)$. In the case where $r = 1$, this is the standard 
dominance order on the set $\CP_n$.  
\par
For a partition $\la = (\la_1, \la_2, \dots)$, we put $n(\la) = \sum_{i \ge 1} (i-1)\la_i$. 
We define a function $n : \CP_{n,r} \to \BZ$ by $n(\Bla) = \sum_{i=1}^r n(\la^{(i)})$.  

\para{5.2.}
We consider the enhanced space of higher level introduced in 1.2. However 
in this section, we redefine them directly, without using the symmetric space setting
in 1.2. 
For an integer $r \ge 1$, we consider a variety 
$\CX\uni = G\uni \times V^{r-1}$, 
where $V$ is an $n$-dimensional vector space over
$\Bk$, $G = GL(V)$ and $G\uni$ is the unipotent variety of $G$. 
We fix a basis $\{ e_1, \dots, e_n \}$ of $V$, and define $M_i$ as a subspace of 
$V$ spanned by $e_1, \dots, e_i$.  Let $B$ be a Borel subgroup of $G$ which is the stabilizer 
of the total flag $(M_i)$.  Let $T$ be a maximal torus of $B$ such that $\{e_i \}$ are 
weight vectors for $T$.  Then $B = TU$, where $U$ is the unipotent radical of $B$.
For $\Bm \in \CQ_{n,r}$, let $\Bp(\Bm) = (p_1, \dots, p_r)$ be as in 1.2.     
We define  
\begin{align*}
\wt\CX_{\Bm,\unip} &= \{ (x, \Bv, gB) \in G\uni \times V^{r-1} \times G/B
                 \mid g\iv xg \in U, g\iv \Bv \in \prod_{i=1}^{r-1}M_{p_i} \}, \\ 
\CX_{\Bm,\unip}  &= \bigcup_{g \in G}g(U \times \prod_{i=1}^{r-1} M_{p_i}). 
\end{align*}
We define a map $\pi_1^{(\Bm)}: \wt\CX_{\Bm, \unip} \to \CX\uni$ by 
$(x,\Bv,gB) \mapsto (x,\Bv)$. 
Clearly $\CX_{\Bm, \unip} = \Im \pi_1^{(\Bm)}$, and $\CX_{\Bm,\unip}$ coincides with 
$\CX\uni$ in the case where $\Bm = (n, 0, \dots, 0)$. Since 
$\wt\CX_{\Bm,\unip} \simeq  G \times^{B}(U \times \prod_i M_{p_i})$, 
$\wt\CX_{\Bm,\unip}$ is smooth and irreducible. Since $\pi_1^{(\Bm)}$ is proper, 
$\CX_{\Bm, \unip}$ is a closed irreducible subvariety of $\CX\uni$.

\para{5.3.}
We shall define a partition of $\CX\uni$ into pieces $X_{\Bla}$ indexed by 
$\Bla \in \CP_{n,r}$,
\begin{equation*}
\tag{5.3.1}
\CX\uni = \coprod_{\Bla \in \CP_{n,r}}X_{\Bla}
\end{equation*}
satisfying the property such that $X_{\Bla}$ is $G$-stable, and that 
if $(x, \Bv) \in X_{\Bla}$, then the Jordan type of $x$ is 
$\la^{(1)} + \cdots + \la^{(r)}$. 
It is known by [AH] and [T] that the set of $G$-orbits in $G\uni \times V$ is finite, and 
they are parametrized by 
$\CP_{n,2}$. The labelling is given as follows;  take $(x,v) \in G\uni \times V$.  
Let $E^x = \{ y \in \End (V) \mid xy = yx\}$.  $E^x$ is a subalgebra of $\End (V)$
stable by the multiplication of $x$.  If we put $W= E^xv$,  $W$ is an $x$-stable 
subspace of $V$. We denote by $\la^{(1)}$ the Jordan type of $x|_W$, and by 
$\la^{(2)}$ the Jordan type of $x|_{V/W}$.  Then the Jordan type of $x$ is 
$\la^{(1)} + \la^{(2)}$ and   
$\Bla = (\la^{(1)}, \la^{(2)}) \in \CP_{n,2}$.  We denote by $\CO_{\Bla}$ 
the $G$-orbit containing $(x,v)$.  This gives the required labelling of $G$-orbits in 
$G\uni \times V$. 
\par
In general,  we define $X_{\Bla}$ by induction on $r$. 
Take $(x, \Bv) \in \CX\uni$ with $\Bv = (v_1, \dots, v_{r-1})$.  
Put $W = E_G^xv_1$, $\ol V = V/W$ and $\ol G = GL(\ol V)$.
We consider the variety 
$\CX'\uni = \ol G \times \ol V^{r-2}$. 
Assume that $(x,v_1) \in G\uni \times V$ is of type $(\la^{(1)}, \nu')$, 
where $\nu = \la^{(1)} + \nu'$ is the type of $x$.
Let $\ol x$ be the restriction of $x$ on $\ol V$.  Then the type of $\ol x \in GL(\ol V)$
is $\nu'$. Put $\ol \Bv = (\ol v_2, \dots \ol v_{r-1})$, where $\ol v_i$ is the image 
of $v_i$ on $\ol V$.
Thus $(\ol x, \ol \Bv) \in \CX'\uni$. By induction, we have a partition 
$\CX'\uni = \coprod_{\Bmu \in \CP_{n',r-1}}  X'_{\Bmu}$, where $\dim \ol V = n'$.       
Thus there exists a unique $X'_{\Bla'}$ containing $(\ol x, \ol\Bv)$.  
If we write $\Bla' = (\la^{(2)}, \dots, \la^{(r)})$, we have 
$\la^{(2)} + \cdots + \la^{(r)} = \nu'$.
It follows that $\Bla = (\la^{(1)}, \dots \la^{(r)}) \in \CP_{n,r}$. 
We define the type of $(x,\Bv)$ by $\Bla$, and define a subset $X_{\Bla}$ of 
$\CX\uni$ as the set of all $(x,\Bv)$ with type $\Bla$. Then 
$X_{\Bla}$ is a $G$-stable subset of $\CX\uni$, and we obtain the required 
partition (5.3.1).  

\par
We show the following proposition.
\begin{prop}  %%%%  Prop. 5.4
Let $\Bla = (\la^{(1)}, \dots, \la^{(r)}) \in \CP_{n,r}$.
Then $X_{\Bla}$ is a smooth irreducible variety with 
\begin{equation*}
\tag{5.4.1}
\dim X_{\Bla} = (n^2 - n - 2n(\Bla)) + \sum_{i=1}^{r-1}(r-i)|\la^{(i)}|.
\end{equation*}
\end{prop}

\begin{proof}
We may assume that $r \ge 2$. Put $\nu = \la^{(1)} + \cdots + \la^{(r)}$ and 
$\nu' = \la^{(2)} + \cdots + \la^{(r)}$. Let 
$\CO = \CO_{(\la^{(1)}, \nu')}$ be 
the $G$-orbit in $G\uni \times V$ corresponding to 
$(\la^{(1)},\nu') \in \CP_{n,2}$, and $\CO_{\nu}$ be the $G$-orbit
in $G\uni$ corresponding to $\nu \in \CP_n$. .  
We have surjective maps $f_{\Bla}: X_{\Bla} \to \CO$, 
$(x,\Bv) \mapsto (x,v_1)$, and $h_{\Bla} : X_{\Bla} \to \CO_{\nu}$,
$(x, \Bv) \mapsto x$.
For each $x \in \CO_{\nu}$, put $h_{\Bla}\iv(x) = X_{\Bla,x}$. 
The proposition certainly holds if $r = 2$ by [AH, Proposition 2.8].  
We show the proposition, together with the statement (5.4.2), 
by induction on $r$.  
\medskip
\par
\noindent  
(5.4.2) \ $X_{\Bla,x}$ is a smooth and irreducible variety with
\begin{equation*}
\dim X_{\Bla,x} = \sum_{i=1}^{r-1}(r-i)|\la^{(i)}|. 
\end{equation*}
\par\medskip
We assume that (5.4.2) holds for $\CX\uni$ with smaller $r$.
For a fixed $(x,v_1) \in \CO$, we consider the variety $\CX'\uni$ 
as in 5.3. Let $\Bla' = (\la^{(2)}, \dots, \la^{(r)})$.  Then 
by the discussion in 5.3, we see that 
$f_{\Bla}\iv(x,v_1) \simeq X'_{\Bla',\ol x} \times W^{r-2}$.  
$Z_G(x,v_1)$ stabilizes the subspace $W$ and so acts on $X'_{\Bla',\ol x}$. 
We have
\begin{equation*}
\tag{5.4.3}
X_{\Bla} \simeq G \times^{Z_G(x,v_1)}(X'_{\Bla',\ol x} \times W^{r-2}).
\end{equation*} 
By induction, $X'_{\Bla',\ol x}$ is smooth and irreducible.  Hence 
$X_{\Bla}$ is smooth and irreducible.  Since 
\begin{equation*}
\tag{5.4.4}
X_{\Bla,x} \simeq Z_G(x) \times^{Z_G(x,v_1)}(X'_{\Bla',\ol x} \times W^{r-2}),
\end{equation*}
and $Z_G(x)$ is connected, $X_{\Bla,x}$ is also smooth and irreducible.  
This proves the first statement of (5.4.1) and (5.4.2).  
We shall compute $\dim X_{\Bla}$ and $\dim X_{\Bla, x}$.  
By (5.4.3), we have
\begin{align*}
\tag{5.4.5}
\dim X_{\Bla} = \dim \CO_{(\la^{(1)},\nu')} + \dim X'_{\Bla', \ol x} 
           + (r-2)|\la^{(1)}|.
\end{align*}
Here
$\dim \CO_{(\la^{(1)}, \nu')} = n^2 - n - 2n(\nu) + |\la^{(1)}|$ by 
[AH, Prop. 2.8].
By applying (5.4.2) to $X'_{\Bla', \ol x}$, we have 
\begin{equation*}
\dim  X'_{\Bla',\ol x} = \sum_{i= 2}^{r-1}(r-i)|\la^{(i)}|.
\end{equation*}
By substituting those formulas to (5.4.5), we obtain (5.4.1) 
(note that $n(\nu) = n(\Bla)$).
\par
By comparing (5.4.3) and (5.4.4), we have
\begin{align*}
\dim  X_{\Bla, x} &= \dim  X_{\Bla} - (\dim G - \dim Z_{G}( x)) \\
                         &= \dim  X_{\Bla} - \dim \CO_{\nu}. 
\end{align*}  
Since $\dim \CO_{\nu} = n^2 - n - 2n(\nu)$,
(5.4.2) follows from (5.4.1).
This proves the proposition.
\end{proof}

\para{5.5.}                                                                                         
Let $\Bla = (\la^{(1)}, \dots, \la^{(r)}) \in \CP(\Bm)$.
We write the dual partition $(\la^{(i)})^*$ of 
$\la^{(i)}$ as ($\mu^{(i)}_1 \le \mu^{(i)}_2 \le \cdots \le \mu^{(i)}_{\ell_i})$,
in the increasing order, where $\ell_i = \la^{(i)}_1$.   
For each $1 \le i \le r, 1\le j < \ell_i$, we define an integer $n(i,j)$ by 
\begin{equation*}
n(i,j) = (|\la^{(1)}| + \cdots + |\la^{(i-1)}|) + \mu_1^{(i)} + \cdots + \mu_j^{(i)}.
\end{equation*} 
Let $P = P_{\Bla}$ be the stabilizer of a partial flag $(M_{n(i,j)})$ in $G$, and $U_P$ 
the unipotent radical of $P$. 
Let us define a variety $\wt X_{\Bla}$ by 
\begin{equation*}
\begin{split}
\wt X_{\Bla} = \{ (x, \Bv, gP) \in G\uni \times V^{r-1} \times G/P
        \mid g\iv xg \in U_P, g\iv \Bv \in \prod_{i=1}^{r-1}M_{p_i} \}.
\end{split}
\end{equation*}
We define a map $\pi_{\Bla} : \wt X_{\Bla} \to \CX\uni$ by 
$(x,\Bv, gP) \mapsto (x,\Bv)$. Then $\pi_{\Bla}$ is a proper map. 
Since 
$\wt X_{\Bla} \simeq G\times^{P}(U_{P} 
          \times \prod_i M_{p_i})$, 
$\wt X_{\Bla}$ is smooth and irreducible. 
\par
We have the following lemma.

 \begin{lem}  %%%%% Lemma 5.6
Let $\Bla = (\la^{(1)}, \dots, \la^{(r)}) \in \CP_{n,r}$.
\begin{enumerate}
\item
$\dim \wt X_{\Bla} =  \dim X_{\Bla}$.
\item
$\Im \pi_{\Bla} = \ol X_{\Bla}$, where $\ol X_{\Bla}$ is the closure of $X_{\Bla}$ 
in $\CX\uni$. 
\end{enumerate} 
\end{lem}

\begin{proof}
We have 
\begin{align*}
\dim \wt X_{\Bla} &= 
  \dim G/P + \dim U_{P} + \dim \prod_{i=1}^{r-1}M_{p_i} \\
    &= 2\dim U_{P} + \sum_{i=1}^{r-1}(r-i)|\la^{(i)}|.
\end{align*}
Here $\dim U_{P} = (\dim G - \dim L)/2$, where $L$ is a
Levi subgroup of $P$, and 
\begin{align*}
\dim L = \sum_{i,j}(\mu^{(i)}_j)^2 = n + 2n(\Bla). 
\end{align*}  
The second equality follows form the formula $n(\la^*) = \sum_i \la_i(\la_i-1)/2$ 
for a partition $\la = (\la_1, \la_2,\dots)$ and its dual $\la^*$. 
By comparing this with Proposition 5.4, we obtain  (i). 
\par
We show (ii).
Take $(x, \Bv) \in X_{\Bla}$.  By the construction of $X_{\Bla}$ in 5.3, 
one can find a sequence of $x$-stable subspaces 
$V_{p_1} \subset V_{p_2} \subset \cdots \subset V_{p_r} = V$
of $V$ such that $v_i \in V_{p_i}$ and that the restriction of $x$ on 
$V_{p_i}/V_{p_{i-1}}$ has type $\la^{(i)}$.  
It is well-known that there exists an $x$-stable flag $(V_{n(i,j)})$ 
which is a refinement of $(V_{p_i})$ such that
$V_{n(i,j)} = gM_{n(i,j)}$ for some $g \in G$ and that $g\iv xg \in U_P$   
(see [AH, Prop. 3.3 (4)]). 
It follows that $(x,\Bv, gP) \in \wt X_{\Bla}$, and we see that 
$(x,\Bv) \in \Im \pi_{\Bla}$.
This proves that $X_{\Bla} \subset \Im \pi_{\Bla}$.  Since $\pi_{\Bla}$ is proper, 
$\Im \pi_{\Bla}$ is a closed subset of $\CX\uni$ and we have 
$\ol X_{\Bla} \subset \Im \pi_{\Bla}$.  
Since $\ol X_{\Bla}$ and $\Im \pi_{\Bla}$ are both irreducible, 
(i) implies that $\Im \pi_{\Bla} = \ol X_{\Bla}$. 
Hence (ii) holds.  The lemma is proved. 
\end{proof}

\remark{5.7.} 
In the case where $r = 2$, Achar-Henderson proved in [AH, Prop. 3.3] that 
the map $\pi_{\Bla} : \wt X_{\Bla} \to \ol X_{\Bla}$
is a resolution fo singularities for $\ol X_{\Bla}$.  By a similar argument, 
by usnig Lemma 5.6 (ii), one can prove that $\pi_{\Bla}$ gives a resolution 
of singularities for $\ol X_{\Bla}$ for any $r \ge 2$. 

\para{5.8.}
Recall the map $\pi_1^{(m)}: \wt\CX_{\Bm,\unip}\to \CX_{\Bm,\unip}$ in 5.2 for 
$\Bm = (m_1, \dots, m_r) \in \CQ_{n,r}$.  Let us define  
$\Bla(\Bm) = (\la^{(1)}, \dots, \la^{(r)})$ by the condition that
$\la^{(i)} = (m_i)$ for each $i$.
We consider the varieties $\wt X_{\Bla}$ in 5.5 
for $\Bla = \Bla(\Bm)$.  In this
case, $P = B$. $\wt X_{\Bla}$ is
isomorphic to $\wt\CX_{\Bm,\unip}$, and $\pi_{\Bla}$ is identified 
with $\pi_1^{(\Bm)}$.  
We have the following result.

\begin{prop}  %%% Prop. 5.9
For $\Bm \in \CQ_{n,r}$, we have 
\begin{enumerate}
\item
$\CX_{\Bm,\unip} = \ol X_{\Bla(\Bm)}$.  
\medskip
\item 
$\dim \CX_{\Bm,\unip} =  n^2 - n + \sum_{i=1}^{r-1}(r-i)m_i$.
\medskip
\item 
For $\Bmu \in \CP(\Bm)$, $X_{\Bmu} \subset \CX_{\Bm,\unip}$. 
\end{enumerate}
\end{prop} 

\begin{proof}
(i) is a direct consequence of Lemma 5.6 (ii) in view of 5.8.
(ii) follows from (i) and Proposition 5.4.
We show (iii).  Take $(x,\Bv) \in X_{\Bmu}$.
As in the proof of Lemma 5.6, there exists an $x$-stable partial flag 
$(V_{n(i,j)})$ with respect to $\Bmu$.  
By our assumption, $v_i \in V_{p_i}$  for each $i$.
Then one can find an $x$-stable total flag $(V_i)$ as a refinement of 
this $(V_{n(i,j)})$.
This shows that $(x,\Bv)$ is contained in $\Im \pi^{(\Bm)}_1 = \CX_{\Bm,\unip}$.
Hence $X_{\Bmu} \subset \CX_{\Bm, \unip}$ as asserted.
\end{proof}

\para{5.10.}
In the case where $r = 2$, there exists a normal basis for $(x,v) \in G\uni\times V$ 
 (cf. [AH, 2.2], [T]).  If $r \ge 3$, one can not expect such a basis since 
$\CX\uni$ has infinitely many $G$-orbits.  However, one can find typical 
elements in $X_{\Bla}$ as follows; put $\nu = (\nu_1, \dots, \nu_{\ell}) \in \CP_n$ 
for $\nu = \la^{(1)} + \cdots + \la^{(r)}$.  Take $x \in G\uni$ of Jordan type 
$\nu$, and let $\{u_{j,k} \mid  1 \le j \le \ell, 1 \le k \le \nu_j \}$ be a Jordan basis of 
$x$ in $V$ having the property $(x-1) u_{j,k} = u_{j,k-1}$ with the convention that $u_{j,0} = 0$.
We define $v_i \in V$ for $i = 1, \dots, r-1$ by the condition that 
\begin{equation*} 
\tag{5.10.1} 
v_i = \sum_{\substack{ 1 \le j \le \ell \\}}
             u_{j, \la_j^{(1)} + \cdots + \la_j^{(i)}} 
\end{equation*} 
and put $\Bv = (v_1, \dots, v_{r-1})$.
Let $W_i$ be a subspace of $V$ spanned by the basis 
\begin{equation*} 
\{ u_{j, k} \mid 1 \le j \le \ell, 
    1 \le k \le \la_j^{(1)} + \cdots + \la^{(i)}_j \}.  
\end{equation*}  
Then $W_1 \subset W_2 \subset \cdots \subset W_{r-1} \subset V$, 
and $W_i$ is an $x$-stable subspace of $V$
such that $v_i \in W_i$.  Let $\ol x$ be the restriction of $x$ on 
$\ol V = V/W_{i-1}$, 
and $\ol v_i$ be the image of $v_i$ on $\ol V$.  
Put $\bar G = GL(\ol V)$.  
Then $E^{\bar x}\ol v_i$ coincides with $W_i/W_{i-1}$, and the restriciton 
of $x$ on $W_i/W_{i-1}$ has type $\la^{(i)}$.   
It follows that $(x,\Bv) \in X_{\Bla}$. 
Such an element  $(x,\Bv) \in X_{\Bla}$ is called a standard element.  
More generally, we consider an element $\Bw = (w_1, \dots, w_{r-1})$  
of the form 
\begin{equation*}
\tag{5.10.2}
w_i = \sum_{j=1}^{\ell}\sum_{k = 1}^{\la_j^{(1)}+\cdots + \la_j^{(i)}} 
                    a_{j,k}u_{j,k}  
\end{equation*}  
with $a_{j,k} \in \Bk$ such that  
$a_{j, \la_j^{(1)}+\cdots + \la_j^{(i)}} \ne 0$.  
Then $w_i \in W_i$, and $E^{\bar x}\ol w_i$ coincides with 
$W_i/W_{i-1}$.  Hence $(x, \Bw) \in X_{\Bla}$.
Here $(x,w_i)$ is conjugate to  
an element $(x,v_i)$ as above under the group $Z_G(x)$ for each $i$.  
We call $(x, \Bw)$ a semi-standard element.  
We define a set $X_{\Bla}^0$ as the set 
of all $G$-conjugates of $(x,\Bw)$.  Hence $X_{\Bla}^0$ is 
a $G$-stable subset of $X_{\Bla}$.  
We note that 
\par\medskip\noindent
(5.10.3) \ $X_{\Bla}^0$ is a $G$-stable open dense subset 
of $X_{\Bla}$. 
\par\medskip
In fact, take $(x, \Bv) \in X_{\Bla}^0$.  Under the notation in the proof of Proposition 5.4, 
put $X_{\Bla,x}^0 = X_{\Bla}^0 \cap X_{\Bla, x}$.  
Then under the isomorphism  in (5.4.4), we have
\begin{equation*}
X_{\Bla,x}^0 \simeq Z_G(x) \times^{Z_G(x,v_1)}(X'^0_{\Bla',\bar x} \times (W^0)^{r-2}), 
\end{equation*} 
where $X_{\Bla',\bar x}'^0$ is a similar variety as $X_{\Bla, x}^0$ defined for $X'_{\Bla'}$,
and $W^0$ is an open dense subset of $W$.   
Hence by induction on $r$, $X_{\Bla, x}^0$ is an open dense subset of $X_{\Bla}$. 
Since, under the isomorphism  in (5.4.3), $X_{\Bla}^0$ can be written as 
\begin{equation*}
X_{\Bla}^0 \simeq G \times^{Z_G(x,v_1)}(X'^0_{\Bla',\bar x} \times (W^0)^{r-2}),  
\end{equation*}
$X_{\Bla}^0$ is open dense in $X_{\Bla}$.  Hence (5.10.3) holds.  
\par
Concerning the closure $\ol X_{\Bla}$, we have the following result.

\begin{prop} %%%% Prop. 5.11.
For each $\Bla \in \CP_{n,r}$, we have
\begin{equation*}
\ol X_{\Bla}  \subset \bigcup_{\Bmu \le \Bla}X_{\Bmu}.
\end{equation*}
\end{prop}

For the proof of the proposition, we need a lemma.  

\begin{lem}  %%% Lemma 5.12
Assume that $(x,\Bv) \in X_{\Bla}$ is a semi-standard element 
with $\Bv = (v_1, \dots, v_{r-1})$.
Let $U_i$ be the $\Bk[x]$-submodule of $V$ generated by $v_i$. 
Then $\dim U_i = \la_1^{(1)} + \cdots + \la^{(i)}_1$. The Jordan type of 
the restriction of $x$ on $V/U_i$ is $\xi = (\xi_1, \dots, \xi_{\ell})$, where  
\begin{equation*} 
\xi_j = \la_j^{(i+1)}+ \cdots + \la_j^{(r)} + 
\la_{j+1}^{(1)} + \cdots + \la_{j+1}^{(i)}, 
\quad \text{ for } j = 1, \dots, \ell.  
\end{equation*}  

\begin{proof}  
In the case where $r = 2$, this result was proved in Lemma 2.5 in [AH].  
The general case is reduced to the case where $r = 2$, by considering  
the double partition $(\mu; \mu')$ with  $\mu = \la^{(1)} + \cdots + \la^{(i)}$, 
$\mu' = \la^{(i+1)} + \cdots + \la^{(r)}$. 
\end{proof}  
\end{lem}                                                                                                   
\para{5.13.}
We prove the proposition following the strategy 
in the proof of (a part of) [AH, Th. 3.9].
We show that $\Bmu \le \Bla$ if  $(x',\Bv') \in \ol X_{\Bla}$ is of type $\Bmu$.  
For this, it is enough to show that                 
\begin{equation*}  
\tag{5.13.1}  
\begin{split}  
\sum _{i= 1}^k(\mu^{(1)}_i + &\cdots + \mu^{(r)}_i)  
   + (\mu^{(1)}_{k+1} + \cdots \mu^{(a)}_{k+1})  \\   
&\le \sum_{i=1}^k(\la^{(1)}_i + \cdots + \la^{(r)}_i)  
  + (\la^{(1)}_{k+1} + \cdots + \la^{(a)}_{k+1})  
\end{split}  
\end{equation*}  
for $k = 0, \dots, \ell$ and $a = 0, \dots, r-1$  
(we put $\la^{(0)}_j = \mu^{(0)}_j = 0$ by convention).  
Put $\nu = \la^{(1)} + \cdots + \la^{(r)}$ and  $\nu' = \mu^{(1)} + \cdots + \mu^{(r)}$. 
Then the Jordan type of $x'$  is $\nu'$.  Let $\CO_{\nu}$ be the $G$-orbit in $G\uni$ 
corresponding to $\nu \in \CP_n$.  
Since 
$X_{\Bla} \subset \CO_{\nu} \times V^{r-1}$, 
we have $\ol X_{\Bla} \subseteq  \ol\CO_{\nu} \times V^{r-1}$.  
Since $(x',\Bv') \in \ol X_{\Bla}$, we have $x' \in \ol\CO_{\nu}$ and so   
$\nu' \le \nu$.  This proves (5.13.1) in the case where $a = 0$.   
By a similar argument as in the proof of [AH, Th. 3.9], it  follows from Lemma 5.12 that 
\begin{equation*} 
\tag{5.13.2}  
\sum_{i=1}^k(\la^{(1)}_i + \cdots + \la^{(r)}_i)  
  + (\la^{(1)}_{k+1} + \cdots + \la^{(a)}_{k+1})  
\end{equation*}  
is the maximal possible dimension of the $x$-stable subspace  
$\Bk[x]\lp v_{a}, u_1, \dots, u_k \rp $ of $V$,  
a $\Bk[x]$-module generated by $v_a, u_1, \dots, u_k$,  
where $(x,\Bv)$ is a semi-standard element in $X_{\Bla}$ with  
$\Bv = (v_1, \dots, v_{r-1})$, and   $u_1, \dots, u_k$ runs over the elements in $V$.   
We note that the condition  
\begin{equation*}  
\tag{5.13.3}  
\dim \Bk [x]\lp v_a, u_1, \dots, u_k \rp \le N_a, \quad   
(u_1, \dots, u_k \in V, 1 \le a \le r-1)  
\end{equation*}  
for given $\{ N_a \mid 1 \le a \le r-1\}$  is a closed condition on  
$(x,\Bv) \in \CX\uni$ (i.e., the set of $(x,\Bv)$  
satisfying the above condition is a closed subset of  
$\CX\uni$)  since  the condition for a fixed $a$ is  
a closed condition on $\CX\uni$  by [loc. cit.].  
Since $\ol X_{\Bla}^0 = \ol X_{\Bla}$ by (5.10.3), for any element  
$(x',\Bv') \in \ol X_{\Bla}$, the dimension of  $\Bk[x']\lp v'_a, u_1, \dots,u_k\rp$ 
is dominated by the number in (5.13.2).  The assertion (5.13.1) then follows from  this.
Hence Proposition 5.11 holds.
\par\medskip
Combining with the previous results, we have the following.

\begin{prop}  %%%  Prop. 5.14
$X_{\Bla}$ is open dense in $\ol X_{\Bla}$.  Hence 
$X_{\Bla}$ is a $G$-stable locally closed, smooth, irreducible subvariety of  
$\CX\uni$.  
\end{prop}

\begin{proof}
Since $X_{\Bla} = \ol X_{\Bla} \backslash \bigcup_{\Bmu < \Bla}\ol X_{\Bmu}$
by Proposition 5.11, $X_{\Bla}$ is open in $\ol X_{\Bla}$. 
$X_{\Bla}$ is smooth and irreducible by Proposition 5.4.  
\end{proof}

\para{5.15.}
We give here some examples on  the closure relation of $X_{\Bla}$.                                   
First assume that $n = 1$ and $r$ is arbitrary.                                                      
Then $G = GL_1 \simeq \Bk^*$, the multiplicative group of $\Bk$.  
We have $G\uni = \{ 1 \}$, and $\CX\uni \simeq \Bk^{r-1}$. 
$\CP_{1,r}$ consists of $\{ \Bla_1, \dots, \Bla_r\}$, where                                          
$\Bla_i = (\la^{(1)}, \dots, \la^{(r)})$ with $\la^{(r+1-i)} = (1)$ and                              
$\la^{(j)} = \emptyset $ for $j \ne r+1-i$   Thus the dominance order                                
in $\CP_{1,r}$ is given as $\Bla_1 < \Bla_2 < \cdots < \Bla_r$.                                      
Under the identification $\CX\uni \simeq \Bk^{r-1}$, we have                              
\begin{equation*}                                                                                    
\begin{split}                                                                                        
X_{\Bla_i} &= \{ (v_1, \dots, v_{r-1}) \in \Bk^{r-1}                                                 
    \mid v_1 = \dots = v_{r-i} = 0, v_{r-i+1} \in \Bk^* \} \\                                        
   &\simeq                                                                                           
      \begin{cases}                                                                                  
           \Bk^{i-2}\times \Bk^* &\quad\text { if } i \ge 2, \\                                      
           \{ 0\} &\quad\text{ if } i = 1.                                                           
      \end{cases}                                                                                    
\end{split}                                                                                          
\end{equation*}       
Thus $X_{\Bla_i}$ is a locally closed, smooth irreducible subvariety of                              
$\Bk^{r-1}$.  Since $G \simeq \Bk^*$, $X_{\Bla_1} = \{ 0\}$ and                                      
$X_{\Bla_2} \simeq \Bk^*$ are single $G$-orbits, but other                                           
$X_{\Bla_i}$                                                                                         
are a union of infinitely many $G$-orbits. Moreover                                                  
$\ol X_{\Bla_i} \simeq \Bk^{i-1}$, and the closure relation is given                                 
as                                                                                                   
\begin{equation*}
\tag{5.15.1}                                                                                    
\ol X_{\Bla_i} = \bigcup_{j \le i}X_{\Bla_j}.                                                        
\end{equation*}                                                                                      
This is an analogue of the result in the case where $r = 2$,
which asserts that $\ol X_{\Bla}$ is a union of $X_{\Bmu}$ for                                       
$\Bmu \le \Bla$ ([AH, Th. 3.9]). 
\par                                                                                                 
However, such a relation does not hold in general for $r \ge 3$.                                     
In fact, there exists $\Bmu < \Bla$ such that                                                        
$X_{\Bmu} \cap \ol X_{\Bla} \ne \emptyset$ and that                                                  
$X_{\Bmu} \not\subset \ol X_{\Bla}$ as the following example shows.                                  
We consider the case where $n = 2, r=3$, hence $\dim V = 2$.   
Take $\Bla = (1; - ; 1)$                                                                             
and $\Bmu = (1^2; -; -)$.  Then we have $\Bmu < \Bla$.  It is easy to                                
check that                                                                                           
\begin{align*}                                                                                       
X_{\Bla} &= \{ (x,(v_1, v_2)) \in G\uni \times V^2 \mid 
                 x\ne 1, v_1 \in \Ker (x-1) \, \backslash \{0\},                                         
        v_2 \in \Ker (x - 1)  \}, \\                                                                       
X_{\Bmu} &= \{ (x,(v_1, v_2)) \in G\uni \times V^2 \mid  
                x = 1,  v_1 \in V \backslash \{0\}, v_2 \in V\}.                                     
\end{align*}                                    
Then for any element $(x,\Bv) \in X_{\Bla}$, we have $v_2 \in \Bk v_1$.
It follows that $v_2 \in  \Bk v_1$ for $(x,\Bv) \in \ol X_{\Bla}$, 
and we have $X_{\Bmu} \cap \ol X_{\Bla} 
    = \{ (x, v_1,v_2) \mid x = 1, v_1 \ne 0, v_2 \in \Bk v_1 \}$.
In particular, $X_{\Bmu} \cap \ol X_{\Bla} \ne \emptyset$ and 
$X_{\Bmu} \not\subset \ol X_{\Bla}$.  
\par
The following result is an analogue  of [AH, Cor. 3.4].

\begin{prop}  %%%% Prop. 5.16
Let $(x, \Bv) \in \CX\uni$, 
and $\Bla,  \Bxi \in \CP(\Bm)$ for some $\Bm$.  
Then $(x, \Bv) \in \ol X_{\Bla}$ if and only if 
there exists a flag $(W_i)_{1 \le i \le r}$ in 
$V$ such that $\dim W_i = p_i$ satisfying the following conditions;
\begin{enumerate}
\item 
$(W_i)$ is $x$-stable, 
\item 
$x|_{W_i/W_{i-1}}$ has type $\xi^{(i)}$ with $\xi^{(i)} \le \la^{(i)}$ for 
$1 \le i \le r$, 
\item
$v_i \in W_i$ for $1 \le i \le r-1$. 
\end{enumerate}
   
\end{prop}

\begin{proof}
Let $\CU$ be the set of $(x, \Bv)$ satisfying the condition in the proposition. 
We show that $\CU = \ol X_{\Bla}$.  
Take $(x, \Bv) \in \CU$.
In order to show that $(x, \Bv) \in \ol X_{\Bla}$,   
by Lemma 5.6, it is enough to construct $(x, \Bv, gP_{\Bla}) \in \wt X_{\Bla}$. 
Under the notation in 5.5, 
put $V_{n(i,0)} = W_i$ for $i = 1, \dots, r-1$.  By assumption,  
the restriction of $x$ on $W_i/W_{i-1}$ has type $\xi^{(i)}$.  
Since 
$\xi^{(i)} \le \la^{(i)}$, there exists an $x$-stable flag 
\begin{equation*}
W_{i-1} = V_{n(i-1,0)} \subset V_{n(i-1,1)} \subset \cdots \subset V_{n(i,0)} = W_i.
\end{equation*}
such that $(x-1)V_{n(i,j)} \subset V_{n(i, j-1)}$. 
In fact this is an application of Lemma 5.6 for $r = 1$, 
which is a well-konwn result for $GL_n$.  
There exists $g \in G$ such that $g(M_{n(i,j)}) = (V_{n(i,j)})$, 
and $(x, \Bv, gP_{\Bla}) \in \wt X_{\Bla}$. 
Hence $(x, \Bv) \in \ol X_{\Bla}$, and we have $\CU \subset \ol X_{\Bla}$.
\par
Let $P$ be a parabolic subgroup of $G$ which is the stabilizer of a partial flag
$(M_{p_i})_{1 \le i \le r}$, and $L$ the Levi subgroup of $P$ containing $T$.
Thus $L \simeq \prod_i GL_{m_i}$.  Let $\CO'_{\Bla}$ be the $L$-orbit in $L\uni$
corresponding to $\Bla = (\la^{(1)}, \dots, \la^{(r)})$, and $\ol\CO'_{\Bla}$ 
the closure of $\CO'_{\Bla}$ in $L\uni$. 
Let $p: P \to L$ be the natural projection. 
We consider a variety
\begin{equation*}
\wt\CP_{\CO'_{\Bla}} = \{(x, \Bv, gP) \in G\uni \times V^{r-1} \times G/P 
           \mid g\iv xg \in p\iv(\ol\CO'_{\Bla}), g\iv v_i \in M_{p_i} \}, 
\end{equation*}
and let $f : \wt\CP_{\CO'_{\Bla}} \to G\uni \times V^{r-1}$ be the projection 
on the first two factors. 
Then $\Im f = \CU$.  Since $f$ is proper, $\CU$ is a closed subvariety of 
$G\uni \times V^{r-1}$.  The construction of $X_{\Bla}$ implies that 
$X_{\Bla} \subset \CU$.  It follows that $\ol X_{\Bla} \subset \CU$, and so 
$\CU = \ol X_{\Bla}$.  
The proposition is proved. 
\end{proof}

As a corollary, we have the following result.

\begin{cor}  %%%% Cor 5.17.
\begin{enumerate}
\item
Let $\CU_{\CO'_{\Bla}}$ be the set of $(x, \Bv) \in G\uni \times V^{r-1}$
satisfying the condition; there exists a flag $(W_i)_{1 \le i \le r}$ in $V$
such that $\dim W_i = p_i$ and that 
\par
$(a)$ \ $(W_i)$ is $x$-stable,
\par
$(b)$ \ $x|_{W_i/W_{i-1}}$ has type $\la^{(i)}$, 
\par
$(c)$ \ $v_i \in W_i$ for $i = 1, \dots, r-1$.
Then 
\begin{equation*}
\CU_{\CO'_{\Bla}} = \ol X_{\Bla} \ \backslash \ \bigcup_{\Bmu}\ol X_{\Bmu},
\end{equation*}
wehere $\Bmu$ runs over elements in $\CP_{n,r}$ such that 
$\mu^{(i)} \le \la^{(i)}$ for any $i$ and that $\Bmu \ne \Bla$.   
\item
$X_{\Bla}$ coincides with the set of $(x, \Bv) \in \CU_{\CO'_{\Bla}}$ 
such that the Jordan type of $x$ is $\la^{(1)} + \cdots + \la^{(r)}$ 
and that the type of $(x|_{W_i/W_{i-1}}, \ol v_i)$ is $(\la^{(i)}, \emptyset)$ 
for each $i$, 
where $\ol v_i$ is the image of $v_i$ to $W_i/W_{i-1}$.   
\end{enumerate}
\end{cor}

\begin{proof}
(i) is clear from Proposition 5.16.  By induction on $r$, the proof of (ii) is reduced 
to the case where $r = 2$.  In that case, the assertion easily follows from 
the discuusion by using the normal basis [AH, 2.2].  
\end{proof}

\para{5.18.}
Take $\Bm \in \CQ_{n,r}$.  
For each $z = (x,\Bv) \in \CX_{\Bm,\unip}$,
put
\begin{equation*}
\CB_z^{(\Bm)} = \{ gB \in G/B \mid 
    g\iv xg \in U, g\iv \Bv \in \prod_{i=1}^{r-1}M_{p_i} \}. 
\end{equation*}
$\CB^{(\Bm)}_z$ is a closed subvariety of $\CB = G/B$, which is 
isomorphic to the fibre $(\pi^{(\Bm)}_1)\iv(z)$ and is called the Springer fibre of $z$. 
In the case where $r = 2$, $\CB_z^{(\Bm)}$ is isomorphic each other 
for any $z \in X_{\Bla}$ 
since $X_{\Bla}$ is a single $G$-orbit.  However this property does not hold 
in general for $r \ge 3$ as the next examples show.  
First assume that $r = 3$ and $n = 3$.  
Let $\Bm = (2,0,1)$ and $\Bla = ((1^3), -, -)$.  
Let $v_1, v_2$ be linearly independent vectors in $V$, and 
take
$z = (x, (v_1, v_2)), z' = (x, (v_1, 0))$ with $x = 1$.    
Then $z, z' \in X_{\Bla}$.  
Let $\CF(V)$ be the set of total flags $(V_i)$ in $V$.  
We have 
\begin{align*}
\CB_z^{(\Bm)} &\simeq \{ (V_i) \in \CF(V) \mid V_2  = \lp v_1, v_2 \rp \}, \\
\CB_{z'}^{(\Bm)} &\simeq \{ (V_i) \in \CF(V) \mid v_1 \in V_2 \}.
\end{align*} 
Hence $\dim \CB_z^{(\Bm)} = 1, \dim \CB_{z'}^{(\Bm)} = 2$,  and 
in particular, $z, z' \in X_{\Bla} \cap \CX_{\Bm, \unip}$. We have
$\dim \CB_z^{(\Bm)} \ne \dim \CB_{z'}^{(\Bm)}$.    
\par
Next assume that $r = 3$ and $n = 4$.  Let $\Bm = (2,0,2)$ and $\Bla = ((1^2),-, (1^2))$. 
We have $X_{\Bla} \subset \CX_{\Bm,\unip}$.  
Let $x$ be an element in $G\uni$ of type $(2^2)$, and 
$\{ u_{i,j} \mid 1 \le i \le 2, 1 \le j \le 2\}$ be a Jordan baasis of $x$ in $V$.
Put $v_1 = u_{1,1}, v_2 = u_{2,1}$ so that $v_1, v_2$ is a basis of $W = \Ker (x-1)$.
Put $z = (x, (v_1, v_2))$ and $z' = (x, (v_1, 0))$. Then $z, z' \in X_{\Bla}$. 
We have
\begin{align*}
\CB_z^{(\Bm)} &\simeq \{ (V_i) \in \CF_x(V) \mid V_2 = W  \}, \\
\CB_{z'}^{(\Bm)} &\simeq \{ (V_i) \in \CF_x(V) \mid v_1 \in V_2  \},
\end{align*}
where $\CF_x(V)$ is the set of $x$-stable flags in $V$. 
Since $x|_W = 1, x_{V/W} = 1$, we see that 
$\CB_z^{(\Bm)} \simeq \CF(W) \times \CF(V/W)$. Hnece 
$\CB_z^{(\Bm)}$ is irreducible with $\dim \CB_z^{(\Bm)} = 2$. 
On the other hand, if $(V_i) \in \CB_{z'}^{(\Bm)}$, either $V_2 = W$ 
or $V_2$ is of the form  $W_{\a} = \lp u_{1,2} + \a v_2, v_1 \rp$ for $\a \in \Bk$. 
Since $x|_{W_{\a}}$ and $x|_{V/W_{\a}}$ have both type (2), $W_{\a}$ determines 
a unique $(V_i) \in \CB_{z'}^{(\Bm)}$. It follows that 
$\CB_{z'}^{(\Bm)} = \CB_z^{(\Bm)} \coprod Y$ with 
$Y = \{ (V_i) \mid V_2 = W_{\a} (\a \in \Bk) \}$, 
where $Y$ is irreducible with $\dim Y = 1$. 
Hence in this case, $\dim \CB^{(\Bm)}_z = \dim \CB^{(\Bm)}_{z'}$, 
but $\CB^{(\Bm)}_z \not\simeq \CB^{(\Bm)}_{z'}$.    

%%%%%
%%%%%
%%%%%
\par\bigskip\bigskip
\section{Unipotent variety of exotic type}

The ``unipotent part" of the exotic space is called the unipotent 
variety of exotic type. In this section, first we study the unipotent 
variety of exotic type in 6.1 - 6.13.  While in 6.14 - 6.19, we discuss 
the case of enahnced type.  After 6.20, we disucss the both cases 
simultaneously.

\para{6.1.}
We follow  the notation in 1.2.  Assume that $\CX$ is of exotic type.  
As in the enhanced case we define varieties, for each $\Bm \in \CQ_{n,r}$,  
\begin{align*}
\wt\CX_{\Bm,\unip} &= \{ (x, \Bv, gB^{\th}) \in G^{\io\th}\uni \times V^{r-1} 
        \times H/B^{\th}
                 \mid g\iv xg \in U^{\io\th}, g\iv \Bv \in \prod_{i=1}^{r-1}M_{p_i} \}, \\ 
\CX_{\Bm,\unip}  &= \bigcup_{g \in H}g(U^{\io\th} \times \prod_{i=1}^{r-1} M_{p_i}), 
\end{align*}
where $B,T$ are as in 1.2, and $U$ is the unipotent radical of $B$. 
We define a map $\pi_1^{(\Bm)}: \wt\CX_{\Bm, \unip} \to G\uni^{\io\th} \times V^{r-1}$ by           
$\pi_1^{(\Bm)}(x,\Bv,gB^{\th}) = (x,\Bv)$.                                                          
Clearly $\CX_{\Bm, \unip} = \Im \pi_1^{(\Bm)}$.
As in 1.2, in the case where $\Bm = (n, 0, \dots, 0)$, we write                                     
$\wt\CX_{\Bm, \unip}, \CX_{\Bm,\unip}, \pi_1^{(\Bm)}$, etc.                                         
by $\wt\CX\uni, \CX\uni, \pi_1$, etc.     
Note that $\CX_{\Bm, \unip} \subset \CX\uni$ for any $\Bm$, but contrast to the enhanced case, 
$\CX\uni$ does not coincide with 
$G^{\io\th}\uni \times V^{r-1}$ if $r \ge 3$.  
The variety $\CX\uni$ is called a unipotent variety of exotic type.
\par
 Since                                               
$\wt\CX_{\Bm,\unip} \simeq  H \times^{B^{\th}}(U^{\io\th} \times \prod_i M_{p_i})$,                 
$\wt\CX_{\Bm,\unip}$ is smooth and irreducible. 
Moreover, we have
\begin{align*}
\tag{6.1.1}
\dim \wt\CX_{\Bm,\unip} &= \dim H/B^{\th} + (\dim U^{\io\th} + \dim \prod_{i=1}^{r-1}M_{p_i}) \\
                  &= \dim U^{\th} + \dim U^{\io\th} + \sum_{i=1}^{r-1}p_i \\
                  &= 2n^2 - n + \sum_{i=1}^{r-1}(r-i)m_i 
\end{align*}
since $\dim U^{\th} + \dim U^{\io\th} = \dim U$. 
Since $\pi_1^{(\Bm)}$ is proper,                    
$\CX_{\Bm, \unip}$ is a closed irreducible subvariety of $G\uni^{\io\th} \times V^{r-1}$.           
    
\para{6.2.}
Assume that $r \ge 2$, and take $\Bm \in \CQ_{n,r}$.
Let $P$ be the $\th$-stable parabolic subgroup of $G$ such that 
$P^{\th}$ is the stabilizer of the isotropic flag 
$(M_{p_i})_{1 \le i \le r-2}$ in $V$.   
Let $L$ be the $\th$-stable Levi subgroup of $P$ containing $T$ and $U_P$  the 
unipotent radical of $P$. 
Put $V_L' = M_{p_{r-2}}$, and $V_L = V/V_L'$.  
Then $P^{\th}$ acts naturally on $V_L$, and we consider the map 
$\pi_P: P^{\io\th}\uni \times V \to L^{\io\th}\uni \times V_L$, 
$(x,v) \mapsto (x', v')$,  where $x' = p(x)$ for the natural projection
$p: P^{\io\th} \to L^{\io\th}$, and $v \mapsto v'$ is the projection 
$V \to V_L = V/V_L'$. 
The map $\pi_P$ is $P^{\th}$-equivariant with respect to the diagonal 
action of $P^{\th}$ on both varieties.
\par
Let $\CO'$ be an $L^{\th}$-orbit in $L^{\io\th}\uni \times V_L$.
We assume that $\CO'$ is $P^{\th}$-stable. 
Since $\pi_P\iv(\CO') = \bigcup_{g \in P^{\th}}g\cdot \pi_P\iv(z')$ for $z' \in \CO'$, 
$\pi_P\iv(\CO')$ is irreducible (note that $\pi_P\iv(z')$ is irreducible).  
Let
$\CO$ be an $H$-orbit in $G^{\io\th}\uni \times V$ such that 
$\CO \cap \pi_P\iv(\CO')$ is open dense in $\pi_P\iv(\CO')$.  Put 
\begin{equation*}
\CU = \{ (z, gP) \in (G^{\io\th}\uni \times V) \times H/P^{\th} 
            \mid g\iv z \in \pi_P\iv(\CO') \}.
\end{equation*}
Then $\CU \simeq H\times^{P^{\th}}\pi_P\iv(\CO')$ and so 
$\CU$ is an irreducible variety.
Let $f: \CU \to G^{\io\th}\uni \times V$ be the first projection, and 
put $\CU_\CO = f\iv(\CO)$.
Then $\CU_{\CO} \simeq H \times^{P^{\th}}(\CO \cap \pi_P\iv(\CO'))$, and so
$\CU_{\CO}$ is irreducible.
\par
Take $z \in \CO$, and consider a variety
\begin{equation*}
\CP_{z, \CO'} = 
\{ gP^{\th} \in H/P^{\th} \mid g\iv z \in \pi_P\iv(\CO') \}.
\end{equation*}
Note that $\CP_{z, \CO'} \neq \emptyset$. 
We show the following proposition.

\begin{prop}  %%%% Prop. 6.3
Under the setting in 6.2, the followings hold.
\begin{enumerate}
\item 
$\CP_{z, \CO'}$ consists of one point.
\item
$\dim Z_H(z) = \dim Z_{L^{\th}}(z')$ for $z' \in \CO'$.
\item
Let $z_1 \in \pi_P\iv(\CO')$ be such that $\dim Z_H(z_1) = \dim Z_H(z)$.
Then $z_1 \in \CO$. 
\item
Take $z \in \CO \cap \pi_P\iv(\CO')$ and put $z' = \pi_P(z)$. 
Let $Q = Z_P(z')$ be a $\th$-stable subgroup of $P$. 
Then $\dim Z_{Q^{\th}}(z) = \dim Z_H(z)$.  
In particular, 
\begin{equation*}
Z_H(z) = Z_{P^{\th}}(z) = Z_{Q^{\th}}(z).
\end{equation*}
\item 
$H$ acts transitively on $\CU_{\CO}$. 
\item 
$P^{\th}$ acts transitively on $\CO \cap \pi_P\iv(\CO')$, and 
$Q^{\th}$ acts transitively on $\CO \cap \pi_P\iv(z')$ under the 
setting in (iv).
\end{enumerate}
\end{prop}

\begin{proof}
First we show that 
\begin{equation*}
\tag{6.3.1}
\dim \CP_{z, \CO'} = 0.
\end{equation*}
Replacing $z$ and $\CP_{z, \CO'}$ by $H$-conjugate, if necessary, 
we may assume that $z \in \CO \cap \pi_P\iv(\CO')$.
Put $z' = (x',v') = \pi_P(z)$. 
Since $\pi_P\iv(z') = (x'U_P)^{\io\th} \times (v' + V_L')
           \simeq U_P^{\io\th} \times V_L'$, 
we have
\begin{equation*}
\tag{6.3.2}
\dim \pi_P\iv(z') = \dim U_P^{\io\th} + \dim V_L'
        = \dim U_P^{\th}.
\end{equation*}
Put $c = \dim \CO$ and $c' = \dim \CO'$. 
Then by [SS11, Prop. 5.7 (i)] (for a correction, see   
[SS2, Appendix]), we have 
$\dim (\CO \cap \pi_P\iv(z')) \le (c - c')/2$.  
Since $\CO \cap \pi_P\iv(\CO')$ is open dense in $\pi_P\iv(\CO')$, 
$\CO \cap \pi_P\iv(z')$ is open dense in $\pi_P\iv(z')$. 
It follows that 
\begin{equation*}
\dim U_P^{\th} \le (c  - c')/2.
\end{equation*}
On the other hand, by [SS1, Prop. 5.7 (ii)] together with [SS2, Appendix],  we have
\begin{align*}
\tag{6.3.3}
\dim \CP_{z, \CO'} &\le (\nu_H - c/2) - (\nu_{L^{\th}} - c'/2)  \\
                       &= \dim U^{\th} - \dim U_L^{\th} - (c - c')/2 \\
                       &=  \dim U_P^{\th} - (c - c')/2.
\end{align*}
Hence $\dim \CP_{z, \CO'} \le 0$.  As 
$\CP_{z, \CO'} \ne\emptyset$, we obtain (6.3.1).
\par
Substituting this into (6.3.3), we have 
$c - c' = 2\dim U_P^{\th}$.   
This implies that $\dim Z_H(z) = \dim Z_{L^{\th}}(z')$. 
Hence (ii) holds. 
\par
Under the setting in (iv), 
$Q^{\th}$ stabilizes $\pi_P\iv(z')$, and so
$\CO_Q \subset \CO \cap \pi_P\iv(z')$.  
Since $\CO'$ is an $P^{\th}$-orbit, we have 
$\dim Z_{P^{\th}}(z') = \dim Z_{L^{\th}}(z') + \dim U_P^{\th}$. 
Hence by (6.3.2) we have
\begin{equation*}
\tag{6.3.4}
\dim U^{\th}_P \ge \dim \CO_Q = \dim Z_{L^{\th}}(z') + \dim U_P^{\th} - \dim Z_{Q^{\th}}(z).
\end{equation*}
It follows that $\dim Z_{Q^{\th}}(z) \ge \dim Z_{L^{\th}}(z')$.  Since 
$\dim Z_{Q^{\th}}(z) \le \dim Z_H(z)$, we have 
$\dim Z_{Q^{\th}}(z) = \dim Z_H(z)$ by (6.3.4).  
Since $Z_H(z)$ is connected, this implies that $Z_H(z) = Z_{Q^{\th}}(z)$. 
Hence $Z_H(z) = Z_{P^{\th}}(z)$.  
Thus (iv) holds.   
\par
For any $z \in \CO$, $\dim f\iv(z) = 0$ by (6.3.1).  Thus 
$\dim \CU_{\CO} = \dim \CO$.  We show that $H$ acts transitively on $\CU_{\CO}$.
Take $\xi \in \CU_{\CO}$, and consider the $H$-orbit $H\xi$ of $\xi$.  Since
$f$ is $H$-equivariant, $f$ maps $H\xi$ onto $\CO$.  Hence $H\xi$ is irreducible 
with $\dim H\xi = \dim \CU_{\CO}$. In particular $H\xi$ is open dense in $\CU_{\CO}$.
If we take another $\xi' \in \CU_{\CO}$, $H\xi'$ is also open dense in $\CU_{\CO}$.  Hence 
$H\xi \cap H\xi' \ne \emptyset$.  Thus $\CU_{\CO} = H\xi$ and (v) holds.
\par
Take $z, z_1 \in \CO \cap \pi_P\iv(\CO')$.  
Since $(z, P^{\th}), (z_1, P^{\th})$ are both contained in $\CU_{\CO}$, they are in the same 
$H$-orbit by (v).  Hence there exists $g \in P^{\th}$ such that $gz = z_1$.      
This proves the first statement of (vi).  
Then for $z_1, z_2 \in \CO \cap \pi_P\iv(z')$, there exists $p \in P^{\th}$ such that
$pz_1 = z_2$.  But since $\pi_P$ is $P^{\th}$-equivariant, 
$p \in Z_{P^{\th}}(z') = Q^{\th}$.  This proves the second statement of (vi). 
\par
We show (i).  We may assume that $z \in \CO \cap \pi_P\iv(\CO')$. 
Then $P^{\th} \in \CP_{z, \CO'}$. 
Assume that $gP^{\th} \in \CP_{z, \CO'}$.  Then $(z, P^{\th}), (z, gP^{\th})$ 
are both contained in $\CU_{\CO}$. By (v) they are conjugate under $H$.
It follows that there exists $h \in Z_H(z)$ 
such that $gP^{\th} = hP^{\th}$.  But by (iv), we have $h \in P^{\th}$, and so 
$gP^{\th} = P^{\th}$.  This proves (i).  
\par
Finally we show (iii). Let $\CO_1$ be the $H$-orbit of $z_1$.  Then 
$\dim \CO_1 = c$.  By a similar argument as in the proof of (6.3.1), we have
$\dim \CP_{z_1, \CO'} = 0$.   Then a similar argument as in the proof of (iv) 
implies that $Z_H(z_1) = Z_{Q^{\th}}(z_1)$.  Hence 
$\dim Z_{Q^{\th}}(z) = \dim Z_{Q^{\th}}(z_1)$.  This shows that the $Q^{\th}$-orbit
of $z_1$ in $\pi_P\iv(z')$ has the same dimension as the $Q^{\th}$-orbit of $z$.  
Since both orbits are open dense in $\pi_P\iv(z')$, they have an intersection.  Hence 
$\CO = \CO_1$.  This proves (iii).  The proposition is proved. 
\end{proof}

\para{6.4.}
Let $P$ be as in 6.2.
For
$\Bla \in \CP(\Bm)$, we define a subset $\CM_{\Bla}$ of 
$P^{\io\th}\uni \times (\prod_{i=1}^{r-2}M_{p_i} \times M_{p_{r-2}}^{\perp})$ as 
the set of $(x, \Bv)$ satisfying the following properties; 
take $(x, \Bv)$ such that $x \in P^{\io\th}\uni$ and that $v_i \in M_{p_i}$ for 
$ i = 1, \dots, r-2$, $v_{r-1} \in M_{p_{r-2}}^{\perp}$.  
Put $\ol M_{p_i} = M_{p_i}/M_{p_{i-1}}$, and let $\ol v_i \in \ol M_{p_i}$ be the 
image of $v_i \in M_{p_i}$  for $i = 1, \dots, r-2$.  
We also put $\ol M'_{p_{r-1}} = M_{p_{r-2}}^{\perp}/M_{p_{r-2}}$.  Thus 
$\ol M_{p_{r-1}}'$ has a structure of a symplectic vector space, and one can 
consider an exotic symmetric space $GL(\ol M_{p_{r-1}}')^{\io\th} \times \ol M_{p_{r-1}}'$.
We also consider enhanced spaces $GL(\ol M_{p_i}) \times \ol M_{p_i}$ for $i = 1, \dots, r-2$.
Let $\ol v_{r-1}$ be the image of $v_{r-1}$ on $\ol M'_{p_{r-1}}$.
We assume that 
the $GL(\ol M_{p_i})$-orbit of $(x|_{\ol M_{p_i}}, \ol v_i)$ has type 
$(\la^{(i)}, \emptyset)$ for $i = 1, \dots, r-2$, and that the 
$GL(\ol M'_{p_{r-1}})^{\th}$-orbit of $(x|_{\ol M'_{p_{r-1}}}, \ol v_{r-1})$ has 
type $(\la^{(r-1)}, \la^{(r)})$.     
Let $\CO$ be the $H$-orbit of $z = (x, v_{r-1}) \in G^{\io\th}\uni \times V$ and 
$\CO'$ the $L^{\th}$-orbit of $z' = \pi_P(z)$. Note that $\CO'$ is $P^{\th}$-stable.
We further assume that 
\par\medskip\noindent
(6.4.1) \ $\CO \cap \pi_P\iv(\CO')$ is open dense in $\pi_P\iv(\CO')$. 
\par\medskip
We define $X_{\Bla}$ by $X_{\Bla} = \bigcup_{g \in H}g\CM_{\Bla}$. 
Note that in the case where $r = 2$, $X_{\Bla}$ coincides with the $H$-orbit 
$\CO_{\Bla}$ in [SS1]. 
\par
It follows from the construction that $\CM_{\Bla}$ is a $P^{\th}$-stable subset 
of $P^{\io\th} \times V^{r-1}$.  
The closure $\ol \CM_{\Bla}$ of $\CM_{\Bla}$ in $P^{\io\th} \times V^{r-1}$ is 
also $P^{\th}$-stable.  
We define a variety $\wt \CF_{\Bla}$ by 
$\wt \CF_{\Bla} = H \times^{P^{\th}}\ol\CM_{\Bla}$.
Let $\pi_{\Bla} : \wt \CF_{\Bla} \to \CX\uni$ be the map induced from the map
$H \times \ol\CM_{\Bla} \to \CX\uni, (g, z) \mapsto g\cdot z$.  Then 
$\pi_{\Bla}$ is a proper map.  
We define a subset $\wt \CF_{\Bla}^0$ of $\wt \CF_{\Bla}$ by 
$\wt \CF_{\Bla}^0 = H \times^{P^{\th}}\CM_{\Bla}$, and let 
$\pi_{\Bla}^0$ be the restriction of $\pi_{\Bla}$ on $\wt \CF_{\Bla}^0$.  
It is clear that $\pi^0_{\Bla}$ is a surjective map onto $X_{\Bla}$.    

\para{6.5.}
Take $x_i \in GL(\ol M_{p_i})\uni$ with Jordan type $\la^{(i)}$ for 
$1 \le i \le r-2$.  Then the set of $\ol v_i \in \ol M_{p_i}$ such that 
$(x_i, \ol v_i)$ is of type $(\la^{(i)}, \emptyset$) is open dense in $\ol M_{p_i}$.
Take $(x_{r-1}, \ol v_{r-1}) \in GL(\ol M_{p_{r-1}}')^{\io\th}\uni \times \ol M_{p_{r-1}}'$
with type $(\la^{(r-1)}, \la^{(r)})$.   
Now $(x_1, \dots, x_{r-1})$ determines a unique element in $L^{\io\th}\uni$ which we denote by 
$x'$. Put $z' = (x', \ol v_{r-1})$.  
The set of $(x, v_{r-1}) \in G^{\io\th}\uni \times V$ such that 
$(x, v_{r-1}) \in \pi_P\iv(z')$ satisfying the condition (6.4.1)
is open dense in $\pi_P\iv(z')$. In fact, it coincides with 
$\CO \cap \pi_P\iv(z')$ for an $H$-orbit $\CO$.
Put $\Bv' = (\ol v_1, \dots, \ol v_{r-1})$.  For a fixed 
$(x', \Bv')$, let $\CM_{(x', \Bv')}$ 
be the set of $(x, \Bv)$ with $\Bv = (v_1, \dots, v_{r-1})$ 
such that $(x, v_{r-1}) \in \CO \cap \pi_P\iv(z')$ as in (6.4.1)
and that the image of $v_i$ on $\ol M_{p_i}$ coincides with 
$\ol v_i$.  As $\CM_{(x', \Bv')}$ is an open dense subset of 
$\prod_{i =1}^{r-3} M_{p_i} \times \pi_P\iv(z')$, we have
\begin{equation*}
\dim \CM_{(x', \Bv')} = \sum_{i=1}^{r-3}\dim M_{p_i} + \dim \pi_P\iv(z').  
\end{equation*}
It follows that $\CM_{(x', \Bv')}$ is smooth irreducible, and
\begin{equation*}
\tag{6.5.1}
\dim \CM_{(x',\Bv')} = \sum_{i=1}^{r-3}(r-i-2)m_i + \dim U_P^{\th}.
\end{equation*}
(Note $\dim \pi_P\iv(z') = \dim U_P^{\th}$ by (6.3.2).)
Since 
\begin{equation*}
\CM_{\Bla} = \bigcup_{g \in L^{\th}}g\CM_{(x', \Bv')} 
  \simeq L^{\th} \times^{Z_{L^{\th}}(x',\Bv')}\CM_{(x',\Bv')},
\end{equation*} 
$\CM_{\Bla}$ is smooth, irreducible, and open dense in $\ol \CM_{\Bla}$.  Moreover
\begin{equation*}
\tag{6.5.2}
\dim \CM_{\Bla} = \dim \CO_{(x', \Bv')} + \dim \CM_{(x',\Bv')}, 
\end{equation*}
where 
$\CO_{(x',\Bv')}$ is the $L^{\th}$-orbit of 
$(x',\Bv') \in L^{\io\th}\uni \times (\prod_{i=1}^{r-2}\ol M_{p_i} 
        \times \ol M'_{p_{r-1}})$.  
We have $\dim \CO_{(x',\Bv')} = \sum_{i = 1}^{r-1}\dim \CO'_i$, where 
$\CO'_i$ is the $GL(\ol M_{p_i})$-orbit of $(x_i, \ol v_i)$ for $i = 1, \dots, r-2$, 
and $\CO'_{r-1}$ is the $GL(\ol M_{p_{r-1}}')^{\th}$-orbit of $(x_{r-1}, \ol v_{r-1})$.
 Hence by [AH, Prop. 2.8] and [SS1, Lemma 2.3],
\begin{equation*}
\dim \CO'_i = \begin{cases}
                 m_i^2 - 2n(\la^{(i)}) 
                    &\quad\text{ for } 1 \le i \le r-2, \\
                 2{m'}_{r-1}^2 - 2{m'}_{r-1} - 4n(\la^{(r-1)} + \la^{(r)}) + 2|\la^{(r-1)}|
                    &\quad\text{ for } i = r-1,
              \end{cases}
\end{equation*}
where $m_{r-1}' = m_{r-1} + m_r$.  This implies that
\begin{equation*}
\tag{6.5.3}
\begin{split}
\dim \CO_{(x',\Bv')} = \sum_{i=1}^{r-2}m_i^2 &+ 2{m'}^2_{r-1} - 2{m'}_{r-1}  \\ 
                  &- 2n(\Bla) -2n(\la^{(r-1)} + \la^{(r)}) + 2|\la^{(r-1)}|.  
\end{split}
\end{equation*} 
We have the following lemma.
\begin{lem}  %%%%  Lemma 6.6.
$\wt \CF_{\Bla}$ is an irreducible variety 
 with 
\begin{equation*}
\dim \wt \CF_{\Bla} = 2n^2 - 2n - 2n(\Bla) - 2n(\la^{(r-1)} + \la^{(r)})
                     + \sum_{i=1}^{r-1}(r-i+1)|\la^{(i)}|.
\end{equation*}
$\wt \CF^0_{\Bla}$ is a smooth and open dense subvareity of $\wt \CF_{\Bla}$.
\end{lem}

\begin{proof}
Since $\wt \CF_{\Bla} = H \times^{P^{\th}}\ol \CM_{\Bla}$, and $\ol\CM_{\Bla}$ is 
irreducible, $\wt \CF_{\Bla}$ is irreducible.  
Since $\wt \CF_{\Bla}^0 = H \times^{P^{\th}}\CM_{\Bla}$, and $\CM_{\Bla}$ is 
smooth and open dense in $\ol\CM_{\Bla}$, $\wt \CF_{\Bla}^0$ is smooth and 
open dense in $\wt \CF_{\Bla}$.  
We have
$\dim \wt \CF_{\Bla} = \dim H/P^{\th} + \dim \CM_{\Bla}
                   = \dim U_P^{\th} + \dim \CM_{\Bla}$.
Here
\begin{align*}
2 \dim U_P^{\th} &= \dim H - \dim L^{\th}  \\
                 &= (2n^2 + n) - (\sum_{i=1}^{r-2}m_i^2 + 2{m'}^2_{r-1} + m'_{r-1}). 
\end{align*} 
Then by (6.5.1) $\sim$ (6.5.3), we see that
\begin{equation*}
\begin{split}
\dim \wt \CF_{\Bla} = 2n^2 + n - 3m'_{r-1} - &2n(\Bla) - 2n(\la^{(r-1)} + \la^{(r)}) \\
                     &+ \sum_{i=1}^{r-3}(r-i - 2)m_i + 2|\la^{(r-1)}|.  
\end{split}
\end{equation*} 
The lemma follows from  this if we note that 
$m _i = |\la^{(i)}|$ for $i = 1, \dots, r$ 
and $m'_{r-1} = m_{r-1} + m_r$. 
\end{proof}

\begin{prop}  %%%%  Prop. 6.7.
Assume that $\Bla \in \CP_{n,r}$.  
\begin{enumerate}
\item
$\Im \pi_{\Bla} = \ol X_{\Bla}$.
\item
$\pi^0_{\Bla} : \wt \CF^0_{\Bla} \to X_{\Bla}$ gives an isomorphism 
$\wt \CF^0_{\Bla}\isom X_{\Bla}$.
\item
$X_{\Bla}$ is a smooth, irreducible, and  
\begin{equation*}
\dim X_{\Bla} = 2n^2 - 2n - 2n(\Bla) - 2n(\la^{(r-1)} + \la^{(r)}) 
          + \sum_{i=1}^{r-1}(r-i+1)|\la^{(i)}|.
\end{equation*} 
\item
$X_{\Bla}$ is a locally closed subvariety of $\CX\uni$.
\end{enumerate}
\end{prop}

\begin{proof}
$\pi^0_{\Bla} : \wt \CF^0_{\Bla} \to X_{\Bla}$ is $H$-equivariant, 
and surjective. Take $(x, \Bv) \in \CM_{\Bla}$ and 
put $z = (x, v_{r-1})$, $z' = \pi_P(z)$.  Let $\CO'$ be 
the $L^{\th}$-orbit of $z'$.  
Then 
\begin{equation*}
(\pi^0_{\Bla})\iv(x, \Bv) \simeq \{ gP^{\th} \in H/P^{\th} 
     \mid g\iv(x, \Bv) \in \CM_{\Bla} \} \subset \CP_{z, \CO'},
\end{equation*} 
where $\CP_{z, \CO'}$ is as in 6.2.
By Proposition 6.3 (i), $\CP_{z, \CO'} = \{ P^{\th} \}$.
By the definition of $\wt \CF_{\Bla}$, $(\pi^0_{\Bla})\iv(x, \Bv)$ 
contains $P^{\th}$.  Hence we see that $(\pi^0_{\Bla})\iv(x, \Bv) = \{ P^{\th} \}$.
It follows that $\pi^0_{\Bla}$ is a bijective morphism.
The correspondence $g(x, \Bv) \mapsto gP^{\th}$ gives a well-defined morphism 
from  the $H$-orbit of $(x, \Bv) \in \CM_{\Bla}$ to $H/P^{\th}$.  This induces 
a morphism from $X_{\Bla}$ to $\wt \CF^0_{\Bla}$.  
It is easy to check that this gives an inverse of $\pi^0_{\Bla}$.  Hence 
$\wt \CF^0_{\Bla} \simeq X_{\Bla}$ and (ii) holds.  (iii) is immediate from 
(ii) and Lemma 6.6.
By (ii) $\Im \pi_{\Bla}$ contains $X_{\Bla}$.  Since $\pi_{\Bla}$ is proper, 
$\Im \pi_{\Bla}$ is closed, hence $\Im \pi_{\Bla}$ contains $\ol X_{\Bla}$.  
We have $\dim \wt \CF_{\Bla} = \dim X_{\Bla}$ by (iii), and $\Im \pi_{\Bla}$, 
$X_{\Bla}$ are irreducible. Thus (i) holds. 
Here $\wt \CF_{\Bla} \backslash \wt \CF_{\Bla}^0$ is a closed subset 
of $\wt \CF_{\Bla}$, and its image by $\pi_{\Bla}$ coincides with 
$\ol X_{\Bla} \backslash X_{\Bla}$.  Hence $\ol X_{\Bla} \backslash X_{\Bla}$ 
is closed.  This implies that $X _{\Bla}$ is open dense in $\ol X_{\Bla}$, and 
so (iv) holds.  The proposition is proved.
\end{proof}

\para{6.8}
Assume that $\Bm \in \CQ_{n,r}^0$.  We define a set $\wt\CP(\Bm)$ by 
\begin{equation*}
\tag{6.8.1}
\wt\CP(\Bm) = \coprod_{0 \le k \le m_{r-1}}\CP(\Bm(k)),
\end{equation*}
(see 1.6 for the definition of $\Bm(k)$).
For $\Bm \in \CQ_{n,r}$, let $\Bla(\Bm) \in \CP_{n,r}$ as in 5.8.
Let $\ol X_{\Bla}$ be the closure of $X_{\Bla}$ in $\CX\uni$.
As an analogue of Proposition 5.9, we prove the following.

\begin{prop}  %%%  Prop. 6.9.  
Assume that $\Bm \in \CQ_{n,r}^0$. 
\begin{enumerate}
\item
$\dim \wt\CX_{\Bm,\unip} = \dim \CX_{\Bm, \unip}$.
\item
$\CX_{\Bm, \unip} = \ol X_{\Bla(\Bm)}$.  
\item
Assume that $\Bmu \in \wt\CP(\Bm)$. 
Then $X_{\Bmu} \subset \CX_{\Bm,\unip}$.
\end{enumerate}
\end{prop}

\begin{proof}
We consider the surjective map $\pi_1^{(\Bm)}: \wt\CX_{\Bm, \unip} \to \CX_{\Bm, \unip}$.
Take $z = (x, \Bv) \in \CM_{\Bmu}$ for $\Bmu \in \CP(\Bm(k))$. 
Then it follows from  the construction of $\CM_{\Bmu}$ that 
$(z, B^{\th}) \in \wt X_{\Bm, \unip}$.  Hence $z \in X_{\Bm, \unip}$. 
Since $X_{\Bm \unip}$ is 
$H$-stable, we see that $X_{\Bmu} \subset \CX_{\Bm, \unip}$.  This proves 
(iii).   Put $\Bmu = \Bla(\Bm)$.  By (iii) we have 
$X_{\Bmu} \subset \CX_{\Bm, \unip}$.  
By (6.1.1), 
\begin{equation*}
\dim \wt \CX_{\Bm \unip} = 2n^2 - n + \sum_{i=1}^{r-1}(r-i)m_i.
\end{equation*}
On the other hand, by Proposition 6.7, we have
\begin{align*}
\dim X_{\Bmu} &= 2n^2 - 2n  + \sum_{i=1}^{r-1}(r-i+1)m_i  \\
              &=   2n^2 - n + \sum_{i=1}^{r-1}(r-i)m_i.
\end{align*}
(Note that $\sum_{i=1}^{r-1}m_i = n$ since $m_r = 0$.)
Since $\ol X_{\Bmu}$ is a closed, irreducible subset of $\CX_{\Bm \unip}$ with 
$\dim \ol X_{\Bla} = \dim \wt\CX_{\Bm \unip}$,  we conclude that 
$\dim \wt\CX_{\Bm\unip} = \dim \CX_{\Bm, \unip}$ and 
 $\ol X_{\Bmu} = \CX_{\Bm, \unip}$.  This proves (i) and (ii).
\end{proof}

\para{6.10.}
For a fixed $\Bm \in \CQ_{n,r}$, let $P$ be as in 6.2.  
For $\Bla \in \CP(\Bm)$, we consider the $L^{\th}$-orbit $\CO'$
and $H$-orbit $\CO$ as in 6.4. Since $\CO'$ is determined by $\Bla$, we
denote it by $\CO' = \CO'_{\Bla}$.  We shall determine the type of $\CO$.
For a partition 
$\nu = (\underbrace{a, \dots, a}_{k\text{-times}}) = (a^k) \in \CP_{ak}$ of rectangular 
type, we define a double 
partition $[\nu] \in \CP_{ak,2}$ by 
$[\nu] = (a^{k'}, a^{k''})$, where $k' = [k+1/2], k'' = k - k'$.
In general, a partition $\nu$ can be decomposed uniquely as a sum of 
partitions of rectangular type 
\begin{equation*}
\tag{6.10.1}
\nu = (a_1)^{k_1} + (a_2)^{k_2} + \cdots + (a_{\ell})^{k_{\ell}}, 
\end{equation*}
where we write the dual partition $\la^*$ as 
$\la^* = ((k_1)^{a_1}, (k_2)^{a_2}, \dots, (k_{\ell})^{a_{\ell}})$ 
for $k_1 > k_2 > \cdots > k_{\ell} > 0$.  
We define $[\nu] \in \CP_{|\nu|,2}$ by 
\begin{equation*}
\tag{6.10.2}
[\nu] = [(a_1)^{k_1}] + [(a_2)^{k_2}] + \cdots + [(a_{\ell})^{k_{\ell}}].
\end{equation*}
For a given $\Bla \in \CP_{n,r}$, we define $[\Bla] \in \CP_{n,2}$ by 
\begin{equation*}
[\Bla] = [\la^{(1)}] + [\la^{(2)}] + \cdots + [\la^{(r-2)}] + (\la^{(r-1)}, \la^{(r)}).
\end{equation*}
We have the following lemma.

\begin{lem}  %%%% Lemma 6.11.
Let $\CO$ be a unique $H$-orbit in $G^{\io\th} \times V$ such that 
$\CO \cap \pi_P\iv(\CO'_{\Bla})$ is open dense in $\pi_P\iv(\CO'_{\Bla})$.  
Then $\CO = \CO_{[\Bla]}$. 
\end{lem}

\begin{proof}
We show that $\CO_{[\Bla]}$ satisifes the required condition.
Take $z \in \CO_{[\Bla]}$, and $z' \in \CO'_{\Bla}$.  
First we show
that 
\begin{equation*}
\tag{6.11.1}
\dim Z_H(z) = \dim Z_{L^{\th}}(z').
\end{equation*} 
In fact, $\dim \CO'_{\Bla} = \sum_{i=1}^{r-2}\dim \CO'_{\la^{(i)}} + \dim \CO'_0$, 
where $\CO'_{\la^{(i)}}$ (resp. $\CO'_0$) is the $GL(\ol M_{p_i})$-orbit of type 
$\la^{(i)}$ 
(resp. $GL(\ol M'_{p_{r-1}})^{\th}$-orbit of type $(\la^{(r-1)}, \la^{(r)})$)
 corresponding to $\CO'_{\Bla}$.  We have
\begin{align*}
\dim \CO'_{\la^{(i)}} &= m_i^2 - m_i - 2n(\la^{(i)}) \text{ for }1 \le i \le r-2, \\
\dim \CO'_0 &= 2{m'}^2_{r-1} - 2m'_{r-1} - 4n(\la^{(r-1)} + \la^{(r)}) + 2|\la^{(r-1)}|. 
\end{align*} 
(The first formula is well-known, see, e.g. [AH, Prop. 2.8 (4)].  The second one 
is by [SS1, Lemma 2.3]. )
Hence
\begin{equation*}
\tag{6.11.2}
\dim Z_{L^{\th}}(z') = \sum_{i=1}^{r-2}(m_i + 2n(\la^{(i)})) 
                        + 3m'_{r-1} + 4n(\la^{(r-1)} + \la^{(r)}) - 2|\la^{(r-1)}|.
\end{equation*}
On the other hand, if we write $[\Bla] = (\nu', \nu'')$, 
$\dim \CO_{[\Bla]}$ is given by a similar formula as $\dim \CO'_0$.  Hence  
we have
\begin{equation*}
\tag{6.11.3}
\dim Z_H(z) = 3n + 4n([\Bla]) - 2n (\nu').
\end{equation*}
Here, for $i = 1, \dots, r-2$, we show the formula 
\begin{equation*}
\tag{6.11.4}
m_i + 2n(\la^{(i)}) = 3m_i + 4n([\la^{(i)}]) - 2|{\nu_{(i)}}'|,
\end{equation*}
where we write $[\la^{(i)}]$ as $({\nu_{(i)}}', {\nu_{(i)}}'')$.
By the additivitiy of $n$-function together with (6.10.1) and (6.10.2), 
it is enough to verify (6.11.4) in the case where 
$\la^{(i)}$ is of rectangular type.  
Assume that $\la^{(i)} = (a^{2k-1})$. 
In this case, $m_i = a(2k-1), |\nu_{(i)}'| = ak, 
n(\la^{(i)}) = a(k-1)(2k-1)$ and $n([\la^{(i)}]) = a(k-1)^2$.
Thus (6.11.4) holds.  The case where $\la^{(i)} = (a^{2k})$ is dealt similarly. 
Now by substituting (6.11.4) into the formula in (6.11.2), and by comparing it 
with (6.11.3), we obtain (6.11.1). 
\par
In order to show the lemma, it is enough to see that 
$\CO_{[\Bla]} \cap \pi_P\iv(\CO'_{\Bla}) \ne \emptyset$ by Proposition 6.3 (iii).
First we show 
\par\medskip\noindent
(6.11.5) \  Assume that $r = 2$ and $m_1 = n, m_2 = 0$.  
Then  $\CO_{[\Bla]} \cap \pi_P\iv(\CO'_{\Bla}) \ne \emptyset$.
\par\medskip\noindent
In this case $\Bla = (\la^{(1)}, \emptyset)$, $[\Bla] = [\la^{(1)}]$. 
We consider the variety $G\uni \times V$, and we denote by $\OO_{\Bxi}$ 
the $G$-orbit in $G\uni \times V$ corresponding to $\Bxi \in \CP_{2n,2}$. 
The map $\pi_P$ is naturally extended to a map 
$\wt\pi_P : P\uni \times V \to L\uni \times V_L$.
Then by [AH, Cor. 3.4] (see Cor. 5.17 in the case where $r = 2$),  
we have
\par\medskip\noindent
(6.11.6) \ 
$\wt\pi_P\iv(\CO'_{\Bla}) \cap \OO_{\Bxi} \neq \emptyset$ if and only if 
               $\Bxi \le (\la^{(1)}, \la^{(1)})$ and 
$\Bxi \not\le \Bmu$ for any $\Bmu = (\mu', \mu'') \ne (\la^{(1)}, \la^{(1)})$ such that 
$\mu' \le \la^{(1)}, \mu'' \le \la^{(1)}$. 
\par\medskip
Moreover, by [AH, Th. 6.1], $\OO_{\Bxi} \cap (G^{\io\th}\uni\times V)  \neq \emptyset$
only when $\Bxi$ is of the form  $\Bxi = (\xi'\cup \xi', \xi''\cup \xi'')$, and in that
case $\OO_{\Bxi} \cap (G^{\io\th}\uni \times V) = \CO_{(\xi', \xi'')}$.  
Write $[\Bla] = (\nu', \nu'')$, and consider $\Bxi = (\nu'\cup \nu', \nu''\cup \nu'')$.
We show that $\Bxi$ satisfies the condition in (6.11.6).
By the decompostion (6.10.2), we may assume that $\la^{(1)}$ is of rectangular type.
So we assume that $\la^{(1)} = (a^{2k-1})$.  In that case,
$(\la^{(1)}, \la^{(1)}) = ((a^{2k-1}), (a^{2k-1}))$ and 
$\Bxi = ((a^{2k}), (a^{2k-2}))$.  The condition for (6.11.6) is easily verified. 
The case where $\la^{(1)} = (a^{2k})$ is trivial since $\Bxi = (\la^{(1)}, \la^{(1)})$. 
Thus (6.11.5) holds.
\par
Next we consider the general case.  
Let $V = V_1 \oplus \cdots \oplus V_{r-1}$ be the decomposition of $V$ into subspaces, 
where $V_i$ is a symplectic subspace spanned by 
$\{e_i, f_i \mid  p_{i-1}+1 \le i \le p_i\}$ for $i =1, \dots, r-2$,  
and $V_{r-1}$ is a symplectic subspace sppaned by 
$\{e_i, f_i \mid p_{r-2}+1 \le i \le n \}$. 
We consider a $\th$-stable subgroup $\wt L$ of 
$G$ containing $L$ such that 
$\wt L \simeq G_1 \times \cdots \times G_{r-1}$ with $G_i = GL(V_i)$.
Hence $\wt L^{\th} \simeq G_1^{\th} \times \cdots \times G_{r-1}^{\th}$ 
with $G_i^{\th} = Sp(V_i)$. 
For $i = 1, \dots, r-2$, 
let $P_i$ be a $\th$-stable parabolic subgroup of  $G_i$ such that its 
$\th$-stable Levi
subgroup $L_i$ is isomorphic to $GL_{m_i} \times GL_{m_i}$.  
Let $\CO'_{\la^{(i)}}$ be the $GL_{m_i}$-orbit 
in $(GL_{m_i})\uni$ as before.  We regard it as a $L_i^{\th}$-orbit in 
$(G_i)\uni^{\io\th}$.  
Let $\CO_{[\la^{(i)}]}$ be the $G_i^{\th}$-orbit in $(G_i)\uni^{\io\th} \times V_i$
corresponding to $[\la^{(i)}]$.  Then by (6.11.5), one can find 
$z_i \in \CO_{[\la^{(i)}]} \cap \pi_{P_i}\iv(\CO'_{\la^{(i)}})$.
Let $\CO'_0$ be the $G_{r-1}^{\th}$-orbit in $(G_{r-1})\uni^{\io\th}$ corresponding to
$(\la^{(r-1)}, \la^{(r)})$.  We choose $z_{r-1} \in \CO_0'$. 
Since $\prod_{i=1}^{r-2}\pi_{P_i}\iv(\CO'_{\la^{(i)}}) \times \CO'_{r-1}$ 
is regarded as a closed subvariety of $\pi_P\iv(\CO'_{\Bla})$,
$z = (z_1, \dots, z_{r-1})$ gives an element in $\pi_P\iv(\CO'_{\Bla})$.    
$z$ is contained in $\wt L^{\io\th} \times V$, and actually $z \in \CO_{[\Bla]}$.
Thus $\CO_{[\Bla]} \cap \pi_P\iv(\CO'_{\Bla}) \ne \emptyset$. 
The lemma is proved. 
\end{proof}

\begin{cor}  %%%%  Cor. 6.12.
For each $\Bla \in \CP_{n,r}$, the map 
$(x, \Bv) \mapsto ((x, v_{r-1}), (v_1, \dots, v_{r-2}))$
gives an embedding $X_{\Bla} \subset \CO_{[\Bla]} \times V^{r-2}$.
\end{cor}

\remark{6.13.}
Proposition 6.9 shows that $\bigcup_{\Bla \in \CP_{n,r}}X_{\Bla}$ covers a dense 
subset of $\CX\uni$.  However it does not coincide with $\CX\uni$ in general.
In the case where $n = 2, r= 3$, one can show by a direct computation 
that $X_{\Bla}$ are mutually disjoint for $\Bla \in \CP_{n,r}$.  It is 
not known whether the disjointness property holds for $X_{\Bla}$ in general.  

\para{6.14.}
The unipotent variety of enhanced type considered in Section 5 can be interpreted 
as a closed subvariety of the variety defined in 1.2.  So we follow the setting 
in 1.2, and consider the varieties $\wt\CX_{\Bm}, \CX_{\Bm}$ as in Section 4.
Then the vareities $\wt\CX_{\Bm, \unip}, \CX_{\Bm, \unip}$ defined in 5.2 can be identified 
with closed subvarieties of $\wt\CX_{\Bm}, \CX_{\Bm}$, defined by similar formulas 
as in 6.1. We shall reformulate $X_{\Bla}$ defined in Section 5 so that it fits to the
discussion in the exotic case.  
\par
First we consider an analogue of Propostion 6.3 in the enhanced case. 
Assume that $r \ge 2$, and take $\Bm \in \CQ_{n,r}$.  Let $P$ be a $\th$-stable
parabolic subgroup of $G$ such that $P^{\th}$ is the stabilizer of the flag 
$(M_{p_i})_{1 \le i \le r}$ in $V$. Let $L$ be the $\th$-stble Levi subgroup of $P$
containing $T$.  Contrast to 6.2, we put $V_L = V$. We consider the map 
$\pi_P: P^{\io\th}\uni \times V \to L^{\io\th}\uni \times V_L$, 
$(x, v) \mapsto (x',v)$, where $x' = p(x)$ for the natural projection 
$p: P^{\io\th}\uni \to L^{\io\th}\uni$.    
We define $\CO, \CO'$ and $\CP_{z, \CO'}$ as in 6.2. 
The following result can be proved word by word following the proof of 
Proposition 6.3.  Note that in the enhanced case, $\dim U_P^{\th} = \dim U_P^{\io\th}$. 

\begin{prop}  %%%%  Prop. 6.15.
Assume that $\CX\uni = G^{\io\th}\uni \times V$ is of enhaned type.  Then similar statements
(i) $\sim$ (vi) as in Proposition 6.3 hold for $\CX\uni$.    
\end{prop}

\para{6.16.}
Let $P$ be as in 6.14.  For $\Bla \in \CP(\Bm)$, by imitating the definition of 
$\CM_{\Bla}$ in 6.4, we define a subset 
$\CM_{\Bla}$ of $P^{\io\th}\uni \times \prod_{i=1}^{r-1}M_{p_i}$ as the set of 
$(x, \Bv)$ satisfying the following properties; 
take $(x, \Bv)$ such that $x \in P^{\io\th}\uni$ and that $v_i \in M_{p_i}$ for 
$i = 1, \dots, r-1$.  Put $\ol M_{p_i} = M_{p_i}/M_{p_{i-1}}$, and let 
$\ol v_i \in \ol M_{p_i}$ be the image of $v_i$ for each $i$. We assume that 
$GL(\ol M_{p_i})$-orbit of $(x|_{\ol M_{p_i}}, \ol v_i)$ in 
$GL(\ol M_{p_i})^{\io\th} \times \ol M_{p_i}$ has type $(\la^{(i)}, \emptyset)$ 
for each $i$. We further assume that 
\par\medskip\noindent
(6.16.1) \ $\CO \cap p\iv(\CO')$ is open dense in $p\iv(\CO')$, where 
$\CO$ is the $H$-orbit of $x$ in $G^{\io\th}\uni$ and $\CO'$ is the $L^{\th}$-orbit
of $x'= p(x) \in L^{\io\th}\uni$.  
\par\medskip
We define $X_{\Bla}'$ by $X_{\Bla}' = \bigcup_{g \in H}g\CM_{\Bla}$. $\CM_{\Bla}$
is a $P^{\th}$-stable subset of $P^{\io\th} \times V^{r-1}$.  
We have a lemma.

\begin{lem}  %%%% Lemma 6.17
$X_{\Bla}'$ coincides with $X_{\Bla}$ for each $\Bla \in \CP_{n,r}$. 
In particular, $X'_{\Bla}$'s are mutually disjoint. 
\end{lem}

\begin{proof}
Assume that $\CO' = \CO'_{\Bla}$.  Then the $H$-orbit $\CO$ in (6.16.1) is
given by $\CO = \CO_{\nu}$ with $\nu = \la^{(1)} + \cdots + \la^{(r)}$. 
The lemma then follows from Corollary 5.17 (ii).   
\end{proof}

\para{6.18.}
As in the exotic case, we defnie a variety $\wt\CF_{\Bla} = H \times^{P^{\th}}\ol\CM_{\Bla}$,
where $\ol \CM_{\Bla}$ is the closure of $\CM_{\Bla}$ in $P^{\io\th}\times V^{r-1}$.
Then the map $\pi'_{\Bla}: \wt\CF_{\Bla} \to X'_{\Bla} = X_{\Bla}$
is defined as in 6.4. (Here we use the notation $\pi'_{\Bla}$ to distinguish it from  
the map $\pi_{\Bla} : \wt X_{\Bla} \to X_{\Bla}$ in 5.5.) 
We also define a subset $\wt\CF_{\Bla}^0$ of $\wt\CF_{\Bla}$ by 
$\wt\CF^0_{\Bla} = H\times^{P^{\th}}\CM_{\Bla}$, and let $\pi'^0_{\Bla}$ be the 
restriciton of $\pi'_{\Bla}$ on $\wt\CF^0_{\Bla}$.  
The following result can be proved by a similar way as Proposition 6.7.

\begin{prop}  %%%%  Prop. 6.19
Assume that $\CX$ is of enhanced type.  For each $\Bla \in \CP_{n,r}$, 
\begin{enumerate}
\item
$\Im \pi_{\Bla}' = \ol X_{\Bla}$.
\item
$\pi_{\Bla}'^0: \wt\CF^0_{\Bla} \to X_{\Bla}$ gives an isomorphism 
$\wt\CF^0_{\Bla} \isom X_{\Bla}$. 
\end{enumerate}
\end{prop}

\para{6.20.}
In the remainder of this section, we assume that $\CX$ is of extoic type or of 
enhanced type.  We follow the formulation in Section 1.  
Put $\CB = H/B^{\th}$.  For each $z = (x,\Bv) \in \CX_{\Bm,\unip}$,
put
\begin{align*}
\CB_z &= \{ gB^{\th} \in \CB \mid g\iv xg \in U^{\io\th}, g\iv \Bv \in M_n^{r-1} \}, \\
\CB_z^{(\Bm)} &= \{ gB^{\th} \in \CB \mid 
    g\iv xg \in U^{\io\th}, g\iv \Bv \in \prod_{i=1}^{r-1}M_{p_i} \}. 
\end{align*}
In the exotic case, 
$\CB_z$ is a closed subvariety of $\CB$, which is 
isomorphic to the fibre $\pi_1\iv(z)$, and is called the Springer fibre of $z$. 
$\CB_z^{(\Bm)}$ is a closed subvariety of $\CB_z$ isomorphic to 
$(\pi_1^{(\Bm)})\iv(z)$, which we call the small Springer fibre.
In the enhanced case, we only consider the Springer fibre 
$\CB_z^{(\Bm)}$ as alreday defined in 5.18. 
\par
We fix $\Bm \in \CQ_{n,r}$. For an integer $d \ge 0$, 
we define a subset $X(d)$ of $\CX_{\Bm,\unip}$ by 
\begin{equation*}
\tag{6.20.1}
X(d) = \{ z \in \CX_{\Bm,\unip} 
                    \mid \dim \CB_z^{(\Bm)} = d \}.  
\end{equation*} 
Then $X(d)$ is a locally closed subvariety of $\CX_{\Bm,\unip}$, and 
$\CX_{\Bm, \unip} = \coprod_{d \ge 0}X(d)$. 
We consider the Steinberg varieties $\CZ^{(\Bm)}$ and 
$\CZ^{(\Bm)}_1$ as follows;
\begin{align*}
\begin{split}
\CZ^{(\Bm)} = \{ (z, gB^{\th}, &g'B^{\th}) \in \CX \times \CB \times \CB  \\
        &\mid (z, gB^{\th}) \in \wt\CX_{\Bm}, (z, g'B^{\th}) \in \wt\CX_{\Bm}\},  
\end{split}  \\  \\
\begin{split}
\CZ^{(\Bm)}_1 = \{ (z, gB^{\th}, &g'B^{\th}) \in \CX\uni \times \CB \times \CB  \\
        &\mid (z, gB^{\th}) \in \wt\CX_{\Bm, \unip}, (z, g'B^{\th}) \in \wt\CX_{\Bm,\unip}\}.  
\end{split}
\end{align*}
\par
Recall that $W = N_H(T^{\th})/T^{\th}$ is the Weyl group $W_n$ of type $C_n$ 
(resp. $S_n$) in the exotic case (resp. the enhanced case). 
In the exotic case, we assume that $\Bm \in \CQ_{n,r}^0$, while in the enhanced case, 
we consider an arbitrary $\Bm  \in \CQ_{n,r}$. 
We define a  subgroup  $\CW_{\Bm}\nat$ of $N_H(T^{\th})/T^{\th}$ by  
by 
\begin{equation*}
\tag{6.20.2}
\CW_{\Bm}\nat = \begin{cases}
       S_{m_1} \times \cdots \times S_{m-2} \times W_{m-1}
           &\quad\text{ (exotic case),} \\
       S_{m_1} \times \cdots \times S_{m_r}
            &\quad\text{ (enhanced case).}
               \end{cases}
\end{equation*} 
We show the following lemma.

\begin{lem} %%%  Lemma 6.21.                                                                         
Under the notation of 6.20, 
\begin{enumerate}                                                                                   
\item                                                                                               
$\dim \CZ^{(\Bm)}_1 = \dim \CX_{\Bm, \unip}$.  
The set of irreducible components of $\CZ^{(\Bm)}_1$     
with maximal dimension is parametrized by $w \in \CW_{\Bm}\nat$.                                        
\item
$\dim \CZ^{(\Bm)} = \dim \CZ^{(\Bm)}_1 + n$.  
The set of irreducible components of $\CZ^{(\Bm)}$     
with maximal dimension is parametrized by $w \in \CW_{\Bm}\nat$.                                        
\item                                                                                               
Assume that $X(d) \ne \emptyset$.  For any $z \in X(d)$ we have
\begin{equation*}
\dim \CB_z^{(\Bm)} \le 
        \frac{1}{2}(\dim \CX_{\Bm,\unip} - \dim X(d)).                                        
\end{equation*}
In particular, $\pi^{(\Bm)}$ is semismall with respect to the 
stratification $\CX_{\Bm,\unip} = \coprod_dX(d)$. 
\end{enumerate}                                                                                     
\end{lem}                                                                                           

\begin{proof}                                                                                       
Let $p_1: \CZ^{(\Bm)}_1 \to \CB \times \CB$ be the projection on secod and third            
factors.                                                                                            
For each $w \in W$, let $\CO_w$ be the $H$-orbit of $(B^{\th}, wB^{\th})$ in                      
$\CB \times \CB$.  We have                                                              
$\CB \times \CB = \coprod_{w \in W}\CO_w$.                                            
Put $Z_w = p_1\iv(\CO_w)$.  Then $Z_w$ is a vector bundle over                                        
$\CO_w \simeq H/(B^{\th} \cap wB^{\th}w\iv)$ with fibre isomorphic to                               
\begin{equation*}                                                                                   
(U^{\io\th} \cap wU^{\io\th}w\iv) \times \prod_{i=1}^{r-1}(M_{p_i} \cap w(M_{p_i})).                
\end{equation*}                    
First we consider the exotic case.                                                                  
We identify $W$ with a subgroup of $S_{2n}$ which is the stabilizer of                            
the element $(1,n+1)(2, n+2) \cdots (n, 2n)$.                                                       
Let $b_w$ be the number of $i$ such that $w\iv(e_i) \in M_n$, i.e.,                                 
$w\iv(i) \in [1, n]$.    
Then we have                                                                 
$\dim (U^{\io\th} \cap wU^{\io\th}w\iv) = \dim (U^{\th} \cap wU^{\th}w\iv) - b_w$.                  
Also we have $\dim (M_n \cap w(M_n)) = b_w$. 
By our assumption $M_{n-1} = M_n$, we have                                                          
\begin{equation*}                                                                                   
\dim Z_w = \dim H - \dim T^{\th} + \sum_{i=1}^{r-2}\dim (M_{p_i} \cap w(M_{p_i})).                  
\end{equation*}                                                                                     
Here $\dim (M_{p_i} \cap w(M_{p_i})) =                                                            
    \sharp\{ j \in [1,p_i] \mid w\iv(j) \in [1,p_i]\} \le p_i$, and                                 
the equality holds if and only if $w$ leaves the set $[1, p_i]$ invariant.                          
It follows that $\dim Z_w$ takes the maximal value                                                  
$2n^2 + \sum_{i=1}^{r-2}p_i$  if and only if $w \in \CW_{\Bm}\nat$.                                     
Since 
\begin{equation*}
\sum_{i=1}^{r-2}p_i = -n + \sum_{i=1}^{r-1}(r-i)m_i
\end{equation*}
(note that $m_r = 0$),               
in that case $\dim Z_w = \dim \CX_{\Bm, \unip}$ by Proposition 6.9 together with 
(6.1.1).                            
Since $\{ \ol Z_w \mid w \in W \}$ gives the set of irreducible components of                     
$\CZ^{(\Bm)}_1$, (i) follows.  
In the enhanced case, a similar computation shows that 
\begin{equation*}
\dim Z_w = \dim H  - \dim T^{\th} + \sum_{i=1}^{r-1}\dim (M_{p_i} \cap w(M_{p_i}))
\end{equation*}
for $w \in S_n$ (note that in this case $\dim (U^{\io\th} \cap wU^{\io\th}w\iv) 
    = \dim (U^{\th} \cap wU^{\th}w\iv)$), and $\dim Z_w$ takes the maximal value 
$n^2 - n + \sum_{i=1}^{r-1}p_i$ if and only if $w \in \CW_{\Bm}\nat$. 
Since $\sum_{i=1}^{r-1}p_i = \sum_{i=1}^{r-1}(r-i)m_i$, in this case 
$\dim Z_w = \dim \CX_{\Bm, \unip}$ by Proposition 5.9 (ii).   
Hence (i) holds also in the enhanced case.
\par
For (ii), we consider $\wt Z_w = p\iv(\CO_w)$, where 
$p: \CZ^{(\Bm)} \to \CB \times \CB$ is the projection on the 
second and third factors. 
Then $\wt Z_w$ is a locally trivial 
fibration over $\CO_w$ with fibre isomorphic to 
\begin{equation*}
\tag{6.21.1}                                                                                   
T^{\io\th} \times 
    (U^{\io\th} \cap wU^{\io\th}w\iv) \times \prod_{i=1}^{r-1}(M_{p_i} \cap w(M_{p_i})).                
\end{equation*}                                                                                     
Hence (ii) is proved by a similar argument as (i).
\par
We show (iii).  Let $q_1: \CZ_1^{(\Bm)} \to \CX_{\Bm,\unip}$ be the projection on 
the first factor.  For each $z \in \CX_{\Bm,\unip}$, 
$q_1\iv(z) \simeq \CB_z^{(\Bm)} \times \CB_z^{(\Bm)}$. 
By (6.20.1), we have 
\begin{equation*}
\dim q_1\iv(X(d)) = \dim X(d) + 2d.
\end{equation*}
Since $\dim q_1\iv(X(d)) \le \dim \CZ_1^{(\Bm)} = \dim \CX_{\Bm, \unip}$, 
we see that $2d \le \dim \CX_{\Bm,\unip} - \dim X(d)$.  
This proves (iii).  
The lemma is proved. 
\end{proof}

\par\bigskip
\section{Springer correspondence}
\para{7.1.}
In this section, we assume that $\CX$ is of exotic type or
of enhanced type.  We shall prove the Springer correspondence for $\CX\uni$. 
However the method used in [SS1], which is based on a evaluation of the number of irreducible 
components of Springer fibres (see [SS1, Lemma 2.5]), does not work well for the case where $r \ge 3$. 
Instead, we apply the method used by Lusztig [Lu] to prove the generalized Springer correspondence
for reductive groups, which makes use of the Steinberg map.  
We follow the notation in Section 1. In the exotic case, let 
$\w' : G \to T/S_{2n}$ be the Steinberg map 
with respect to $G$.  Then we have $\Xi = T^{\io\th}/S_n \hra T/S_{2n}$, where $S_n$ is embedded 
in $S_{2n}$ as a subgroup of the centralizer of $(1,n+1)(2, n+2) \cdots (n,2n)$. 
In the enhanced case, let $\w': G \to T/(S_n \times S_n)$ be the Steinberg map
and conisder $\Xi = T^{\io\th}/S_n \hra T/(S_n \times S_n)$.  In either case, we denote by 
$\w$ the map $G^{\io\th} \to \Xi$ induced from $\w'$.   
Take $\Bm \in \CQ_{n,r}$, and consider the map $\pi^{(\Bm)}: \wt\CX_{\Bm} \to \CX_{\Bm}$
as in 1.1.
Let $\CZ^{(\Bm)}$ be the Steinberg variety with respect to $\pi^{(\Bm)}$ defined in 6.20.
We denote by $\vf$ the natural map $\CZ^{(\Bm)} \to \CX_{\Bm}$. 
We define a map $\wt\a : \CZ^{(\Bm)} \to T^{\io\th}$ by 
$(x, \Bv, gB^{\th}, g'B^{\th}) \mapsto p_T(g\iv xg)$. 
Then we have a commutative diagram

\begin{equation*}
\begin{CD}
\CZ^{(\Bm)} @>\wt\a>>  T^{\io\th} \\
@V\vf VV                 @VV\w_T V     \\
\CX_{\Bm}  @>\wt\w >>  \Xi,
\end{CD}
\end{equation*}
where $\wt\w$ is the composite of the projection $\CX_{\Bm} \to G^{\io\th}$ with $\w$, 
and $\w_T$ is the restriction of $\w$ on $T^{\io\th}$.  Note that 
$\w_T$ is a finite morphism.
\par
As in 6.20, we assume that $\Bm \in \CQ_{n,r}^0$ in the exotic case, 
while we consider arbitrary $\Bm \in \CQ_{n,r}$ in the enhanced case. 
Recall that $d'_{\Bm} = \dim \CX_{\Bm,\unip}$, and put $\s = \w_T\circ \wt\a$. 
We define a constructible sheaf $\CT$ on 
$\Xi$ by 
\begin{equation*}
\tag{7.1.1}
\CT = \CH^{2d'_{\Bm}}(\s_!\Ql) = R^{2d'_{\Bm}}\s_!\Ql.
\end{equation*}
\par
Recall the definition of perfect sheaves in [Lu, 5.4].
A constructible sheaf $\CE$ on an irreducible variety $X$ is said to be
perfect if $\CE$ coincides with an intersection cohomology complex 
(reduced to a single sheaf on degeree zero) on $X$, and the support of 
any nonzero constructible subsheaf of $\CE$ is $X$.
\par
The following gives examples of perfect sheaves.
\par\medskip\noindent
(7.1.2) \ If $\pi: Y \to X$ is a finite morphism with $Y$ smooth and 
if $\CE'$ is a locally constant sheaf on $Y$, then $\CE = \pi_*\CE'$ 
is a perfect sheaf on $X$.
\par\medskip\noindent
(7.1.3) \ If $0 \to \CE_1 \to \CE_2 \to \CE_3 \to 0$ is an exact  sequence 
of constructible sheaves on $X$, with $\CE_1, \CE_3$ perfect, then $\CE_2$
is perfect.
\par\medskip

We show the following lemma.

\begin{lem}  %%%%  Lemma 7.2
The sheaf $\CT$ is a perfect sheaf on $\Xi$.
\end{lem}

\begin{proof}
Under the notation in the proof of Lemma 6.20, we consider $\wt Z_w = p\iv(\CO_w)$
for each $w \in W$. 
Let $\s_w$ be the restriction of $\s$ on $\wt Z_w$, and 
put $\CT_w = \CH^{2d'_{\Bm}}((\s_w)_!\Ql)$.  We also put $\a_w$ the restriction of 
$\wt\a$ on $\wt Z_w$.  
Let $\CW_{\Bm}\nat$ be the subgroup of $W$ given in (6.20.2).
We show that 
\begin{equation*}
\tag{7.2.1}
\CT_w \simeq \begin{cases}
                (\w_T)_!\Ql  &\quad\text{ if } w \in \CW_{\Bm}\nat \\
                 0           &\quad\text{ otherwise}.
              \end{cases}
\end{equation*}
In fact, since $\wt Z_w \to \CO_w$ is a locally trivial fibration with fibre 
isomorphic to (6.21.1), we see that $\a_w$ is a locally trivial fibration 
with fibre isomorphic to 
\begin{equation*}
\tag{7.2.2}
 H \times^{(B^{\th} \cap wB^{\th}w\iv)} 
    \biggl((U^{\io\th} \cap wU^{\io\th}w\iv) \times 
          \prod_{i=1}^{r-1}(M_{p_i} \cap w(M_{p_i}))\biggr).                
\end{equation*}
By the computation in the proof of Lemma 6.21, we see that 
$\dim \a_w\iv(s) \le  d'_{\Bm}$  for any $s \in T^{\io\th}$, and that 
the equality holds if and only if $w \in \CW_{\Bm}\nat$.
Moreover in that case, each fibre is an irreducible variety.
It follows that 
\begin{equation*}
\tag{7.2.3}
R^{2d'_{\Bm}}(\a_w)_!\Ql \simeq 
            \begin{cases} 
                 \Ql  &\quad\text{ if } w \in \CW\nat_{\Bm}, \\
                  0   &\quad\text{ otherwise.}
            \end{cases} 
\end{equation*} 
Since $\w_T$ is a finite morphism, we have
$R^{2d'_{\Bm}}(\s_w)!\Ql \simeq R^0(\w_T)_!R^{2d'_{\Bm}}(\a_w)_!\Ql$.
Thus (7.2.1) follows from (7.2.3).
\par
Since $\w_T$ is a finite morphism, (7.2.1) and (7.1.2) imply that 
$\CT_w$ is a perfect sheaf if $w \in \CW\nat_{\Bm}$.
By (7.2.2), each fibre $\a_w\iv(s)$ is a vector bundle over 
$\CO_w$.  In turn, $\CO_w$ is a vector bundle over $H/B^{\th}$ with fibre 
isomorphic to $U^{\th} \cap wU^{\th}w\iv$.  It follows that 
$H^i_c(\a_w\iv(s), \Ql) = 0$ for odd $i$. 
This implies that $R^{2d'_{\Bm}-1}(\s_w)_!\Ql = 0$ for $w \in W_n$.       
Now $\CW\nat_{\Bm}$ is a parabolic subgroup of $W$, and the closure relations 
for $\CO_w (w \in \CW\nat_{\Bm})$ are described by the Bruhat order on $\CW\nat_{\Bm}$.
It follows, by a similar argument as in [Lu, 5.4] by using the property (7.1.3), 
we see that $\CT$ is a perfect sheaf on 
$\Xi$. The lemma is proved.    
\end{proof}

\begin{prop}  %%%%  Prop. 7.3.
$\CT \simeq \bigoplus_{w \in \CW\nat_{\Bm}}\CT_w$ as sheaves on $\Xi$. 
\end{prop}

\begin{proof}
Put $\Xi\reg = \w_T(T^{\io\th}\reg)$.  Then $\Xi\reg$ 
is an open dense subset of $\Xi$.  Since $\CT$ and $\bigoplus_w\CT_w$ are perfect sheaves on 
$\Xi$, it is enough to show that their restrictions on $\Xi\reg$ are isomorphic.
Put $\CZ_0^{(\Bm)} = \s\iv(\Xi\reg)$.  Then 
$\CZ_0^{(\Bm)} \simeq \wt\CY_{\Bm} \times_{\CY_{\Bm}}\wt\CY_{\Bm}$.
Let $\s_0$ be the restriction of $\s$ on $\CZ^{(\Bm)}_0$, which is a composite 
of the natural map $\CZ^{(\Bm)}_0 \to \CY_{\Bm}$ with the map
$\CY_{\Bm} \to \Xi$.
The restriction of $\CT$ on $\Xi\reg$ is isomorphic to $R^{2d'_{\Bm}}(\s_0)_!\Ql$.    
Put 
\begin{equation*}
\BM^{\flat} = \prod_{i=1}^{r-1}(M_{[p_{i-1}+1, p_i]} + M_{p_{i-1}} ). 
\end{equation*}
Then $\BM^{\flat}$ coincides with $\BM_{\BI^{\circ}}$ for 
$\BI^{\circ} = (I_1^{\circ}, \dots, I_{r}^{\circ})$ in 1.3, where 
$I_i^{\circ} = [p_{i-1}+1, p_i]$.
Put $\wt\CY^{\flat} = \wt\CY_{\BI^{\circ}}$ in the notation in 1.3 (see also 4.1).
Under the isomorphism (1.2.2), $\wt\CY^{\flat}$ is regarded as 
an open dense subset of $\wt\CY_{\Bm}$.  Put 
$\CZ^{\flat} = \wt\CY^{\flat} \times_{\CY_{\Bm}}\wt\CY^{\flat}$ and
let $\s_{\flat}$ be the restriction of $\s_0$ on $\CZ^{\flat}$.  
For each $w \in W_n$, let $\s_{\flat}^w$ be  the restriction of $\s_{\flat}$ on 
$\CZ^{\flat}(w) = \CZ^{\flat} \cap \wt Z_w$.
Note that $\CZ^{\flat}(w) = \emptyset$ unless $w \in \CW\nat_{\Bm}$.
Now $\CZ^{\flat}$ is an open subset of $\CZ_0^{(\Bm)}$ and the inclusion 
map $\CZ^{\flat} \hra \CZ^{(\Bm)}_0$ induces a morphism
$R^i(\s_{\flat})_!\Ql \to R^i(\s_0)_!\Ql$ for each $i$.  Similarly, we 
have a morphism  
$R^i(\s_{\flat}^w)_!\Ql \to R^i(\s^w_0)_!\Ql$ for each $w$. 
We have
\begin{align*}
\tag{7.3.1}
R^{2d'_{\Bm}}(\s_{\flat})_!\Ql &\isom R^{2d'_{\Bm}}(\s_0)_!\Ql,   \\
R^{2d'_{\Bm}}(\s_{\flat}^w)_!\Ql &\isom R^{2d'_{\Bm}}(\s_0^w)_!\Ql  \quad \text{ for }
  w \in \CW\nat_{\Bm}.
\end{align*}
\par
We show (7.3.1).
Let $\wt\a_0: \CZ^{(\Bm)}_0 \to T^{\io\th}$ be the restriction of $\wt\a$ on 
$\wt\CZ^{(\Bm)}_0$, and $\a_0^w$ the restriction of $\wt\a_0$ on $\wt Z_w$.  
We also denote by $\a_{\flat}$ (resp. $\a_{\flat}^w$) 
the restriciton of $\wt\a_0$ on $\CZ^{\flat}$ (resp. on $\CZ^{\flat}(w)$).
From  the compution in the proof of Lemma 7.2, we know that the map 
$R^{2d'_{\Bm}}(\a_{\flat}^w)_!\Ql \to R^{2d'_{\Bm}}(\a_0^w)_!\Ql$ is surjective. 
For any $\BI \subset \CI(\Bm')$ with 
$\Bm' \le \Bm$, we consider $\psi_{\BI} : \wt\CY_{\BI} \to \CY_{\Bm'}^0$ as in 1.3.
For each $\BI, \BJ \in \CI(\Bm')$, put 
$\CZ_{\BI, \BJ} = \wt\CY_{\BI}\times_{\CY_{\Bm}}\wt\CY_{\BJ}$ under the inclusion 
$\CY^0_{\Bm'} \hra \CY_{\Bm}$.
Note that $\CZ_{\BI, \BJ} = \emptyset$ unless $\BJ = w(\BI)$ for some 
$w \in \CW\nat_{\Bm}$. 
We have a partition $\CZ_0^{(\Bm)} = \coprod_{\BI, \BJ}\CZ_{\BI, \BJ}$ by locally closed 
subsets $\CZ_{\BI, \BJ}$, and $\CZ^{\flat}$ coincides with $\CZ_{\BI^{\circ}, \BI^{\circ}}$.   
Let $\CZ_1$ be the complement of $\CZ^{\flat}$ in $\CZ^{(\Bm)}_0$, and 
put $\CZ_1(w) = \CZ_1 \cap \wt Z_w$. Hence $\CZ_1 = \coprod_w \CZ_1(w)$.  
Again by a similar computation as in the proof of Lemma 7.2, we see that 
$\dim \a_w\iv(s) \cap \CZ_1(w) < d'_{\Bm}$ 
for any $s \in T^{\io\th}$.     
This implies that 
$R^{2d'_{\Bm}}(\a_{\flat}^w)_!\Ql \simeq R^{2d'_{\Bm}}(\a_0^w)_!\Ql$. 
By applying $R^0(\w_T)_!$ on both sides, we have
$R^{2d'_{\Bm}}(\s_{\flat}^w)_!\Ql \simeq R^{2d'_{\Bm}}(\s_0^w)_!\Ql$ 
for any $w \in \CW\nat_{\Bm}$. 
This proves the second statement of (7.3.1).  By considering the long exact sequence
arising from  the stratification 
$\CZ^{(\Bm)}_0 = \coprod_w(\CZ^{(\Bm)}_0 \cap \wt Z_w)$, we obtain the first statement.
Thus (7.3.1) holds.

\par
By (7.3.1), the proof of the proposition is reduced to showing 
\begin{equation*}
\tag{7.3.2}
R^{2d'_{\Bm}}(\s_{\flat})_!\Ql \simeq \bigoplus_{w \in \CW\nat_{\Bm}}
                       R^{2d'_{\Bm}}(\s_{\flat}^w)_!\Ql.       
\end{equation*} 
Note that in this case $\CZ_{\flat} \to \CY_{\Bm}^0$ is a finite Galois covering.
(This is clear in the enhanced case.  In the exotic case,  
since $m_{r} = 0$, we have $I^{\circ}_r = \emptyset$.  Hence 
$\wt\CY_{\flat} \to \CY_{\Bm}^0$ is a finite Galois covering.) 
Also note that 
$\CZ^{\flat}(w) = \emptyset$ 
unless $w \in \CW\nat_{\Bm}$.   
In that case,  
$\CZ^{\flat}(w)$ is an open and closed subset of 
$\CZ^{\flat}$ since $\vf_{\flat}$ is a finite Galois covering. 
We have 
a decompsotion $\CZ^{\flat} = \coprod_{w \in \CW\nat_{\Bm}}\CZ^{\flat}(w)$
into open and closed subsets. 
This implies that 
\begin{equation*}
\tag{7.3.3}
(\s_{\flat})_!\Ql \simeq \bigoplus_{w \in \CW\nat_{\Bm}}(\s^w_{\flat})_!\Ql.
\end{equation*} 

Hence (7.3.2) holds. The proposition is proved.
\end{proof}

\para{7.4.}
By the K\"unneth formula, 
$\vf_!\Ql \simeq \pi^{(\Bm)}_!\Ql \otimes \pi^{(\Bm)}_!\Ql$.
By Theorem 3.2 and Theorem 4.5, $\pi^{(\Bm)}_!\Ql$ has a natural structure of 
$\CW_{\Bm}\nat$-module.  Hence $\vf_!\Ql$ has a structure of 
$\CW_{\Bm}\nat \times \CW_{\Bm}\nat$-module.  It follows that 
$\CT = \CH^{2d'_{\Bm}}(\s_!\Ql) \simeq \CH^{2d'_{\Bm}}(\wt\w_!(\vf_!\Ql))$ 
has a natural action of $\CW_{\Bm}\nat \times \CW_{\Bm}\nat$.  
Under the decomposition of $\CT$ in Proposition 7.3, the action of 
$\CW_{\Bm}\nat \times \CW_{\Bm}\nat$ has the following property.
For each $w_1, w_2 \in \CW_{\Bm}\nat$, 
\begin{equation*}
\tag{7.4.1}
(w_1, w_2) \cdot \CT_w = \CT_{w_1ww_2\iv}.
\end{equation*}
In fact, since $\CT$ is a perfect sheaf by Lemma 7.2, it is enough
to check (7.4.1) for the restriction of $\CT$ on $\Xi\reg$. 
Here $(\vf_{\flat})_!\Ql$ has already a structure of 
$\CW_{\Bm}\nat \times \CW_{\Bm}\nat$-module.  Hence by (7.3.1), 
it is enough to check 
a similar property for the decomposition of $(\vf_{\flat})_!\Ql$ 
in (7.3.3).  But this can be verified directly from the discussion 
in 7.3. Hence (7.4.1) holds.
\par
Let $a_0$ be the element in $\Xi$ corresponding to the $S_n$-orbit of 
$1 \in T^{\io\th}$, and  $\CT_{a_0}$ be the stalk of $\CT$ at $a_0 \in \Xi$.
By Proposition 7.3, we have a decomposition 
\begin{equation*}
\CT_{a_0} = \bigoplus_{w \in \CW_{\Bm}\nat}(\CT_w)_{a_0},
\end{equation*}  
where $(\CT_w)_{a_0}$ is the stalk of $\CT_w$ at $a_0$. 
$\CW_{\Bm}\nat \times \CW_{\Bm}\nat$ acts on $\CT_{a_0}$ 
following (7.4.1).
By (7.2.1), $\CT_w \simeq (\w_T)_!\Ql$. Since 
$\w_T\iv(a_0) = \{1\} \subset T^{\io\th}$, 
$(\CT_w)_{a_0} \simeq H^0_c(\w_T\iv(a_0),\Ql) \simeq \Ql$. 
Thus we have proved

\begin{prop}  %%%%% Prop. 7.5
$\CT_{a_0}$ has a structure of $\CW_{\Bm}\nat \times \CW_{\Bm}\nat$-module, 
which coincides with the two-sided regular representation of $\CW_{\Bm}\nat$.
\end{prop}

The following lemma is a variant of [Lu, Lemma 6.7].
\begin{lem}  %%%%% Lemma 7.6.
Let $A, A'$ be simple perverse sheaves on $\CX_{\Bm,\unip}$.
Then we have

\begin{equation*}
\dim \BH^0_c(\CX_{\Bm,\unip}, A \otimes A') = 
            \begin{cases}    1  \quad\text{ if } A' \simeq D(A), \\
                             0   \quad\text{ otherwise.} 
            \end{cases} 
\end{equation*}
where $D(A)$ is the Veridier dual of $A$. 
\end{lem}

\begin{proof}
Assume that  $A = \IC(\ol X, \CE)[\dim X]$ and $A' = \IC(\ol X', \CE'][\dim X']$,
where $X, X'$ are smooth irreducible subvarieties of $\CX_{\Bm,\unip}$ and 
$\CE$ (resp. $\CE'$) is a simple local system on $X$ (resp. on $X'$). 
We have
$\BH^0_c(\CX_{\Bm,\unip}, A\otimes A') \simeq \BH^0_c(\ol X \cap \ol X', A\otimes A')$.
First assume that $\ol X \ne \ol X'$, and put $Y = \ol X \cap \ol X'$.  
We show 
\begin{equation*}
\tag{7.6.1}
\BH^0_c(Y, A\otimes A') = 0. 
\end{equation*}
For this, by using the hypercohomology 
spectral sequence, it is enough to show the following.
\par\medskip\noindent
(7.6.2) \ If $H^i_c(Y, \CH^jA \otimes \CH^{j'}A') \ne 0$, then 
$i + j+ j' < 0$. 
\par\medskip
We show (7.6.2).  Suppose that $H^i_c(Y, \CH^jA\otimes \CH^{j'}A') \ne 0$.  Put 
\begin{equation*}
Y_{j,j'} = \supp \CH^jA \cap \supp \CH^{j'}A' \subset Y.
\end{equation*} 
We have $H^i_c(Y_{j,j'}, \CH^jA\otimes \CH^{j'}A') 
    \simeq H^i_c(Y, \CH^jA\otimes \CH^{j'}A') \ne 0$.  It follows that 
$i \le 2\dim Y_{j,j'}$.  By using the property of intersection cohomology, 
we have
\begin{align*}
\dim Y_{j,j'} &\le \dim \supp \CH^jA \le -j, \\ 
\dim Y_{j,j'} &\le \dim \supp \CH^{j'}A' \le -j'
\end{align*} 
and so 
\begin{equation*}
\tag{7.6.3}
j \le - \dim Y_{j,j'}, \qquad j' \le - \dim Y_{j,j'}.
\end{equation*}
Since $\ol X \ne \ol X'$, we have $\dim Y < \dim X$ or $\dim Y < \dim X'$, and 
one of the inequalities in (7.6.3) is a strict inequality.
It follows that $i + j+ j' < 0$.  Hence (7.6.2) holds and (7.6.1) follows.
\par
Next assume that $\ol X = \ol X'$.  We may assume that $X, X'$ are open dense in $\ol X$.
By replacing $X, X'$  by $X \cap X'$, if necessary, we may assume that $X = X'$.
Put $Y = \ol X \backslash X$.  We show that 
\begin{equation*}
\tag{7.6.4}
\BH^0_c(Y, A\otimes A') = 0 \text{ and } \BH^{-1}_c(Y, A\otimes A') = 0.
\end{equation*} 
As in the previous case, we consider the hypercohomology spectral sequence.
Suppose that 
$H^i_c(Y, \CH^jA\otimes \CH^{j'}A') \ne 0$.  We may replace $Y$ by $Y_{j,j'}$, 
where 
$Y_{j,j'} = Y \cap \supp \CH^jA \cap \supp \CH^{j'}A'$. 
Then we have $i \le 2\dim Y_{j,j'}$, and we obtain a similar formula as (7.6.3), 
but in this case, both of them  are strict inequalities since 
$\dim Y < \dim X$.  It follows that $i + j + j' < -1$.  
This proves (7.6.4). By using the cohomology long exaxt sequence with respect to 
$Y = \ol X \backslash X$, we see that
\begin{equation*}
\BH^0_c(\ol X, A\otimes A') \simeq \BH^0_c(X, A\otimes A') 
     \simeq H_c^{2\dim X}(X, \CE\otimes \CE').
\end{equation*}
The last space is isomorphic to $\Ql$ if $\CE' \simeq \CE^{\vee}$, the dual local
system,  and is equal to zero
otherwise.  Since $D(A) = \IC(\ol X, \CE^{\vee})[\dim X]$, the lemma is proved.  
\end{proof}

\para{7.7.}
We consider the map $\pi^{(\Bm)}_1 : \wt\CX_{\Bm,\unip} \to \CX_{\Bm, \unip}$.
Put $K_{\Bm,1} = (\pi_1^{(\Bm)})_!\Ql[d'_{\Bm}]$.
By Lemma 6.21 (iii), the map $\pi_1^{(\Bm)}$ is semi-small. 
Hence 
$K_{\Bm,1}$ is a semisimple perverse sheaf on $\CX_{\Bm, \unip}$ and is 
decomposed as 
\begin{equation*}
\tag{7.7.1}
K_{\Bm,1} \simeq \bigoplus_{A \in \ZC_{\Bm}}V_A \otimes A,
\end{equation*}
where $\ZC_{\Bm}$ is the set of (isomorphism  classes of) 
simple perverse sheaves appearing in the 
decomposition of $K_{\Bm, 1}$, and 
$V_A = \Hom (K_{\Bm,1}, A)$ is the multiplicity space for $A$.
We have the following.

\begin{prop}  %%%%  Prtop. 7.8.
Under the notation as above, put $m_A = \dim V_A$ for each $A \in \ZC_{\Bm}$.  Then 
we have
\begin{equation*}
\sum_{A \in \ZC_{\Bm}} m_A^2  =  |\CW_{\Bm}\nat|.
\end{equation*}
\end{prop}

\begin{proof}
By the computation in 7.4, we have
\begin{align*}
\CT_{a_0} &\simeq \BH_c^{2d'_{\Bm}}(\CX_{\Bm, \unip}, 
        (\pi_1^{(\Bm)})_!\Ql \otimes (\pi_1^{(\Bm)})_!\Ql)  \\
          &\simeq \BH_c^0(\CX_{\Bm,\unip}, K_{\Bm,1} \otimes K_{\Bm, 1}).
\end{align*}
Hence by (7.7.1), we have 
\begin{equation*}
\dim \CT_{a_0} = \sum_{A, A' \in \ZC_{\Bm}} (m_Am_{A'}) 
     \dim \BH_c^0(\CX_{\Bm,\unip}, A \otimes A').
\end{equation*}
By Lemma 7.6, $\BH^0_c(\CX_{\Bm,\unip}, A \otimes A') \ne 0$ only when 
$D(A) = A'$, in which case, $\dim \BH^0_c(\CX_{\Bm,\unip}, A \otimes A') = 1$.
But since $K_{\Bm,1}$ is self dual, $m_A = m_{D(A)}$ for each $A$. 
It follows that $\dim\CT_{a_0} = \sum_{A \in \ZC_{\Bm}} m_A^2$. 
On the other hand, by Proposition 7.5, we have $\dim \CT_{a_0} = |\CW\nat_{\Bm}|$.
This proves the proposition.  
\end{proof}   

\para{7.9.}
Since $\pi^{(\Bm)}_!\Ql$ is equipped with the $\CW_{\Bm}\nat$-action, 
for each $z = (x, \Bv) \in \CX_{\Bm}$, $H^i((\pi^{(\Bm)})\iv(z), \Ql)$ 
turns out to be a $\CW_{\Bm}\nat$-module. 
In the case where $z_0 = (1, \mathbf{0})$, $(\pi^{(\Bm)})\iv(z_0) \simeq H/B^{\th}$
and so $H^i(H/B^{\th}, \Ql)$ has a structure of $\CW_{\Bm}\nat$-module.  Note that
$\CW_{\Bm}\nat$ is a subgroup of $W_n$ (resp. of $S_n$) in the exotic case 
(resp. the enhanced case). 
It is well-known that the Weyl group $W_n$ (resp. $S_n$) acts naturally 
on $H^i(H/B^{\th}, \Ql)$, 
which we call the classical action of $W_n$ (resp. $S_n$) on it. 
We have the following lemma.

\begin{lem} %%%%  Lemma 7.10.
The action of $\CW_{\Bm}\nat$ on $H^i(H/B^{\th}, \Ql)$ coincides with the restriction 
of the classical action of $W_n$ (resp. $S_n$) on it in the exotic case 
(resp. the enhanced case). 
\end{lem} 

\begin{proof}
First assume that $\CX_{\Bm}$ is of exotic type.  
Put $\Bm' = (0, \dots, 0, n, 0) \in \CQ_{n,r}^0$. Then $\CY_{\Bm'} \subset \CY_{\Bm}$, 
and $\CX_{\Bm'} \subset \CX_{\Bm}$ (in fact, 
$\CY_{\Bm'} = \{ (x, \Bv) \in \CY_{\Bm} \mid v_1 = \cdots = v_{r-2} = 0 \}$ and 
similarly for $\CX_{\Bm'}$).  
We have $\CW\nat_{\Bm'} = W_n$. In this case, 
$\psi^{(\Bm)}_!\Ql|_{\CY_{\Bm'}} \simeq \psi^{(\Bm')}_!\Ql$, and 
$\pi^{(\Bm)}_!\Ql|_{\CX_{\Bm'}} \simeq \pi^{(\Bm')}_!\Ql$.  
It follows from  the construction that the action of $\CW\nat_{\Bm}$ on 
$\psi^{(\Bm)}_!\Ql|_{\CY_{\Bm'}}$
coincides with the restriciton of the action of $\CW\nat_{\Bm'} \simeq W_n$ on 
$\psi^{(\Bm')}_!\Ql$. Hence a similar fact holds also for $\pi^{(\Bm)}_!\Ql$.
In particular, the $\CW\nat_{\Bm}$-action on $\CH^i_{z_0}(\pi^{(\Bm)}_!\Ql)$ 
coincides with the 
restriction of the $\CW\nat_{\Bm'}$-action on $\CH^i_{z_0}(\pi^{(\Bm')}_!\Ql)$.     
Here $\CX_{\Bm'} \simeq G^{\io\th} \times V$, and $\pi^{(\Bm')}_!\Ql$ is isomorphic to 
$\pi_!\Ql$, where $\pi: \wt\CX \to \CX = G^{\io\th} \times V$ is the map defined 
in 1.2 for the case $r = 2$. The complex $\pi_!\Ql$ was studied in [SS1], and one sees that 
the action of $\CW\nat_{\Bm'}$ on 
$\pi^{(\Bm')}_!\Ql$ is nothing but the $W_n$-action on $\pi_!\Ql$ constructed in [SS1].
It induces a $W_n$-action on $H^i(H/B^{\th}, \Ql)$, which is called the exotic action 
of $W_n$.  Hence in order to prove the lemma, it is enough to see that the exotic action 
and the classical action of $W_n$ on $H^i(H/B^{\th}, \Ql)$ coincide with each other.   
But this is proved in [SS1, Lemma 5.2]. Hence the lemma is proved in the exotic case.
\par
Next assume that $\CX_{\Bm}$ is of enhanced type. 
In this case, we consider  $\CY_{\Bm'} \subset \CY_{\Bm}$ and 
$\CX_{\Bm'} \subset \CX_{\Bm}$  for  
$\Bm' = (0, \dots, 0, n) \in \CQ_{n,r}$.
Then $\CX_{\Bm'} \simeq G^{\io\th} \times V$, and as in the exotic case, the proof is 
reduced to the case where $r = 2$. So we consider $\CX = G \times V$ for $G = GL(V)$
and let $\pi: \wt\CX \to \CX$ be the corresponding map.  We have an action of $S_n$ on 
$H^i(G/B, \Ql)$, called the enhanced action of $S_n$, obtained from  the $S_n$ action 
on $\pi_!\Ql$. One can prove that the enhanced action coincides with the classical
action, by a similar (but simpler) argument as in the proof of Lemma 5.2 in [SS1]. 
Hence the lemma holds for the enhanced case.

\end{proof}

\para{7.11.}
We keep the assumption as before.  
By applying  Theorem 3.2 and Theorem 4.5 to the case where $\CE = \Ql$, 
one can write as 
\begin{equation*}
\tag{7.11.1}
\pi^{(\Bm)}_!\Ql[d_{\Bm}] \simeq \bigoplus_{\r \in (\CW\nat_{\Bm})\wg}
\rho  \otimes K_{\r}, 
\end{equation*} 
where $d_{\Bm} = \dim \CX_{\Bm}$, $K_{\r}$ is a simple perverse sheaf on $\CX_{\Bm}$ as given in 
Theorem 3.2 and Theorem 4.5.  
More precisely, if $\CX_{\Bm}$ is of exotic type then  
$K_{\r} = \IC(\CX_{\Bm(k)}, \CL_{\r_1})[d_{\Bm(k)}]$ for
$\r \simeq V\nat_{\r_1}$ with 
$\r_1 \in \CW_{\Bm(k)}\wg$, and if $\CX_{\Bm}$ is of enhanced type then 
$K_{\r}  = \IC(\CX_{\Bm}, \CL_{\r})[d_{\Bm}]$.  
We consider the complex $(\pi_1^{(\Bm)})_!\Ql[d'_{\Bm}]$, where 
$d_{\Bm'} = \dim \CX_{\Bm, \unip}$.  
The following result gives the Springer correspondence with respect to the 
action of $\CW\nat_{\Bm}$.  In the exotic case, this result is regarded as 
a weak version of the Springer correspondence.

\begin{thm}[Springer correspondence for $\CW\nat_{\Bm}$]  %%%% Theorem 7.12.
Let $\CX_{\Bm}$ be of exotic type or of enhanced type for $\Bm \in \CQ_{n,r}$. 
In the exotic case, assume further that $\Bm \in \CQ_{n,r}^0$. Then 
$(\pi_1^{(\Bm)})_!\Ql[d'_{\Bm}]$ is a semisimple perverse sheaf on $\CX_{\Bm,\unip}$ 
equipped with $\CW\nat_{\Bm}$-action,
and is decomposed as 
\begin{equation*}
\tag{7.12.1}
(\pi_1^{(\Bm)})_!\Ql[d'_{\Bm}] \simeq \bigoplus_{\r \in (\CW\nat_{\Bm})\wg}
                 \r \otimes L_{\r}, 
\end{equation*}
where $L_{\r}$ is a simple perverse sheaf on $\CX_{\Bm,\unip}$ such that 
\begin{equation*}
\tag{7.12.2}
K_{\r}|_{\CX_{\Bm, \unip}} \simeq L_{\r}[d_{\Bm} - d'_{\Bm}].
\end{equation*}
\end{thm}

\begin{proof}
As discussed in 7.7, $K_{\Bm,1} = (\pi^{(\Bm)}_1)_!\Ql[d'_{\Bm}]$ is a semisimple perverse
 sheaf.  Since $K_{\Bm,1}$ is the restriction of $(\pi^{(\Bm)})_!\Ql$ on 
$\CX_{\Bm, \unip}$, $K_{\Bm,1}$ has a strucute of $\CW\nat_{\Bm}$-module.
Thus we can define an algeba homomorphism   
\begin{equation*}
\begin{CD}
\Ql[\CW\nat_{\Bm}] @>\a >>  \End K_{\Bm,1} @>\b >> \End H^{\bullet}(H/B^{\th}, \Ql).
\end{CD}
\end{equation*}
In order to show (7.12.1), it is enough to see that $\a$ gives an isomorphism 
\begin{equation*}
\tag{7.12.3}
\a: \Ql[\CW\nat_{\Bm}] \isom \End K_{\Bm,1}.
\end{equation*}
We show (7.12.3).  We assume that $\CX_{\Bm}$ is of exotic type.  
The proof for the enhanced case is similar. 
$\b\circ\a$ is a homomorphism induced from the action of $\CW\nat_{\Bm}$ 
on $H^{\bullet}(H/B^{\th}, \Ql)$. By Lemma 7.10, this action is the restriction of the classical 
action of $W_n$ on $H^{\bullet}(H/B^{\th}, \Ql)$.    
Since $H^{\bullet}(H/B^{\th}, \Ql) \simeq \Ql[W_n]$ as $W_n$-module,  
with respect to the classical 
action, $\Ql[W_n] \to H^{\bullet}(H/B^{\th}, \Ql)$ is injective. 
Hence $\b\circ \a$ is injective.
It follows that $\a$ is injective. 
On the other hand, Proposition 7.8 implies that $\dim \End K_{\Bm,1} = |\CW\nat_{\Bm}|$.
This shows that $\a$ is surjective, and so (7.12.3) holds.
\par
(7.12.2) now follows easily by comparing (7.11.1) and (7.12.1).  The theorem  is proved.
\end{proof}

\para{7.13.}
We now return to the setting in 1.6, and consider the complex reflection group 
$W_{n,r}$.  For each $\Bm \in \CQ_{n,r}^0$,
we denote by $(W_{n,r}\wg)_{\Bm}$ the set of irreducible representations $\wt V_{\r}$
(up to isomorphism) of $W_{n,r}$ obtained from $\r \in \CW_{\Bm(k)}\wg$ for 
various $0 \le k \le m_{r-1}$ as in 1.6.
Then we have 
\begin{equation*}
\tag{7.13.1}
W_{n,r}\wg = \coprod_{\Bm \in \CQ_{n,r}^0}
                       (W_{n,r}\wg)_{\Bm}.
\end{equation*}
It follows from the construction of $\wt V_{\r}$ and of $V_{\r}\nat$, 
there exists a natual bijection between $(W_{n,r}\wg)_{\Bm}$ and $(W_{\Bm}\nat)\wg$.
We denote by $V(\r)$ the irreducible representation of $W_{n,r}$ belonging to
$(W_{n,r}\wg)_{\Bm}$ corresponding to $\r \in (W_{\Bm}\nat)\wg$.  
Now assume that $\CX_{\Bm}$ is of exotic type. 
We consider the map $\ol\pi_{\Bm}: \pi\iv(\CX_{\Bm}) \to \CX_{\Bm}$, and  
the complex $(\ol\pi_{\Bm})_!\Ql[d_{\Bm}]$ as in 2.1.          
Then by Theorem 2.2, $(\ol\pi_{\Bm})_!\Ql[d_{\Bm}]$ is a semisimple perverse sheaf equipped with 
$W_{n,r}$-action, and is decomposed as 
\begin{equation*}
\tag{7.13.2}
(\ol\pi_{\Bm})_!\Ql[d_{\Bm}] \simeq \bigoplus_{\r \in (W_{\Bm}\nat)\wg}V(\r) \otimes K_{\r},
\end{equation*} 
where $K_{\r}$ is a simple perverse sheaf on $\CX_{\Bm}$ given in (7.11.1).
Let $\ol\pi_{\Bm,1} : \pi\iv(\CX_{\Bm, \unip}) \to \CX_{\Bm, \unip}$ be the restriciton of
$\ol \pi_{\Bm}$ on $\CX_{\Bm, \unip}$.  Since $\ol\pi_{\Bm,1}$ is proper, 
$\ol\pi_{\Bm,1}\Ql[d'_{\Bm}]$ is a semisimple complex on $\CX_{\Bm, \unip}$.
By applying (7.12.2), we see that $\ol\pi_{\Bm,1}\Ql[d'_{\Bm}]$ is a semisimple perverse sheaf.
As a corollary to Theorem 7.12, we obtain the Springer correspondence for $W_{n,r}$.

%%%%
%%%%
\begin{cor}[Springer correspondence for $W_{n,r}$]  %%%%  Corollary 7.14.
Assume that $\CX_{\Bm}$ is of exotic type with 
$\Bm \in \CQ_{n,r}^0$.  Then $\ol\pi_{\Bm,1}\Ql[d'_{\Bm}]$ 
is a semisimple perverse sheaf on 
$\CX_{\Bm,\unip}$ with $W_{n,r}$-action, and is decomposed 
as 
\begin{equation*}
\ol\pi_{\Bm,1}\Ql[d'_{\Bm}]\simeq \bigoplus_{\r \in (W_{\Bm}\nat)\wg}V(\r) \otimes L_{\r},
\end{equation*}
where $L_{\r}$ is a simple perverse sheaf on $\CX_{\Bm, \unip}$ as in Theorem 7.12, 
and $V(\r)$ is an irreducible representation of $W_{n,r}$ as defined in 7.13.
\end{cor}
\par\bigskip
%%%%%%
%%%%%%
\section{Determination of the Springer correspondence}

\para{8.1.}
Assume that $ r \ge 2$.  In this section we shall determine 
$L_{\r}$ appearing in the Springer correspondence explicitly.
For a fixed $\Bm \in \CQ_{n,r}$, 
we consider $\CX_{\Bm}$ of exotic type or of enhanced type.
In the exotic case, we assume that $\Bm \in \CQ_{n,r}^0$.
First we consider the case where $\CX_{\Bm}$ is of exotic type.  
By modifiying the definition of $\CK_{\Bm}$ in 2.5, we define a variety 
$\CG_{\Bm}$ by
\begin{equation*}
\begin{split}
\CG_{\Bm} = \{ (x, &\Bv, (W_i)_{1 \le i \le r-2}) \mid (x, \Bv) \in \CX,
 (W_i) \text{ : partial isotropic flag in } V, \\
  &\dim W_i = p_i, x(W_i) = W_i, v_i \in W_i \ (1 \le i \le r-2), 
          v_{r-1} \in W_{r-2}^{\perp} \}.  
\end{split}
\end{equation*}
Let $\z: \CG_{\Bm} \to \CX$ be the projection to the first two factors.
We consider the map $\pi^{(\Bm)}: \wt\CX_{\Bm} \to \CX_{\Bm}$.                        
Then $\pi^{(\Bm)}$ is decomposed as                                                               
\begin{equation*}                                                                                   
\begin{CD}                                                                                          
\pi^{(\Bm)} : \wt\CX_{\Bm} @>\vf >> \CG_{\Bm}                                               
             @>\z >> \CX_{\Bm},                                                              
\end{CD}                                                                                            
\end{equation*}                                                                                     
where $\vf$ is defined by 
$(x, \Bv, gB^{\th}) \mapsto (x, \Bv, (gM_{p_i})_{1 \le i \le r-2})$.
The map $\vf$ is surjective since $\Bm \in \CQ_{n,r}^0$. It follows that 
$\z(\CG_{\Bm})$ coincides with $\CX_{\Bm}$. 
Since $\Bm \in \CQ_{n,r}^0$, we have $\dim \wt\CX_{\Bm} = \dim \CX_{\Bm}$.
This implies that $\dim \CG_{\Bm} = \dim \CX_{\Bm}$. 
Note that in the case where $r = 2$, we have $\Bm = (n, 0)$, and 
$\CG_{\Bm} = G^{\io\th} \times V = \CX_{\Bm}$.  
We define a variety $\CH_{\Bm}$ as in 2.5 by 
\begin{equation*}
\begin{split}
\CH_{\Bm} = \{ (x, \Bv, &(W_i), \f_1, \f_2) \mid (x, \Bv, (W_i)) \in  \CG_{\Bm}  \\
                     &\f_1: W_1 \isom V_0, \f_2: W_1^{\perp}/W_1 \isom \ol V_0
                             \text{ (symplectic isom.) } \},
\end{split}
\end{equation*}
where $V_0 = M_{p_1}$ and $\ol V_0 = M_{p_1}^{\perp}/M_{p_1}$. 
We also define a variety $\wt\CZ_{\Bm}$ by 

\begin{equation*}
\begin{split}
\wt\CZ_{\Bm} = \{ (x,\Bv, &gB^{\th}, \f_1, \f_2) \mid (x, \Bv, gB^{\th}) \in \wt\CX_{\Bm}, \\
              &\f_1: g(M_{p_1}) \isom V_0, \f_2: g(M_{p_1})^{\perp}/g(M_{p_1}) \isom \ol V_0 \}. 
\end{split}
\end{equation*}
Assume that $r \ge 3$, and let $\Bm' = (m_2, \dots, m_r)$ for $\Bm = (m_1, \dots, m_r)$, and   
$G_1, \wt G_1$, $\CX'_{\Bm'}, G_2$, etc.  be as in 2.5. 
$\wt\CX'_{\Bm'}$ is defined 
for $\CX'_{\Bm'}$ in a similar way as $\wt\CX_{\Bm}$ is defined for $\CX_{\Bm}$.    
(Hence in the case where $r = 3$, $\CX'_{\Bm'} = \CG'_{\Bm'} = G_2^{\io\th} \times \ol V_0$.) 
We have the following commutative diagram  
\begin{equation*}
\tag{8.1.1}  
\begin{CD}  
\wt G_1 \times \wt\CX'_{\Bm'} @<\wt\s<<   \wt\CZ_{\Bm} @>\wt q >>  \wt\CX_{\Bm} \\
        @V\pi^1 \times \vf' VV    @VV\wt\vf V    @VV\vf V  \\ 
G_1 \times \CG'_{\Bm'}  @<\s <<  \CH_{\Bm} @>q >> \CG_{\Bm} \\
  @V\id \times \z' VV     @.   @VV \z V   \\
G_1 \times \CX'_{m'}  @.      @.    \CX_m,  
\end{CD}  
\end{equation*} 
where $q, \s$ are defined in a similar way as in 2.5, and
\begin{align*}
\wt q &: (x, \Bv, gB^{\th}, \f_1, \f_2) \mapsto (x,\Bv, gB^{\th}), \\
\wt\vf &: (x,\Bv, gB^{\th}, \f_1, \f_2) \mapsto 
                       (x, \Bv, (gM_{p_i})_{1 \le i \le r-2}, \f_1, \f_2).
\end{align*}
The map $\wt\s$ is defined as follows.
Let $\CF^{\th}(V)$ be the set of isotropic flags in $V$, and 
$\wt\CF^{\th}_{\Bm}$ the set of $((x, \Bv), (V_i)) \in \CX_{\Bm} \times \CF^{\th}(V)$  
such that $(V_i)$ is $x$-stable and that $v_i \in V_{p_i}$ for $i = 1, \dots, r-1$.
Then $\wt\CF^{\th}_{\Bm}$ is isomorphic to $\wt\CX_{\Bm}$, and   
\begin{equation*}  
\begin{split}  
\wt\CZ_{\Bm} \simeq \{ (x,\Bv, (V_i), &\f_1,\f_2) \mid  
                           (x,\Bv, (V_i)) \in \wt\CF^{\th}_{\Bm},  \\  
                             &\f_1: V_{p_1} \isom V_0,  
                             \f_2 : V_{p_1}^{\perp}/V_{p_1} \isom \ol V_0 \}.  
\end{split}  
\end{equation*}  
We define 
$\wt\s : \wt\CZ_{\Bm} \to \wt G_1 \times \wt\CX'_{\Bm'}$ by  
$(x, \Bv, (V_i), \f_1, \f_2) \mapsto (\xi_1, \xi_2)$ with
\begin{align*}  
\xi_1 &= (\f_1(x')\f_1\iv, (\f_1(V_i))_{i \le p_1}) \in \wt G_1, \\  
\xi_2 &= (\f_2(x'')\f_2\iv, \f_2(\ol \Bv), (\f_2(V_i/V_{p_1}))_{i \ge p_1}) \in\wt\CX'_{\Bm'},
\end{align*}  
where $x'$ (resp. $x''$) is the restriction of $x$ on $V_{p_1}$  
(resp. $V_{p_1}^{\perp}/V_{p_1}$),  
$\ol \Bv = (\ol v_2, \dots, \ol v_{r-1})$ with $\ol v_i$ the image of  
$v_i$ on $V_{p_1}^{\perp}/V_{p_1}$ for $\Bv = (v_1, \dots, v_{r-1})$.  
One can check that squares in the diagram are both cartesian squares.
Put $H_0 = G_1 \times G_2^{\th}$.  Then as in (2.5.2), and 
(2.5.3), we have
\par\medskip\noindent
(8.1.2) \ $q$ is a principal bundle with fibre isomorphic to $H_0$, and 
$\s$ is a locally trivial fibration with smooth fibre of dimension 
$\dim H + (r-2)m_1$.
\par\medskip  

\para{8.2}
For a fixed $k$, we consider the variety 
$\wt\CY_{\Bm(k)}^{\dag} = (\psi^{(\Bm)})\iv(\CY^0_{\Bm(k)})$ 
as in 3.3. 
Let $\CG_{\Bm(k), \rg} = \z\iv(\CY_{\Bm(k)}^0)$ be a locally closed 
subvariety of $\CG_{\Bm}$ (not of $\CG_{\Bm(k)}$, note that 
$\CG_{\Bm(k)}$ is not defined since $\Bm(k) \notin \CQ_{m,r}^0$).   
We define $\Bm'(k) = \Bm'(r-2, k)$ similar to $\Bm(k)$, by replacing $\Bm$ by 
$\Bm'$.  Then the varieties 
$\CY'^0_{\Bm'(k)}, \wt \CY'^{\dag}_{\Bm'(k)}$ and $\CG'_{\Bm'(k),\rg}$ are 
defined similarly with respect to $\CX'_{\Bm'}$.
The commutative diagram (8.1.1) induces a commutative diagram

\begin{equation*}
\tag{8.2.1}
\begin{CD}
\wt G_{1,\rg} \times \wt\CY'^{\dag}_{\Bm'(k)} @<<<  \wt\CZ_{\Bm(k)}^{\dag}
               @>>>  \wt\CY_{\Bm(k)}^{\dag} \\
     @VVV   @VVV  @VV\vf_0 V   \\
G_{1,\rg} \times \CG'_{\Bm'(k), \rg} @<<<  \CH_{\Bm(k), \rg} @ >>>  \CG_{\Bm(k), \rg} \\
        @VVV    @.    @VV\z_0 V   \\
 G_{1,\rg} \times \CY'^0_{\Bm'(k)}    @.    @.     \CY_{\Bm(k)}^0,
\end{CD}
\end{equation*}
where $\CH_{\Bm(k),\rg} = q\iv(\CG_{\Bm(k), \rg}), 
        \wt\CZ_{\Bm(k)}^{\dag} = \wt q\iv(\wt\CY_{\Bm(k)}^{\dag})$, and 
$\vf_0, \z_0$ are restrictions of $\vf$, $\z$, respectively. 
Again, the squares in the diagrams are both cartesian. 
\par
By (3.3.1) we have a decomposition 
\begin{equation*}
\tag{8.2.2}
\wt\CY_{\Bm(k)}^{\dag} = \coprod_{\BI \in \CI^{\bullet}(\Bm(k))}\wt\CY_{\BI},  
\end{equation*}
into irreducible components, 
where 
\begin{equation*}
\CI^{\bullet}(\Bm(k)) = \{ \BI \in \CI(\Bm(k)) \mid I_i = I_i^{\circ}\  
 (1 \le i \le r-2), I_{r-1}, I_r \subset I^{\circ}_{r-1} \}.
\end{equation*}
Let $\psi_{\BI}$ be the restriction of the map
$\wt\CY_{\Bm(k)}^{\dag} \to \CY_{\Bm(k)}^0$
on $\wt\CY_{\BI}$ for each $\BI \in \CI^{\bullet}(\Bm(k))$ as in 1.3. 
Under the notation in 1.3, 
the map $\psi_{\BI}$ factors through $\wh\CY_{\BI}$ 
as $\psi_{\BI} = \e_{\BI}\circ \xi_{\BI}$ (see (1.3.1)).
We define a map 
$\e^{\bullet}_{\BI} : \wh\CY_{\BI} \to \CG_{\Bm(k), \rg}$ as the quotient of the map 
$(g,(t, \Bv)) \mapsto (gtg\iv, g\Bv, (g\ol M_{I_i})_{1 \le i \le r - 2})$. 
Thus $\e_{\BI} : \wh\CY_{\BI} \to \CY_{\Bm(k)}^0$ factors through $\CG_{\Bm(k), \rg}$
as $\e_{\BI} = \z_0\circ \e_{\BI}^{\bullet}$.   
\par
The variety $\wt\CY_{\Bm'(k)}'^{\dag}$ is also decomposed into irreducible 
components as in (8.2.2), by using the parameter set $\CI^{\bullet}(\Bm'(k))$.
For $\BI' \in \CI^{\bullet}(\Bm'(k))$, the varieties
$\wh\CY'_{\BI'}$ are defined with respect to $\CX'_{\Bm'}$. 
Note that teh set $\CI^{\bullet}(\Bm'(k))$ is in bijection with the set $\CI^{\bullet}(\Bm(k))$
under the correspondence $\BI' \lra \BI = (I_1^{\circ}, \BI')$.
Now the commutative diagram (8.2.1) impies the following commutative diagram
for each $\BI \in \CI^{\bullet}(\Bm(k))$
\begin{equation*}
\tag{8.2.3}
\begin{CD}
\wt G_{1,\rg} \times \wh\CY'_{\BI'} @<<<  \wh\CZ_{\BI}
                       @>>>  \wh\CY_{\BI} \\
     @V\psi^1_0 \times \e'^{\bullet}_{\BI'} VV   @VVV  @VV\e_{\BI}^{\bullet} V   \\
G_{1,\rg} \times \CG'_{\Bm'(k),\rg} @<<<  \CH_{\Bm(k), \rg} @ >>>  \CG_{\Bm(k), \rg}, \\
\end{CD}
\end{equation*}
where $\wh\CZ_{\BI}$ is the fibre product of 
$\CH_{\Bm(k),\rg}$ and $\CY_{\BI}$ over $\CG_{\Bm(k),\rg}$.  
The both squares in the diagram 
are cartesian squares. 
By 1.3, we know that $\e_{\BI} : \wh\CY_{\BI} \to \CY_{\Bm(k)}^0$ is a finite 
Galois covering with group $\CW_{\BI}$.  Here we note that

\par\medskip\noindent
(8.2.4) \ $\e^{\bullet}_{\BI} : \wh \CY_{\BI} \to \CG_{\Bm(k), \rg}$ is a finite Galois covering 
with group $\CW_{\BI}$.  In particular, the restriction of $\z_0$ gives an isomorphism 
$\CG_{\Bm(k),\rg} \isom \CY_{\Bm(k)}^0$. 
\par\medskip

In fact,  this is clear in the case where $r = 2$, since $\CG_{\Bm(k),\rg} = \CY^0_{\Bm(k)}$.
Assume that $r \ge 3$, and that (8.2.4) holds for $r-1$.  
By induction, $\e_{\BI'}'^{\bullet} : \wh\CY'_{\BI'} \to \wh\CG'_{\Bm'(k),\rg}$ 
is a finite Galois covering with group $\CW'_{\BI'}$.  Then    
$\psi^1_0 \times \e_{\BI'}'^{\bullet}$ is a finite Galois covering with group 
$S_{m_1} \times \CW'_{\BI'} = \CW_{\BI}$. 
Since both squares in the diagram  (8.2.3) are cartesians, $\e^{\bullet}_{\BI}$ 
is a finite Galois covering with group $\CW_{\BI}$.  Since $\e_{\BI}$ 
is also a finite Galois covering with group $\CW_{\BI}$, we conclude that 
$\CG_{\Bm(k),\rg} \simeq \CY_{\Bm(k)}^0$.  This proves (8.2.4). 
\par
Since $\e_{\BI}$ is a finite Galois covering with group $\CW_{\BI} \simeq \CW_{\Bm(k)}$, 
by (1.5.3) we have 
\begin{equation*}
(\e_{\BI})_!\Ql \simeq \bigoplus_{\r_0 \in \CW_{\Bm(k)}\wg}\r_0 \otimes \CL_{\r_0},
\end{equation*}
where $\CL_{\r_0}$ is a simple local system on $\CY_{\Bm(k)}^0$.
We regard $\CL_{\r_0}$ as a simple local system on $\CG_{\Bm(k),\rg}$ under the isomorphism 
$\CG_{\Bm(k),\rg} \simeq \CY_{\Bm(k)}^0$. 
\par
Now take $\r \in (\CW_{\Bm}\nat)\wg$.  There exist a unique integer $k$ and 
$\r_0 \in \CW_{\Bm(k)}\wg$
such that $\r = V_{\r_0}\nat$.  Then we have
$K_{\r} = \IC(\CX_{\Bm(k)}, \CL_{\r_0})[d_{\Bm(k)}]$.  Put 
$A_{\r} = \IC(\ol{\CG}_{\Bm(k)}, \CL_{\r_0})[d_{\Bm(k)}]$.  $A_{\r}$ is an $H$-equivariant 
simple perverse sheaf on $\ol{\CG}_{\Bm(k)}$, 
and we regard it as a perverse sheaf on $\CG_{\Bm}$ by extension by zero.

We show the following fact.

\begin{prop}  %%%%  Proposition 8.3
Assume that $\CX_{\Bm}$ is of exotic type. 
\begin{enumerate}
\item
$\vf_!\Ql[d_{\Bm}]$ is a semisimple perverse 
sheaf on $\CG_{\Bm}$ equipped with $\CW_{\Bm}\nat$-action, and is decomposed as
\begin{equation*}
\vf_!\Ql[d_{\Bm}] \simeq \bigoplus_{\r \in (\CW_{\Bm}\nat)\wg}
     \r \otimes A_{\r}.
\end{equation*}
\item
$\z_!A_{\r} \simeq K_{\r}$.
\end{enumerate}
\end{prop}

\begin{proof}
We prove the proposition by induction on $r$.  In the case where $r = 2$, 
$\Bm = (m_1, m_2) = (n, 0)$ since $\Bm \in \CQ_{n,r}^0$.  Thus 
$\CG_{\Bm} = \CX_{\Bm} = \CX$ and $\vf = \pi: \wt\CX \to \CX$.
Moreover, $\z$ is the identity map, and  $\CW_{\Bm}\nat = W_n$. 
Hence the proposition follows from [SS1, Theorem 4.2].  Assume that $r \ge 3$, 
and that the proposition holds for $r-1$. 
We know, under the notation of 2.6, that 
\begin{equation*}
\tag{8.3.1}
\pi^1_!\Ql[\dim G_1] \simeq \bigoplus_{\r_1 \in S_{m_1}\wg}\r_1 \otimes 
          \IC(G_1, \CL_{\r_1})[\dim G_1].
\end{equation*}
On the other hand, by applying the induction hypothesis to $\CX'_{\Bm'}$, we have
\begin{equation*}
\tag{8.3.2}
\vf'_!\Ql[d_{\Bm'}] \simeq 
             \bigoplus_{\r' \in (\CW\nat_{\Bm'})\wg}\r'\otimes A_{\r'},
\end{equation*}
where $A_{\r'}$ is a simple perverse sheaf on $\CG'_{\Bm'}$ defined similarly to
$A_{\r}$. 
By applying the argument in 2.6, 
one can find a unique $H$-equivariant simple perverse sheaf $\wt A_{\r}$ on 
$\CG_{\Bm}$ such that 
\begin{equation*}
\tag{8.3.3}
q^*\wt A_{\r}[\b_2] \simeq \s^*(K_{\r_1}\boxtimes A_{\r'})[\b_1],
\end{equation*}  
where $\b_1 = \dim H + (r-2)m_1$, $\b_2 = \dim H_0$, and  
$K_{\r_1} = \IC(G_1, \CL_{\r_1})[\dim G_1]$.
It follows from the discussion in 8.2 that $\wt A_{\r}$ actually coincides with $A_{\r}$.
Put $K = \vf_!\Ql[d_{\Bm}], K' = \vf'_!\Ql[d_{\Bm'}]$ and 
$K^1 = (\pi^1)_!\Ql[\dim G_1]$.
Since both squares in (8.1.1) are cartesian, we have
\begin{equation*}
q^*K[\b_2] \simeq \s^*(K_1 \boxtimes K')[\b_1].
\end{equation*}
Combining (8.3.1), (8.3.2) and (8.3.3), we obtain 

\begin{equation*}
K \simeq \bigoplus_{\r \in (\CW_{\Bm}\nat)\wg}
                      \r \otimes A_{\r}.
\end{equation*}
By this decomposition, 
$K = \vf_!\Ql[d_{\Bm}]$ is regarded as a complex with $\CW_{\Bm}\nat$-action.
This proves (i). 
\par
Next we show (ii). 
Since $\z$ is proper, $\z_!A_{\r}$ is a semisimple complex on $\CX_{\Bm}$.
By (i), $K = \vf_!\Ql[d_{\Bm}]$ is a semisimple perverse sheaf.
Since $\z_!K \simeq (\pi^{(\Bm)})_!\Ql[d_{\Bm}]$ is 
a semisimple perverse sheaf, it follows that 
$\z_!A_{\r}$ is a semisimple perverse sheaf.  
By (8.2.4) we have $\z_!A_{\r}|_{\CY^0_{\Bm(k)}} \simeq K_{\r}|_{\CY^0_{\Bm(k)}}$.
Hence $\z_!A_{\r}$ contains $K_{\r}$ as a direct summand.  
By applying $\z_!$ to the formula in (i), we have
\begin{equation*}
(\pi^{(\Bm)})_!\Ql[d_{\Bm}] \simeq \bigoplus_{\r \in (\CW_{\Bm}\nat)\wg}
                \r \otimes \z_!A_{\r}.
\end{equation*} 
By Theorem 3.2, we have
\begin{equation*}
(\pi^{(\Bm)})_!\Ql[d_{\Bm}] \simeq \bigoplus_{\r \in (\CW_{\Bm}\nat)\wg}
           \r \otimes K_{\r}.
\end{equation*}
By comparing these two formulas, we obtain (ii).  The proposition is 
proved. 
\end{proof}

\para{8.4.}
For each $\Bm \in \CQ^0_{n,r}$, put $\CG_{\Bm \unip} = \z\iv(\CX_{\Bm,\unip})$.
Then the map $\pi^{(\Bm)}_1$ is decomposed as 
\begin{equation*}
\begin{CD}
\pi^{(\Bm)}_1: \wt\CX_{\Bm, \unip} @>\vf_1 >>  \CG_{\Bm, \unip}  @>\z_1 >>  \CX_{\Bm, \unip},
\end{CD}
\end{equation*}
where $\vf_1, \z_1$ are restrictions of $\vf, \z$, respectively.  Note that 
$\vf_1$ is surjective. 
Put $\CH_{\Bm \unip} = q\iv(\CG_{\Bm, \unip})$.  The inclusion map 
$\CG_{\Bm,\unip} \hra \CG_{\Bm}$ is compatible with the diagram (8.1.1), namely we
have a commutative diagram 
\begin{equation*}
\tag{8.4.1}
\begin{CD}
G_1 \times \CG'_{\Bm'} @<\s <<  \CH_{\Bm} @>q >>  \CG_{\Bm} \\
    @AAA   @AAA    @AAA  \\
G_{1,\unip} \times \CG'_{\Bm', \unip}  @<\s_1 <<  \CH_{\Bm,\unip} @>q_1 >>  \CG_{\Bm, \unip}, 
\end{CD}
\end{equation*}
where $\s_1, q_1$ are restrictions of $\s, q$ respectively, and vertical maps are
natural inclusions.  
A simialr property as (8.1.2) still holds for $q_1, \s_1$, and 
both squares are cartesian squares.
\par  
For each $\Bla \in \CP(\Bm(k))$, we define a subset $\CG_{\Bla}$ of $\CG_{\Bm \unip}$ 
inductively as follows; 
Write $\Bla = (\la^{(1)}, \Bla')$ with $\Bla' \in \CP_{n-m_1, r-1}$. 
Assume that $G_2^{\th}$-stable subset $\CG'_{\Bla'}$ of $\CG'_{\Bm',\unip}$ was defined. 
Let $\CO_{\la^{(1)}}$ be the $G_1$-orbit in $(G_1)\uni$ corresponding to $\la^{(1)}$, 
and put $Z = \CO_{\la^{(1)}} \times \CG'_{\Bla'}$.
Then $\s_1\iv(Z)$ is an $H_0$-stable subset of $\CH_{\Bm,\unip}$, and 
$q_1\circ \s_1\iv(Z)$ coincides with the quotient of $\s_1\iv(Z)$ by $H_0$.  We define 
$\CG_{\Bla}$ by $\CG_{\Bla} = q_1\circ \s_1\iv(Z)$. 
In the case where $r = 3$,  
$\Bm = (m_1, m_2,0)$ and 
$\Bm(k) = (m_1, k, k')$ with $k + k' = m_2$. $\Bm' = (m_2, 0)$ and
$\CG'_{\Bm',\unip} = \CX'\uni$, where $\CX'$ is the exotic symmetric space for $r = 2$
associated to $G_2$.
In this case, $\Bla' = (\la^{(2)}, \la^{(3)}) \in \CP_{m_2,2}$,
and we take $\CG'_{\Bla'}$ as the $G_2^{\th}$-orbit in $\CX'\uni$ corresponding to
$\Bla'$. 
Thus $\CG_{\Bla}$ is well-defined, and $\CG_{\Bla}$ turns out to be a smooth irreducible
$H$-stable subvariety of $\CG_{\Bm, \unip}$.   
By induction on $r$, we show the following formulas.
\begin{equation*}
\tag{8.4.2}
\dim \CG_{\Bla} =
         2n^2 -2n - 2n(\Bla) - 2n(\la^{(r-1)} + \la^{(r)}) + 
       \sum_{i=1}^{r-1}(r-i+1)|\la^{(i)}|.  
\end{equation*} 

In fact by (8.1.2), we have
\begin{equation*}
\tag{8.4.3}
\dim \CG_{\Bla} = \dim \CO_{\la^{(1)}} + \dim \CG'_{\Bla'} + (\dim H + (r-2)m_1) - \dim H_0. 
\end{equation*}
By applying the induction hypothesis for $\Bla'$, (with $n' = n- m_1$), we have
\begin{equation*}
\dim \CG'_{\Bla'} = 2n'^2 - 2n' - 2n(\Bla') - 2n(\la^{(r-1)} + \la^{(r)})
      + \sum_{i=2}^{r-1}(r-i+1)|\la^{(i)}|, 
\end{equation*}
and 
$\dim \CO_{\la^{(1)}} = m_1^2 - m_1 -2n(\la^{(1)})$.
Moreover $\dim H = 2n^2 + n$, $\dim H_0 = m_1^2 + 2n'^2 + n'$. 
Substituting these formulas into (8.4.3), we obtain (8.4.2).
\par
Let $\ol\CG_{\Bla}$ be the closure of $\CG_{\Bla}$ in $\CG_{\Bm, \unip}$.  
Reall $\wt\CF_{\Bla}, \wt\CF_{\Bla}^0$ in 6.4. 
It follows from the construction, $\wt\CF_{\Bla}$ is a closed subset of $\CG_{\Bm, \unip}$. 
We have the following lemma.

\begin{lem}  %%%%  Lemma 8.5.
Assume that $\CX_{\Bm}$ is of exotic type.  
\begin{enumerate}
\item
$\ol\CG_{\Bla}$ coincides with $\wt \CF_{\Bla}$.  In particular, 
$\wt \CF_{\Bla}^0$ is an open dense subset of $\ol\CG_{\Bla}$.  
\item
$\z_1(\ol\CG_{\Bla}) = \ol X_{\Bla}$, and 
$\z_1\iv(X_{\Bla}) = \wt \CF_{\Bla}^0$.  Hence 
the restriction of $\z_1$ on $\z_1\iv(X_{\Bla})$ gives an isomorphism 
$\z_1\iv(X_{\Bla}) \isom X_{\Bla}$.  
 
\end{enumerate}
\end{lem}

\begin{proof}
By induction on $r$, one can verify that $\wt \CF_{\Bla}^0 \subset \CG_{\Bla}$.
Hence $\wt \CF_{\Bla} \subset \ol\CG_{\Bla}$.  By Lemma 6.6 and by (8.4.2), 
$\dim \wt \CF_{\Bla} = \dim \ol\CG_{\Bla}$.  Since both are irredcuible closed 
subsets of $\CG_{\Bm, \unip}$, we have $\wt \CF_{\Bla} = \ol\CG_{\Bla}$.  This proves (i).
Then the restriction of $\z_1$ on $\ol\CG_{\Bla}$ coincides with the map 
$\pi_{\Bla} : \wt \CF_{\Bla} \to \ol X_{\Bla}$. Hence (ii) follows from 
Proposition 6.7.  
\end{proof}

\para{8.6}
 Recall the set $\wt\CP(\Bm)$ in (6.8.1) for each $\Bm \in \CQ^0_{n,r}$.
It is well-known that 
$(\CW_{\Bm}\nat)\wg$ is naturally parametrized by the set $\wt\CP(\Bm)$.
We denote by $\r_{\Bla}$ the irreducible representation of $\CW_{\Bm}\nat$ 
corresponding to $\Bla \in \wt\CP(\Bm)$. 
\par
For each $\Bla \in \wt\CP(\Bm)$, let $V(\Bla)$ be the irreducible 
representation of $W_{n,r}$ obtained from $\r_{\Bla}$ as in 7.13.  Then we have
\begin{equation*}
\tag{8.6.1}
W_{n,r}\wg \simeq \coprod_{\Bm \in \CQ_{n,r}^0}\wt\CP(\Bm).
\end{equation*}

\par
We have the following refinements of Theorem 7.12 and Corollary 7.14.

\begin{thm}  %%%%  Theorem 8.7.
Assume that $\CX_{\Bm, \unip}$ is of exotic type with $\Bm \in \CQ_{n,r}^0$.
\begin{enumerate}
\item
Let $L_{\r}$ be as in Theorem 7.12.  Assume that 
$\r = \r_{\Bla} \in (\CW_{\Bm}\nat)\wg$  for $\Bla \in \wt\CP(\Bm)$.
Then we have 
\begin{equation*}
L_{\r} \simeq \IC(\ol X_{\Bla}, \Ql)[\dim X_{\Bla}].
\end{equation*}
\item
$($Springer correspondence for $\CW\nat_{\Bm}$$)$ 
\begin{equation*}
(\pi_1^{(\Bm)})_!\Ql[d_{\Bm}'] \simeq \bigoplus_{\Bla \in \wt\CP(\Bm)}
             \r_{\Bla} \otimes \IC(\ol X_{\Bla}, \Ql)[\dim X_{\Bla}]. 
\end{equation*}
\item
$($Springer correspondence for $W_{n,r}$$)$
\begin{equation*}
(\ol \pi_{\Bm,1})_!\Ql[d'_{\Bm}] \simeq \bigoplus_{\Bla \in \wt\CP(\Bm)}
  V(\Bla)\otimes \IC(\ol X_{\Bla}, \Ql)[\dim X_{\Bla}].  
\end{equation*}
\end{enumerate}
\end{thm}

\begin{proof}
By Proposition 8.3, we know that $\z_!A_{\r} = K_{\r}$.  Hence by the base change 
theorem, $(\z_1)_!(A_{\r}|_{\CG_{\Bm, \unip}}) \simeq K_{\r}|_{\CX_{\Bm, \unip}}$.
For $\Bla \in \CP(\Bm(k))$, we define a simple perverse 
sheaf $B_{\Bla}$ on $\CG_{\Bm, \unip}$ inductively as follows;  in the case where $r = 2$, 
put $B_{\Bla} = \IC(\ol\CO_{\Bla}, \Ql)[\dim \CO_{\Bla}]$, where $\CO_{\Bla}$ is 
the $H$-orbit in $G^{\io\th}\uni \times V$ corresponding to $\Bla$.  In general for
$\Bla \in \CP(\Bm(k))$, put $\Bla = (\la^{(1)}, \Bla')$ with $|\la^{(1)}| = m_1$, 
$\Bla' \in \CP(\Bm'(k))$. We assume that a simple perverse sheaf 
$B_{\Bla'}$ on $\CG'_{\Bm',\unip}$ is already defined.  By a similar construction as
$\wt A_{\r}$ in the proof of Proposition 8.3, there exists a unique simple perverse sheaf
$B_{\Bla}$ on $\CG_{\Bm, \unip}$ satisfying the relation
\begin{equation*}
q_1^*B_{\Bla}[\b_2] \simeq \s_1^*(K_{\la^{(1)}}\boxtimes B_{\Bla'})[\b_1],
\end{equation*}
where $K_{\la^{(1)}} = \IC(\ol\CO_{\la^{(1)}}, \Ql)[\dim \CO_{\la^{(1)}}]$ 
for the $GL_{m_1}$-orbit $\CO_{\la^{(1)}}$ in $(GL_{m_1})\uni$ corresponding to
$\la^{(1)}$, and $\b_1, \b_2$ are as in (8.3.3). 
Assume that $\r = \r_{\Bla}$.  By comparing the construction of $\wt A_{\r}$ and of 
$B_{\Bla}$, and by using induction hypothesis, we 
see that the restriction of $A_{\r}$ on $\CG_{\Bm, \unip}$ coincides with $B_{\Bla}$, 
up to shift. Moreover, by induction, one can show that the restriction of 
$B_{\Bla}$ on $\CG_{\Bla}$ is a constant sheaf $\Ql$.  In particular, 
$\supp B_{\Bla} = \ol\CG_{\Bla}$. Then by Lemma 8.5, the support of $(\z_1)_!B_{\Bla}$ 
coincides with $\ol X_{\Bla}$. By (7.12.2), we know that the restriction of $K_{\r}$  on 
$\CX_{\Bm,\unip}$ 
is a simple perverse sheaf $L_{\r}$, up to shift.  We show that 
\begin{equation*}
\tag{8.7.1}
L_{\r} \simeq \IC(\ol X_{\Bla}, \Ql)[\dim X_{\Bla}].
\end{equation*}
  For this it is enough to see that 
$L_{\r}|_{X_{\Bla}}$ is a constant sheaf $\Ql$.  But by Lemma 8.5 (ii), 
$\z_1\iv(X_{\Bla}) = \wt \CF_{\Bla}^0 \subset \CG_{\Bla}$, and the restriction of 
$B_{\Bla}$ on $\CG_{\Bla}$ is the constant sheaf $\Ql$, up to shift.  Since 
$\z_1\iv(X_{\Bla}) \isom X_{\Bla}$ by Lemma 8.5 (ii), we see that 
$(\z_1)_!B_{\Bla}|_{X_{\Bla}}$ coincides with $\Ql$ up to shift.  
Thus (8.7.1) holds, and (i) follows.  (ii) and (iii) then follows from 
Theorem 7.12 and Corollary 7.14.  The theorem is proved.  
\end{proof}

\para{8.8.}
We now assume that $\CX$ is of enhanced type with $\Bm \in \CQ_{n,r}$. 
We define a variety $\CG_{\Bm}$ by

\begin{equation*}
\begin{split}
\CG_{\Bm} = \{ (x, \Bv, &(W_i)_{1 \le i \le r-1}) \mid (x, \Bv) \in \CX, 
                  (W_i) \text{ : partial flag in }V, \\
                &\dim W_i = p_i, x(W_i) = W_i, v_i \in W_i \ (1 \le i \le r-1) \}. 
\end{split}
\end{equation*}
Let $\z: \CG_{\Bm} \to \CX$ be the projection to the first and second factors.  
Then the map $\pi^{(\Bm)} : \wt\CX_{\Bm} \to \CX_{\Bm}$ is decomposed as 
$\pi^{(\Bm)} = \z\circ \vf$, where $\vf: \wt\CX_{\Bm} \to \CG_{\Bm}, 
\z: \CG_{\Bm} \to \CX_{\Bm}$ are defined similarly to the exotic case. 
Here $\vf$ and $\z$ are surjective maps.  Since $\dim \wt\CX_{\Bm} = \dim \CX_{\Bm}$
by Lemma 4.2, we have $\dim \CG_{\Bm} = \dim \CX_{\Bm}$. 
Put $V_0 = M_{p_1}, \ol V_0 = V/M_{p_1}$.  We define a variety $\CH_{\Bm}$ by 

\begin{equation*}
\begin{split}
\CH_{\Bm} = \{ (x, \Bv, &(W_i), \f_1, \f_2) \mid (x, \Bv, (W_i)) \in \CG_{\Bm}, \\
                 &\f_1: W_1 \isom V_0, \f_2: V/W_1 \isom  \ol V_0 \}. 
\end{split}
\end{equation*} 

We also define a variety $\wt\CZ_{\Bm}$ by 

\begin{equation*}
\begin{split}
\wt\CZ_{\Bm} = \{ (x, \Bv, &gB^{\th}, \f_1, \f_2) \mid (x, \Bv, gB^{\th}) \in \wt\CX_{\Bm}, \\
              &\f_1: g(M_{p_1}) \isom V_0, \f_2: V/g(M_{p_1}) \isom \ol V_0 \}.
\end{split}
\end{equation*}
Assume that $r \ge 3$, and let $\Bm' = (m_2, \dots, m_r)$ for $\Bm = (m_1, \dots, m_r)$.
Let $G_1 = GL(V_0)$ and $G_2 = GL(\ol V_0) \times GL(\ol V_0)$ with 
the permutation action $\th: G _2 \to G_2$.  Then $\wt G_1, \wt\CX'_{\Bm'}, \CX'_{\Bm'}$ etc. 
are defined
similarly to the exotic case.  Note that $\CX'_{\Bm'}$ is a closed subvareity of 
$G_2^{\io\th} \times \ol V_0^{r-2} \simeq GL(\ol V_0) \times \ol V_0^{r-2}$.   
Then a similar commutative diagram as (8.1.1) holds also for the enhanced case. 
We use the same notaion as in (8.1.1).  One can check that the maps $q, \s$ satisfy 
similar properties as in (8.1.2), namely we have
\par\medskip\noindent
(8.8.1) \ $q$ is a principal bundle with fibre isomorphic to $H_0$, and 
$\s$ is a locally trivial fibration with smooth fibre of dimension 
$\dim H + (r-1)m_1$.
(It should be noticed about the difference on the dimension of the fibre.  Also note in this case, 
$H_0 \simeq GL(V_0) \times GL(\ol V_0)$, and $H \simeq GL(V)$.)  

\para{8.9.}
For $\Bm \in \CQ_{n,r}$, put $\wt\CY_{\Bm}^{\dag} = (\psi^{(\Bm)})\iv(\CY_{\Bm}^0)$. 
Then under the notation in 4.1, $\wt\CY_{\Bm}^{\dag}$ coinicdes with 
$\wt\CY_{\BI}$ for $\BI = \BI^{\bullet}$ in 3.3.   In particular, 
$\psi_{\BI}: \wt\CY_{\BI} \to \CY_{\Bm}^0$ is a finite Galois covering with group 
$\CW_{\BI} = \CW_n\nat$.  $\wt\CY^{\dag}_{\Bm}$ is an open dense subsete of $\wt\CY_{\Bm}$.
Put $\CG_{\Bm, \rg} = \z\iv(\CY_{\Bm}^0)$.  $\CG_{\Bm, \rg}$ is an open subset of $\CG_{\Bm}$.
The varieties $\CY'^0_{\Bm'}, \wt\CY'^{\dag}_{\Bm'}$ and $\CG'_{\Bm', \rg}$ are defined similarly.
As in the exotic case, we have the following commutative diagram.

\begin{equation*}
\tag{8.9.1}
\begin{CD}
\wt G_{1,\rg} \times \wt\CY'^{\dag}_{\Bm'} @<<<  \wt\CZ_{\Bm}^{\dag}
               @>>>  \wt\CY_{\Bm}^{\dag} \\
     @VVV   @VVV  @VV\vf_0 V   \\
G_{1,\rg} \times \CG'_{\Bm', \rg} @<<<  \CH_{\Bm, \rg} @ >>>  \CG_{\Bm, \rg} \\
        @VVV    @.    @VV\z_0 V   \\
 G_{1,\rg} \times \CY'^0_{\Bm'}    @.    @.     \CY_{\Bm}^0,
\end{CD}
\end{equation*}

By using (8.9.1) and by induction on $r$, one can show that 
$\vf_0: \wt\CY^{\dag}_{\Bm} \to \CG_{\Bm, \rg}$ is a finite Galois covering with group
$\CW\nat_{\Bm}$.  It follows that 

\par\medskip\noindent
(8.9.2) \  $\z_0$ gives an isomorphism $\CG_{\Bm, \rg} \isom \CY_{\Bm}^0$.  
\par\medskip
For each $\r \in \CW\nat_{\Bm}$, we consider $K_{\r} = \IC(\CX_{\Bm}, \CL_{\r})[d_{\Bm}]$, 
where $\CL_{\r}$ is a simple local system on $\CY_{\Bm}^0$ obtained from the Galois covering 
$\CW\nat_{\Bm}$.  
We regard $\CL_{\r}$ as a local system on $\CG_{\Bm, \rg}$ under the isomorphism 
$\CG_{\Bm,\rg} \simeq \CY^0_{\Bm}$.  
Put $A_{\r} = \IC(\CG_{\Bm,\rg}, \CL_{\r})[d_{\Bm}]$. 
\par
The following result can be proved in a similar way as Proposition 8.3, by making use 
of Theorem 4.5, instead of Theorem 3.2.

\begin{prop}  %%%%%  Prop. 8.10.
Assume that $\CX_{\Bm}$ is of enhanced type. 
\begin{enumerate}
\item
$\vf_!\Ql[d_{\Bm}]$ is a semisimple perverse sheaf on $\CG_{\Bm}$ equipped with 
$\CW\nat_{\Bm}$-action, and is decomposed as

\begin{equation*}
\vf_!\Ql[d_{\Bm}] \simeq \bigoplus_{\r \in (\CW\nat_{\Bm})\wg}\r \otimes A_{\r}.
\end{equation*}

\item
$\z_!A_{\r} \simeq K_{\r}$.
\end{enumerate}
\end{prop}

\para{8.11.}
For each $\Bm \in \CQ_{n,r}$, put $\CG_{\Bm, \unip} = \z\iv(\CX_{\Bm, \unip})$. Then the 
map $\pi^{(\Bm)}_1$ is decomposed as $\pi_1^{(\Bm)} = \z_1\circ \vf_1$ as in the 
exotic case (see 8.4), where $\vf_1: \wt\CX_{\Bm, \unip} \to \CG_{\Bm, \unip}$ 
and $\z_1: \CG_{\Bm, \unip} \to \CX_{\Bm, \unip}$ are restrictions of $\vf, \z$.  
Note that $\vf_1$ is surjective. Put $\CH_{\Bm \unip} = q\iv(\CG_{\Bm, \unip})$.
Then we have a similar commutative diagram as (8.4.1). 
\par
For $\Bla \in \CP(\Bm)$, we define a subset $\CG_{\Bla}$ of $\CG_{\Bm, \unip}$ 
inductively, by applying the discussion in 8.4 for the exotic case.  Note that 
in the case where $r = 2$, $\Bm = (m_1, m_2)$ and $\Bm' = (m_2)$.  
$\CG_{\Bm'} = \CX'\uni$.  In this case, we take $\CG'_{\Bla'}$ as the $G_2^{\th}$-orbit
in $\CX'\uni$ corresponding to $\Bla' = \la^{(2)}$. 
Thus $\CG_{\Bla}$ is defined, and $\CG_{\Bla}$ is a smooth irreducible $H$-stable subvariety 
of $\CG_{\Bm, \unip}$. 
\par
As in the exotic case (see (8.4.2)), 
one can compute the dimension of $\CG_{\Bla}$ by making use of 
(8.8.1).  We have 

\begin{equation*}
\tag{8.11.1}
\dim \CG_{\Bla} = n^2 - n - 2n(\Bla) + \sum_{i=1}^{r-1}(r-i)|\la^{(i)}|.
\end{equation*}     

\par
Recall the definition of $\wt\CF_{\Bla}$ and $\wt\CF^0_{\Bla}$ (in the enhanced case) 
in 6.18.  The following lemma can be proved in a similar way as Lemma 8.5. 
Note that $\dim \wt\CF_{\Bla} = \dim \ol\CG_{\Bla}$ by Proposition 5.4 and 
Proposition 6.19, together with (8.11.1).  We use Proposition 6.19 instead of Proposition 6.7.

\begin{lem}  %%%  Lemma 8.12
Assume that $\CG_{\Bla}$ is of enhanced type.
A similar statement as in Lemma 8.5 holds also for $\CG_{\Bla}$. 
\end{lem}

We can now obtain a refinement of Theorem 7.12 in the enhanced case.  The proof is 
similar to the proof of Theorem 8.7.
Note that Thorem 8.13 (ii) is obtained by Li [Li, Th. 3.2.6] by a different method.

\begin{thm}  %%%%  Theorem 8.13.
Assume that $\CX_{\Bm, \unip}$ is of enhanced type.
\begin{enumerate}
\item
Let $L_{\r}$ be as in Theorem 7.12.  Assume that 
$\r = \r_{\Bla} \in (\CW\nat_{\Bm})\wg$ for $\Bla \in \CP(\Bm)$.
Then we have
\begin{equation*}
L_{\r} \simeq \IC(\ol X_{\Bla}, \Ql)[\dim X_{\Bla}].
\end{equation*}

\item
(Springer correspondence for $\CW\nat_{\Bm}$)
\begin{equation*}
(\pi_1^{(\Bm)})_!\Ql[d'_{\Bm}] \simeq \bigoplus_{\Bla \in \CP(\Bm)}
              \r_{\Bla} \otimes \IC(\ol X_{\Bla}, \Ql)[\dim X_{\Bla}]. 
\end{equation*}
\end{enumerate}

\end{thm}

\para{8.14.}
Assume that $\CX_{\Bm}$ is of exotic type or of enhanced type, under 
the setting in 8.1.  For each $z \in \CX_{\Bm}$, we consider 
the (small) Springer fibre $\CB^{(\Bm)}_z = (\pi^{(\Bm)})\iv(z)$.  In the exotic case, 
we also consider the Springer fibre $\CB_z = \pi\iv(z)$ (see 6.20).  We have 
$\CB_z^{(\Bm)} \subset \CB_z$.  The cohomology group $H^i(\CB_z^{(\Bm)}, \Ql)$ 
has a structure of $\CW\nat_{\Bm}$-module.  In turn, $H^i(\CB_z, \Ql)$ has
a structure of $W_{n,r}$-module.    
Put
$d_{\Bla} = (\dim \CX_{\Bm,\unip} - \dim X_{\Bla})/2$ for 
$\Bla \in \wt\CP(\Bm)$ in the exotic case, and for $\Bla \in \CP(\Bm)$ in the enhanced case.
Explicitly, we have
\begin{equation*}
\tag{8.14.1}
d_{\Bla} = \begin{cases}
              (m_{r-1} -k) + n(\Bla) + n(\la^{(r-1)} + \la^{(r)}) 
                   &\quad\text{: exotic case, $\Bla \in \CP(\Bm(k))$}, \\
              n(\Bla)
                    &\quad\text{: enhanced case, $\Bla \in \CP(\Bm)$}. 
            \end{cases}
\end{equation*}

As shown in the example in 5.18, $\dim \CB_z^{(\Bm)}$ is not constant for $z \in X_{\Bla}$. 
We show the following lemma. 

\begin{lem}  %%% Lemma 8.15
Assume that $\CX_{\Bm}$ is of exotic type or of enhanced type. 
\begin{enumerate}
\item
For any $z \in X_{\Bla}$, $\dim \CB^{(\Bm)}_z \ge d_{\Bla}$. 
The set $z \in X_{\Bla}$ such that $\dim \CB^{(\Bm)}_z = d_{\Bla}$
forms an open dense subset of $X_{\Bla}$.   
\item
For any $z \in X_{\Bla}$, $H^{2d_{\Bla}}(\CB_z^{(\Bm)}, \Ql)$
contains an irreducible $\CW\nat_{\Bm}$-module $\r_{\Bla}$.
\end{enumerate}
\end{lem}

\begin{proof}
First we show (ii).  Assume that $\CX_{\Bm}$ is of exotic type. 
For any $z \in \CX_{\Bm,\unip}$, Theorem 8.7 (ii) implies that 
\begin{equation*}
\tag{8.15.1}
H^i(\CB^{(\Bm)}_z, \Ql) \simeq \bigoplus_{\Bmu \in \wt\CP(\Bm)}
          \r_{\Bmu}\otimes \CH_z^{i-d'_{\Bm} + \dim X_{\Bmu}}(\IC(\ol X_{\Bmu}, \Ql))
\end{equation*}
as $\CW\nat_{\Bm}$-modules. 
Assume that $z \in X_{\Bla}$, and put $i = 2d_{\Bla}$.  
Since $\CH^0_z(\IC(\ol X_{\Bla}, \Ql)) = \Ql$, 
$H^{2d_{\Bla}}(\CB_z^{(\Bm)}, \Ql)$ contains $\r_{\Bla}$.  
This proves (ii).  The enhanced case is proved in a similar way 
by using Theorem  8.13 (ii).  
\par
(ii) implies, in particular, $\dim \CB_z^{(\Bm)} \ge d_{\Bla}$. 
Put $d = \dim (\pi^{(\Bm)})\iv(X_{\Bla}) - \dim X_{\Bla}$. 
Let $X(d)$ be as in 6.20. 
Then $X(d) \cap X_{\Bla}$ is open dense in $X_{\Bla}$. 
Hence $\dim X_{\Bla} \le \dim X(d)$. 
By Lemma 6.21 (iii), we have

\begin{equation*}
\dim\CB_z^{(\Bm)} \le \frac{1}{2}(\dim \CX_{\Bm, \unip} - \dim X(d))
                  \le \frac{1}{2}(\dim \CX_{\Bm, \unip} - \dim X_{\Bla}) = d_{\Bla}
\end{equation*}
for any $z \in X_{\Bla} \cap X(d)$. 
Hence $\dim \CB_z^{(\Bm)} = d_{\Bla}$ and $d = d_{\Bla}$.   This proves (i).   
\end{proof}

We show the following result.  
In the enhanced case, a similar result was proved in [Li, Cor. 3.2.9] 
for the Borel-Moore homology. 

\begin{prop} %%%%  Prop. 8.16.
Assume that $\CX_{\Bm}$ is of exotic type or of enhanced type.
Take $z \in X_{\Bla}$ such that $\dim \CB^{(\Bm)}_z = d_{\Bla}$. 
\begin{enumerate}
\item
$H^{2d_{\Bla}}(\CB_z^{(\Bm)},\Ql) \simeq \r_{\Bla}$ as $\CW\nat_{\Bm}$-modules.
\item
Assume that $\CX_{\Bm}$ is of exotic type. Then  $\dim \CB_z = d_{\Bla}$, and  
$H^{2d_{\Bla}}(\CB_z, \Ql) \simeq V(\Bla)$ as $W_{n,r}$-modules.
\end{enumerate} 
\end{prop}

\begin{proof}
We prove (i) by induction on $r$. 
Assume that $r = 2$. In the exotic case, (i) holds by  [SS1, Remarks 5.5 (ii)].
A similar method also works for the enhanced case, since the number of $H$-orbits
in $\CX\uni$ is finite. Assume that $r \ge 3$, and that (i) holds for $r-1$. 
We consider the diagram as in (8.1.1) restricted to the unipotent varieties
as discussed in 8.4. 
Put $\Bla = (\la^{(1)}, \Bla')$ as in 8.4. 
By (8.14.1), we have $d_{\Bla} = d_{\Bla'} + d_{\la^{(1)}}$, where 
$d_{\Bla'}$ is defined similarly to $d_{\Bla}$, and 
$d_{\la^{(1)}} = (\dim G_{1, \unip} - \dim \CO_{\la^{(1)}})/2$.  
Take $z \in X_{\Bla}$ such that $\dim \CB_z^{(\Bm)} = d_{\Bla}$.  By Lemma 8.5
and Lemma 8.12, $\z_1$ gives an isomorphism $\z_1\iv(X_{\Bla}) \to X_{\Bla}$. 
Hence there exists a unique $z_* \in \z_1\iv(X_{\Bla})$ such that 
$\z_1(z_*) = z$.  Then by using the diagram (8.1.1), 
one can find $(x_1, z') \in \CO_{\la^{(1)}} \times X'_{\Bla'}$
and $z'_* = (\z_1')\iv(z')$ such that
$\s_1\iv(x_1, z'_*) = q_1\iv(z_*)$. Here $z' \in X'_{\Bla'}$ satisfies the condition 
that $\dim \CB_{z'}^{(\Bm')} = d_{\Bla'}$.
By using the isomorphism $(\z'_1)\iv(X'_{\Bla'}) \isom X'_{\Bla'}$, we have

\begin{equation*}
H^{2d_{\Bla'}}(\CB^{(\Bm')}_{z'}, \Ql) \simeq (R^{2d_{\Bla'}}(\pi'_1)_!\Ql)_{z'}
    \simeq (R^{2d_{\Bla'}}(\vf'_1)_!\Ql)_{z'_*},
\end{equation*}
where $\vf_1'$ is the restriction of $\vf'$.
Similarly we have 
$H^{2d_{\Bla}}(\CB_z^{(\Bm)}, \Ql) \simeq (R^{2d_{\Bla}}(\vf_1)_!\Ql)_{z_*}$ 
by using $\z\iv(X_{\Bla}) \isom X_{\Bla}$.
Let $\xi$ be an element contained in $\s_1\iv(x_1, z'_*) = q_1\iv(z_*)$.
By (8.1.1), we have
\begin{equation*}
(R^{2d_{\la^{(1)}}}(\pi^1_1)_!\Ql)_{x_1} \otimes (R^{2d_{\Bla'}}(\vf_1')_!\Ql)_{z'_*}
   \simeq (R^{2d_{\Bla}}(\wt\vf_1)_!\Ql)_{\xi} 
   \simeq (R^{2d_{\Bla}}(\vf_1)_!\Ql)_{z_*},
\end{equation*}
where $\pi^1_1, \wt\vf_1$ are restrictions of $\pi^1, \wt\vf$, respectively.  
By induction, we know that $\dim H^{2d_{\Bla'}}(\CB_{z'}^{(\Bm')},\Ql) = \dim \r_{\Bla'}$.
It is well-konwn that 
$\dim H^{2d_{\la^{(1)}}}(\CB_{x_1}, \Ql) = \dim \r_{\la^{(1)}}$. 
Hence $\dim H^{2d_{\Bla}}(\CB_z^{(\Bm)}, \Ql) 
      = \dim \r_{\la^{(1)}} + \dim \r_{\Bla'} = \dim \r_{\Bla}$. 
On the other hand, by Lemma 8.15, 
$H^{2d_{\Bla}}(\CB_z^{(\Bm)}, \Ql)$ contains $\r_{\Bla}$. 
Thus $H^{2d_{\Bla}}(\CB_z^{(\Bm)},\Ql) \simeq \r_{\Bla}$.  
(i) is proved.
\par
Next we show (ii).  Assume that $\CX_{\Bm}$ is of exotic type and 
$z \in X_{\Bla}$ such that $\dim \CB_z^{(\Bm)} = d_{\Bla}$. 
We consider the decomposition (8.15.1) in the case where $z \in X_{\Bla}$.
By Theorem 8.7 (iii), we have a similar decomposition
\begin{equation*}
\tag{8.16.1}
H^i(\CB_z, \Ql) \simeq  \bigoplus_{\Bmu \in \wt\CP(\Bm)}
              V(\Bmu)\otimes \CH_z^{i-d'_{\Bm} + \dim X_{\Bmu}}(\IC(\ol X_{\Bmu},\Ql)).  
\end{equation*}
(8.15.1) shows that $\CH_z^{i - d'_{\Bm} + \dim X_{\Bmu}}(\IC(\ol X_{\Bmu}, \Ql)) = 0$ 
for any choice of $i > 2d_{\Bla}$ and of $\Bmu \in \wt\CP(\Bm)$. 
This implies, by (8.16.1),  that $H^i(\CB_z, \Ql) = 0$ for any $i > 2d_{\Bla}$. 
Since $\CB_z^{(\Bm)} \subset \CB_z$, we conclude that 
$\dim \CB_z = \dim \CB_z^{(\Bm)} = d_{\Bla}$.
Now assume that $i = 2d_{\Bla}$.  By (i) and (8.15.1), we see that 
$\CH_z^{2d_{\Bla} -d'_{\Bm} + \dim X_{\Bmu}}(\IC(\ol X_{\Bmu}, \Ql)) = 0$ for 
any $\Bmu \ne \Bla$.   Hence by (8.16.1), $H^{2d_{\Bla}}(\CB_z, \Ql) \simeq V(\Bla)$. 
The proposition is proved.
\end{proof}

\par

\par\vspace{1cm}
\noindent
T. Shoji \\
Department of Mathematics, Tongji University \\ 
2139 Siping Road, Shanghai 200092, P.R. China  \\
E-mail: \verb|shoji@tongji.edu.cn|
 \end{document}